\documentclass[11pt]{amsart}

\usepackage{amsmath,amssymb,amsthm}
\usepackage{mathrsfs}
\usepackage{graphicx}
\usepackage{pifont}
\numberwithin{equation}{section}

\theoremstyle{plain}
\newtheorem{theorem}{Theorem}[section]
\newtheorem{proposition}[theorem]{Proposition}
\newtheorem{lemma}[theorem]{Lemma}
\newtheorem{corollary}[theorem]{Corollary}
\newtheorem{definition}[theorem]{Definition}
\newtheorem{example}[theorem]{Example}
\newtheorem{remark}[theorem]{Remark}

\newtheorem*{theoremA}{Theorem A}
\newtheorem*{theoremB}{Theorem B}
\newtheorem*{theoremC}{Theorem C}

\DeclareMathOperator{\Ext}{Ext}
\DeclareMathOperator{\Hom}{Hom}
\DeclareMathOperator{\Def}{Def}
\DeclareMathOperator{\End}{End}
\DeclareMathOperator{\Ker}{Ker}
\DeclareMathOperator{\im}{im}
\DeclareMathOperator{\Pf}{Pf}
\DeclareMathOperator{\Tr}{Tr}
\DeclareMathOperator{\Res}{Res}
\DeclareMathOperator{\Spec}{Spec}
\DeclareMathOperator{\GL}{GL}
\DeclareMathOperator{\Aut}{Aut}

\DeclareMathOperator{\rk}{rk}
\newcommand{\cO}{\mathcal{O}}
\newcommand{\cE}{\mathcal{E}}

\newcommand{\cI}{\mathcal{I}}

\newcommand{\cV}{\mathcal{V}}
\newcommand{\cL}{\mathcal{L}}
\newcommand{\cP}{\mathcal{P}}
\newcommand{\ka}{\kappa}

\newcommand{\ot}{\otimes}
\newcommand{\LF}{\textup{\textsc{lf}}}

\usepackage[colorlinks=true,linkcolor=blue,citecolor=blue,urlcolor=blue,pdftitle={Leafrians, Frobenius splittings, and moduli spaces over Poisson surfaces},pdfauthor={Jun Peng}]{hyperref}

\begin{document}

\title[Leafrians, Frobenius splittings, and moduli spaces over Poisson surfaces]{Leafrians, Frobenius splittings, and moduli spaces over Poisson surfaces}
\author{Jun Peng}

\begin{abstract}
Let $X$ be a smooth projective Poisson surface over an algebraically closed field of characteristic $p>2$, and let $s\in H^0(X,\omega_X^{-1})$ be nonzero. We construct an anticanonical section associated to $s$, called the \emph{leafrian}, on moduli spaces $M$ of stable objects on $X$, including the moduli of stable sheaves, of principal $G$-bundles, and of instantons. The leafrian splits $M$ if and only if $s^{p-1}$ Frobenius splits the base. It is invariant under the Hamiltonian flows of Bottacin's Poisson structure on $M$, so every compatibly split subvariety of $M$ is a Poisson subvariety. Furthermore, the leafrian splits the Gieseker compactification $\bar M$ under the same criterion. As a consequence, in characteristic zero, for $X$ a smooth projective Poisson surface with a smooth or empty anticanonical divisor and $\rho>0$ (positive leaf rank), the Gieseker compactification $\bar M$ satisfies Kodaira-type cohomology vanishing, including singular $\bar M$.
\end{abstract}
\maketitle

\makeatletter\renewcommand{\tocsection}[3]{\indentlabel{\@ifnotempty{#2}{\hb@xt@2em{#2.\hfil}\quad}}#3}\makeatother\setcounter{tocdepth}{1}
\tableofcontents

\section{Introduction}\label{sec:intro}

Throughout, $k$ is an algebraically closed field of characteristic $p>2$.  A smooth variety $Y$ over $k$ is Frobenius split when $\cO_Y\to F_*\cO_Y$ admits an $\cO_Y$-linear retraction. Frobenius split varieties satisfy vanishing theorems and enjoy good intersection theory \cite{BKu}.  A single anticanonical section $\sigma\in H^0(Y,\omega_Y^{-1})$ can produce a splitting by the criterion of Brion and Kumar \cite[Thm.~1.3.8]{BKu}: in slice coordinates at a closed point, writing $\sigma=f\,\partial_{[d]}$, the \emph{top coefficient}
\[
 b_0=[\,t_1^{p-1}\cdots t_d^{p-1}\,]\,f^{p-1}
\]
governs the trace of $\sigma^{p-1}$.  When that trace is a nonzero constant, a rescaling of $\sigma$ splits $Y$.  The nonvanishing of $b_0$ is what we produce.

The Lie-theoretic examples like Bott--Samelson varieties, $G/U$, and $G/B$ carry Poisson $T$-Pfaffians which split \cite{Lu,PY}; the brackets come from $r$-matrices, and the local reason is coordinate-bound.  This paper isolates the local mechanism and applies it to moduli spaces over surfaces.

The mechanism is one local object.  The deformation module of a fat point $Z=\Spec A$ carries a weighted trace pairing
\[
 B(v_1,\dots,v_n)=\kappa_0\,\Res_A\bigl(g\cdot\det[v_1|\cdots|v_n]\bigr),
\]
where $g$ is a weight on the local algebra and $n=\dim X$ is the codimension of the fat point.  The norm $N_{A/k}(g)$ carries the weight into the block.  The top coefficient of $N_{A/k}(g)^{p-1}$ in the block coordinates is the \emph{block coefficient} $T(g)$; it is the only part of the mechanism that sees the weight.

Concretely, let $X$ be a smooth projective variety of dimension $n$ and let $M$ be a moduli space of sheaves on $X$.  The construction rests on \emph{members}: fat points at which the weighted trace pairing above is realized; when the general object is not locally free and its non-locally free points carry the pairing, those points are the members.  In all cases the units are supplied by a master chain (Section~\ref{sec:transfer}).  A member of length $\ell$ splits into $\ell$ \emph{units}.  The number of units at a member-saturated datum is
\[
 \rho\ =\ \sum_\lambda \ell_\lambda .
\]
This integer $\rho$ is intrinsic to $M$: it depends on neither $s$ nor the presentation.  On a surface it is realized by Bottacin's tensor at generic data for generic $s$, $\rho=\frac12\operatorname{rk}\theta_s$ (Remark~\ref{rem:rho-open}), and this is the $\rho$ of the statements below.  The leafrian exists exactly when $\rho>0$, equivalently when the leaf is nonempty. The block is the deformation module $A^{\oplus n}$ of a codimension-$n$ complete intersection.

At each unit, a polyvector field on $X$ pushes forward to $M$ along the support map; summing over units produces the \emph{transferred polyvector}
\[
 \hat s=\sum_{a}\beta_a,\qquad
 \hat s^{[\rho]}=\bigwedge_{a}\beta_a,
\]
the sum and the normalized wedge of the unit polyvectors.  Wedging with an Euler frame $\Omega$, the wedge of the transverse scaling fields $c_j\partial_{c_j}$, gives $\hat s^{[\rho]}\wedge\Omega$, defined on the member locus; when it extends to a nonzero anticanonical section of $M$, we call it the leafrian. Its top coefficient factors as a product of block contributions.  Once the trace of the leafrian is constant, the leafrian splits $M$ if and only if every block coefficient is nonzero.

The members and the Euler frame above are supplied by a presentation.  A presentation (Definition~\ref{def:present}) is a chain of elementary and point modifications realizing the object.

From here on we specialize to surfaces.  The directions of $M$ split into leaf directions and transverse directions, and their numbers satisfy $2\rho+c=d$, where $c$ is the number of transverse directions.

We apply this to moduli spaces over a smooth projective Poisson surface $X$: stable sheaves, principal bundles, and instantons over a pair $(X,L)$.  Here Poisson means that $\omega_X^{-1}$ carries a nonzero section.  The classification of such surfaces is classical: they are the rational surfaces, the ruled surfaces, the $K3$ surfaces, and the abelian surfaces.

Over a Poisson surface, the transferred polyvector of a bivector $s$ on $X$ is identified with Bottacin's Poisson tensor $\theta_s$ on moduli spaces, so the leafrian is $\LF_s=\theta_s^{[\rho]}\wedge\Omega$.

For the trace computation, we use a filtered degeneracy complex. Write
\[
 \mathsf t(s)=\operatorname{Tr}_X(s),
\]
the Frobenius trace of the weight $s$ on $X$. When $D=Z(s)$ is a smooth elliptic curve, $\mathsf t(s)$ is the Hasse invariant of $D$.

At a stable smooth point, the Bottacin deformation complex is reduced by the radical coming from restriction to $D$. Its determinant is $\LF_s^2$, and its Cartier determinant is computed by the modification blocks. A weighted norm member of length $\ell$ contributes $\mathsf t(s)^\ell$, an unweighted Koszul block contributes $1$, and a split End block of multiplicity $m$ on $D$ contributes $\mathsf t(s)^{m^2-1}$. Split End blocks on ordinary divisors contribute only fixed units. These computations use the block coefficient of Section~\ref{sec:trans} and the split End formula of Section~\ref{sub:degcomplex}. A single-point filtered representation of the object suffices; it is not necessary to construct moving members for every leaf direction.

On a surface the total Cartier determinant is $\mathsf t(s)^N$ up to a fixed unit. Since $\LF_s^{p-1}$ has degree $\rho(p-1)$ in $s$, the degree of the determinant forces $N=\rho$. This gives the trace formula
\[
 \Tr(\LF_s^{p-1})=\Tr_X(s)^{\rho}\cdot u,\qquad u\in k^\times .
\]

The leafrian $\LF_s$ is independent of the Euler frame up to a regular unit, by Remark~\ref{rem:frame-indep}.  Over surfaces the presentation theory of Section~\ref{sec:transfer} supplies the Euler frame and the member locus.  When $\rho>0$ the trace formula gives the criterion $\Tr_X(s)\neq0$, the Hasse invariant of $X$ or of $X\setminus L$.

\medskip\noindent\textbf{Standing hypothesis.}  For principal bundles, $G$ is smooth reductive, $\mathfrak g$ carries an invariant nondegenerate bilinear form, and $z(\mathfrak g)\otimes H^1(\cO_X)=0$.

\begin{theoremA}
Let $X$ be a smooth projective Poisson surface, let $s$ be nonzero in $H^0(X,\omega_X^{-1})$, or in $H^0(X,\omega_X^{-1}(-2L))$ in the instanton case. Let $M$ be $M(r,c_1,c_2)$, $M_G$, or $M_{r,k}(X,L)$. Then:
\begin{enumerate}
\item[\rm(i)] \emph{(existence)} the leafrian $\LF_s$ exists if and only if $\rho>0$;
\item[\rm(ii)] \emph{(splitting)} $\LF_s$ splits $M$ if and only if $s^{p-1}$ Frobenius splits $X$, or $X\setminus L$ in the instanton case;
\item[\rm(iii)] \emph{(Torelli)} if $s$ splits the base as in (ii), the assignment $[\,s\,]\longmapsto[\,\LF_s\,]$ is injective;
\item[\rm(iv)] \emph{(Poisson)} $\LF_s$ is invariant under the Hamiltonian flows of Bottacin's Poisson structure on $M$ and hence every compatibly split subvariety is a Poisson subvariety whenever $\LF_s$ splits $M$.
\end{enumerate}
\end{theoremA}

The leafrian lives on $M$, which is quasi-projective.  The next theorem extends it to the Gieseker compactification $\bar M$ and records the splitting there.

\begin{theoremB}
Let $X$, $s$, and $M$ be as in Theorem~A, and let $\bar M$ denote the Gieseker compactification of $M$.  If $\rho>0$ then $\LF_s$ extends to a section $\bar\LF_s$ of $\omega_{\bar M}^{[-1]}$; in the instanton case the extension holds when $\Tr_U(s^{p-1})$ is constant.  Moreover,
\[
 \bar\LF_s\ \text{splits}\ \bar M
 \qquad\Longleftrightarrow\qquad
 s^{p-1}\ \text{Frobenius splits}\ X,\ \text{or}\ X\setminus L\ \text{in the instanton case},
\]
where both sides fail when the extension fails.  If $\bar M$ is Gorenstein then $\bar\LF_s$ is a section of $\omega_{\bar M}^{-1}$.
\end{theoremB}

For instantons $\bar M$ is the framed Quot scheme, which is smooth, and the extension is immediate when $\Tr_U(s^{p-1})$ is constant (Remark~\ref{rem:g-instantons}); the content of Theorem~B lies in the sheaf and principal bundle cases, where $\bar M$ has a singular strictly semistable locus.  Since $\bar M$ is projective and Frobenius split, Kodaira-type cohomology vanishing holds in characteristic $p$ \cite{BKu}, and in characteristic zero one obtains the following statement.

\begin{theoremC}\label{thm:vanishing}
In characteristic zero, let $X$ be \ding{192}\,a $K3$ surface or an abelian surface, or \ding{193}\,a rational surface or a ruled surface admitting $s\in H^0(X,\omega_X^{-1})$ with $Z(s)$ smooth.  If $\rho>0$ the Gieseker compactifications of $M(r,c_1,c_2)$, $M_G$, and $M_{r,k}(X,L)$ satisfy
\[
 H^i(\bar M,\,\omega_{\bar M}\otimes L)=0,\qquad i>0,\ L\ \text{ample}.
\]
\end{theoremC}

For $r=1$ and $c_1=0$ the moduli space $M(r,c_1,c_2)$ is $\operatorname{Hilb}^{c_2}(X)$, and Theorem~A refines \cite{KT} on projective Poisson surfaces: it produces the splitting explicitly and gives the converse.

The restriction $\rho>0$ marks the range of the leafrian method.  It does not claim that the vanishing fails when $\rho=0$.  When the leaf is empty no leafrian is formed; the assembly is the deformation volume $\eta_{\mathrm{can}}$ of Proposition~\ref{prop:empty}, whose existence is a geometric condition on $M$ studied in Section~\ref{sec:principle}.  Empty leaf is an isolated numerical condition.  Once $r$ and $c_1$ are fixed, the equality $d=c_0$ holds only for finitely many $c_2$ (and only for special $(r,k)$ in the instanton case), since $d$ grows with $c_2$ while $c_0$ is bounded at generic data.  Thus the generic numerical type has $\rho>0$, and Theorem~C covers almost all cases.  The presentation theory is verified over surfaces in Sections~\ref{sec:moduli}--\ref{sec:inst}.

\begin{remark}[Reading off $\rho$ from $M$]\label{rem:read-rho}
When the general object is not locally free and its non-locally free points realize the leaf, those points are the members and
\[
 \rho=\operatorname{length}(E^{\vee\vee}/E)=\operatorname{ch}_2(E^{\vee\vee})-\operatorname{ch}_2(E),
\]
as in Lemma~\ref{lem:hd-members}.  The equality holds at member-saturated objects; at a stable object one has $\operatorname{length}(E^{\vee\vee}/E)=\sum_\nu\ell_\nu\le\rho$, with strict inequality in general (Section~\ref{sub:master}).  This includes Hilbert schemes and framed sheaves with instanton points.  For Hilbert schemes $\rho$ counts the points of $Z$, so $\rho=0$ exactly when the general object is locally free.  In general $\rho=0$ is the empty leaf of the preceding paragraph, and a locally free general object does not force $\rho=0$, since the units may come from special objects or from a master chain.  One has $2\rho=\operatorname{rk}\theta_s$ at generic data.
\end{remark}

\medskip\noindent\emph{Roadmap.}  Section~\ref{sec:prelim} records the splitting criterion and the standing conventions.  Section~\ref{sec:trans} computes the block coefficient.  Section~\ref{sec:hat} constructs the transferred polyvector and identifies it with Bottacin's tensor.  Section~\ref{sec:principle} defines the leafrian.  Section~\ref{sec:transfer} builds presentations, and Section~\ref{sub:degcomplex} proves the trace formula by filtered degeneracy complexes.  Section~\ref{sec:moduli} is the model case, stable sheaves; Sections~\ref{sec:others} and~\ref{sec:inst} treat principal bundles and instantons.  Section~\ref{sec:cfp} is the Poisson enrichment and the Torelli theorem.  Section~\ref{sec:gieseker} extends the leafrian to the Gieseker compactification and yields Theorem~C.  Appendix~\ref{app:calc} collects the elementary transformation calculus.

\section{Preliminaries}\label{sec:prelim}

\subsection{Frobenius splitting}

\begin{definition}\label{def:split}
A smooth variety $Y$ over $k$ is \emph{Frobenius split} if $\cO_Y\hookrightarrow F_*\cO_Y$ admits an $\cO_Y$-linear retraction.  A splitting is given by $\varphi\in H^0(Y,\omega_Y^{1-p})$ with $\Tr(\varphi)=1$, where $\Tr$ is the trace of Frobenius.  A subscheme $Z\subset Y$ is \emph{compatibly split} if the splitting restricts to $Z$.

For $\sigma\in H^0(Y,\omega_Y^{-1})$ we write $\varphi_\sigma:=\sigma^{p-1}$.  We say that $\sigma$ \emph{splits} $Y$ when $\Tr(\varphi_\sigma)$ is a nonzero constant; a rescaling of $\sigma$ then normalises the trace to $1$.  If $s\in H^0(X,\omega_X^{-1})$, we write
\[
 \operatorname{Tr}_X(s)\,:=\,\operatorname{Tr}(s^{p-1})\in H^0(X,\cO_X)
\]
for the Frobenius trace of $s^{p-1}$.  When $X$ is complete this is a constant.
\end{definition}

\begin{theorem}[\mbox{\cite[Thm.~1.3.8]{BKu}}]\label{thm:bk}
Let $Y$ be a smooth variety and $\varphi\in H^0(Y,\omega_Y^{1-p})$.  Via the canonical identification, $\varphi$ splits $Y$ if and only if at every closed point $x\in Y$, in local coordinates $t_1,\dots,t_d$, the monomial $(t_1\cdots t_d)^{p-1}$ occurs with coefficient $1$ in the local expansion of $\varphi$ (equivalently, $\Tr(\varphi)=1$ as a function).  When $\Tr(\varphi)$ is a constant function (for instance when $Y$ is complete) it suffices that the coefficient is $1$ at one closed point.  If $\sigma\in H^0(Y,\omega_Y^{-1})$ and $\Tr(\sigma^{p-1})$ is a nonzero constant $c$, a rescaling $\sigma\mapsto\lambda\sigma$ with $\lambda^{p-1}c=1$ yields a splitting; $k$ being algebraically closed, such a $\lambda$ exists.
\end{theorem}

\begin{remark}[Extension of a splitting from a dense open]\label{rem:nearsplit}
The criterion of Theorem~\ref{thm:bk} holds for smooth $Y$ only in its
stated form. For a reduced $Y$ the following consequence of
\cite[Ch.~1]{BKu} is enough. If $\varphi\in\Hom_{\cO_Y}(F_*\cO_Y,\cO_Y)$
satisfies $\varphi\circ F=\mathrm{id}$ on a dense open $U\subset Y$, then
$\varphi$ is a splitting of $Y$. Indeed $\varphi\circ F$ is multiplication
by $a=\varphi(F(1))$, the hypothesis gives $a|_U=1$, and $Y$ reduced with
$U$ dense forces $a=1$. Smoothness of $Y$ is not used in this step.
\end{remark}

\begin{remark}[One-chart principle]\label{rem:onechart}
Rescaling $\sigma$ by $\lambda$ rescales $\varphi_\sigma$ by $\lambda^{p-1}$.  The nonvanishing of the coefficient of $(t_1\cdots t_d)^{p-1}$ is independent of the chart.  Signs and normalizations in $\Pf$ are therefore immaterial.  The map $f\mapsto\lambda f$ scales $b_0$ by $\lambda^{p-1}\neq0$; only nonvanishing is at issue.
\end{remark}

\begin{theorem}[\mbox{\cite[Prop.~1.3.11]{BKu}}]\label{thm:compat}
If $Y$ is a smooth complete variety and $\sigma\in H^0(Y,\omega_Y^{-1})$ is such that $\sigma^{p-1}$ splits $Y$, then the zero scheme $Z(\sigma)$ is compatibly split; in particular $Z(\sigma)$ is reduced.
\end{theorem}

\begin{lemma}[Local compatibility]\label{lem:loccompat}
Let $Y$ be smooth and $\sigma\in H^0(Y,\omega_Y^{-1})$ with $\sigma^{p-1}$ splitting $Y$ after a constant rescaling.  Then $Z(\sigma)$ is compatibly split, and in particular reduced.  Completeness of $Y$ is not used.
\end{lemma}

\begin{proof}
Locally $\sigma=f\,\partial_{[d]}$, so $Z(\sigma)$ is the Cartier divisor $(f)$.  A local function $h$ vanishing on $Z(\sigma)$ is $h=f\cdot g$.  At a point $x\in Z(\sigma)$ one has $f(0)=0$, and in the Frobenius box $\widehat{\cO}/(t_1^p,\dots,t_d^p)$ one has $h\,f^{p-1}=g\,f^p=0$ (there $f^p\equiv f(0)^p=0$).  Thus $\varphi_\sigma=\sigma^{p-1}$ carries $F_*\cI_{Z(\sigma)}$ into $\cI_{Z(\sigma)}$.  A compatibly split scheme is reduced: the composition $\cO\to F_*\cO\to\cO$ is the identity, the first arrow sending $a$ to $a^p$, so $a^p=0$ implies $a=0$.
\end{proof}

\begin{lemma}[Unit rigidity]\label{lem:unitrigid}
Let $Y$ be reduced and connected, and let $u$ be a regular function on $Y$ such that $u^{p-1}$ is a nonzero constant.  Then $u\in k^\times$.
\end{lemma}

\begin{proof}
Write $u^{p-1}=\lambda^{p-1}$ with $\lambda\in k^\times$.  The polynomial $T^{p-1}-1$ splits into distinct linear factors over $k$, so $u/\lambda$ takes values in the finite set $\mu_{p-1}$.  Each fibre of $u/\lambda$ is open and closed, and $Y$ is connected, so $u$ is the constant $\lambda\mu$ for some $\mu\in\mu_{p-1}$.
\end{proof}

\subsection{The top coefficient}

Everything takes place in the formal local ring $\widehat\cO=\widehat{\cO}_{Y,x}$ at one point $x$.  A \emph{slice coordinate system} is a regular system of parameters $t=(t_1,\dots,t_d)$ of $\widehat\cO$.

\begin{definition}[Top coefficient]\label{def:top}
For $f\in\widehat\cO$, the \emph{top coefficient} is
\[
 b_0(f)=[\,t^{(p-1)\mathbf 1}\,]\,f^{p-1},\qquad t^{(p-1)\mathbf 1}:=\prod_{i=1}^{d}t_i^{p-1},
\]
the socle coordinate of $f^{p-1}$ in the Frobenius box $\widehat\cO/(t_1^p,\dots,t_d^p)$.  For a section $\sigma=f\,\partial_{[d]}$ we write $b_0(\sigma):=b_0(f)$.
\end{definition}

\begin{lemma}[Socle identity; coordinate invariance of $b_0$]\label{lem:socle}
Let $A=k[t_1,\dots,t_d]/(t_1^p,\dots,t_d^p)$ and $P\in\GL_d(k)$.  The linear substitution $\Phi_P$ is a $k$-algebra automorphism of $A$ and
\[
 \Phi_P\big(t^{(p-1)\mathbf 1}\big)=\prod_{j=1}^d\ell_j^{p-1}=\det(P)^{p-1}\,t^{(p-1)\mathbf 1},\qquad \ell_j:=\Phi_P(t_j).
\]
Consequently, for a section the top coefficient is independent of the slice coordinate system up to a nonzero factor.  For a bare function the value can change; the invariance is a statement about sections.
\end{lemma}

\begin{proof}
Since $(\Phi_P(t_i))^p=\sum_j a_{ij}^p t_j^p=0$ in $A$, the substitution is well defined.  The map $P\mapsto\Phi_P(t^{(p-1)\mathbf 1})/t^{(p-1)\mathbf 1}$ and the map $P\mapsto\det(P)^{p-1}$ are both algebraic characters of $\GL_d(k)$, hence both are powers of $\det$.  Comparing on diagonal matrices gives the power $p-1$, so the identity holds.  As a check, a permutation fixes $t^{(p-1)\mathbf 1}$ and $(\pm1)^{p-1}=1$; for a transvection $I+E_{ij}$ only the term of $(t_i+t_j)^{p-1}t_j^{p-1}$ with $t_i$-exponent $p-1$ survives, since exponents at least $p$ die in $A$, giving $t^{(p-1)\mathbf 1}$.  A linear coordinate change acts on the socle generator by the unit $\det(J)^{p-1}$: the section $\sigma=f\,\partial_{[d]}$ has local expression $(f\circ\Psi)\det(J)^{-1}\partial^{u}_{[d]}$, and the frame factor $\det(J)^{-(p-1)}$ compensates the change of the socle extraction.  For a bare function the value can change: for $d=1$, $p=3$, $f=1+t$ and $t=u+u^2$ the values are $1$ and $0$.  A general coordinate change need not preserve the Frobenius box exactly, and only the nonvanishing of $b_0$ is used below; that nonvanishing is invariant because it is the nonvanishing of the intrinsic trace $\Tr(\sigma^{p-1})$.
\end{proof}

\begin{lemma}[Invariance]\label{lem:inv}
The nonvanishing of $b_0(\sigma)$ does not depend on the choice of slice coordinates.
\end{lemma}

\begin{proof}
This is the last assertion of Lemma~\ref{lem:socle}.
\end{proof}

\subsection{Parameters}

A member is a fat point of length $\ell$; its deformation module is $A^{\oplus n}$ of dimension $n\ell$ and it carries $n\ell$ leaf coordinates.  A member of length $\ell$ splits into $\ell$ units, and $\rho=\sum_\lambda\ell_\lambda$ counts units at a member-saturated datum.  On a surface $\rho$ is also half the rank of Bottacin's tensor $\theta_s$ at generic data for generic $s\in H^0(X,\omega_X^{-1})$ (Remark~\ref{rem:rho-open}).  The number $\rho$ depends on $M$ alone.  The budget is
\[
 n\rho+c=d.
\]
For $n=2$ the deformation module is $A\oplus A$ of dimension $2\ell$, the budget is $2\rho+c=d$, and this is the case used in the moduli applications.

$\rho$ counts leaf \emph{units}, not blocks: a unit has $n$ leaf coordinates, and for $n=2$ a unit is one leaf pair.  A block of length $\ell$ splits into $\ell$ units by Definition~\ref{def:tpoly}, so the leaf polyvector has degree $n\rho$.  When the members realize the leaf, $n\rho$ is the generic rank of the leaf factor.  For $n=2$ we write $[\rho]$ for the Pfaffian normalization of a wedge power: $\pi^{[\rho]}$ is the section whose local expression in a leaf frame is $\Pf(A)\,e_1\wedge\cdots\wedge e_{2\rho}$, where $A$ is the matrix of $\pi$ in the leaf frame.  The wedge of $\rho$ copies of $\pi$ equals $\rho!$ times $\pi^{[\rho]}$ and may vanish when $\rho\ge p$.

The extraction of $t^{(p-1)\mathbf 1}$ from $f^{p-1}$ reads a coefficient in the Frobenius box.  It is unaffected by the vanishing of the constant term of $f$.  Splitting statements therefore need no hypothesis on the rank of the leaf factor at the evaluation point.

\subsection{Pairings and powers}

Two pairings occur.  The trace pairing of a block is a form on the deformation module $\Def_Z$, which is tangent data.  Bottacin's pairing $\beta_s$ is a pairing on the cotangent space $T^*M$.  Serre duality identifies $\Ext^1(E,E\ot\omega)$ with $\Ext^1(E,E)^\vee$, so the two descriptions agree.  The transpose of Bottacin's bivector across this duality is the block pairing, and the leaf factor of the leafrian is the Pfaffian of that matrix in the deformation basis.  The inverse of the bivector does not occur.

The leaf factor is the Pfaffian of the block-diagonal pairing in the deformation bases.  The local expression of the leafrian is
\[
 f=\prod_\nu N(g_\nu)\cdot\prod_j c_j ,
\]
with each norm factor $N(g_\nu)$ occurring to the first power.  The inverse power has poles along $N(g_\nu)=0$ and does not occur.

\section{The block and the unit top coefficient}\label{sec:trans}

This section computes the block coefficient and the unit top coefficient for codimension two, the case of the moduli applications.  The transferred polyvector of Definition~\ref{def:tpoly} and the leaf factor of Definition~\ref{def:leaffactor} are defined for general $n$; only the block calculus below is specialized to $n=2$, and the global statements below are made for surfaces.

\subsection{Fat points on a branch}

\begin{definition}[Fat point on a branch]\label{def:fat}
A \emph{fat point} is a finite scheme $Z=\Spec A$ with $A$ a local Artin $k$-algebra of length $\ell\geq1$.  It lies on a \emph{branch} if $A$ is monogenic, $A\cong k[u]/(m)=k[z]$ with $z=u$ modulo $m$, and $z$ is called the \emph{branch coordinate}.  Write
\[
 m=u^\ell-q(u),\qquad q(u)=\sum_{i<\ell}q_iu^i,
\]
for the characteristic polynomial of $z$; then $1,z,\dots,z^{\ell-1}$ is a $k$-basis of $A$.  For $\ell=1$ the fat point is a single reduced point.
\end{definition}

\begin{lemma}[Leaf identity]\label{lem:leaf}
$N_{A/k}(z)=(-1)^{\ell+1}q_0$, where $N_{A/k}(a)=\det(\operatorname{mult}_a)$ is the norm.
\end{lemma}

\begin{proof}
$N_{A/k}(z)=(-1)^{\ell}m(0)=(-1)^{\ell+1}q_0$.  The norm is also the product of the values of $a$ over the points of $Z$ with multiplicities over $\bar k$; the two descriptions agree by polynomiality in the coefficients of $m$, so no separability is assumed.
\end{proof}

\begin{definition}[Family of fat points]\label{def:family}
A \emph{family of fat points} is a finite set $\{Z_\lambda\}_\lambda$ of fat points with pairwise disjoint supports, each on a branch, with $A_\lambda=\cO_{Z_\lambda}=k[z_\lambda]$ of length $\ell_\lambda$, together with deformation parameters for each member as follows.  Let $n=\dim X$.  A codimension-$n$ complete intersection $Z_\lambda$ of a smooth variety has
\[
 \Def_{Z_\lambda}=\Hom(I_{Z_\lambda}/I_{Z_\lambda}^2,\cO_{Z_\lambda})\cong A_\lambda^{\oplus n},
\]
of dimension $n\ell_\lambda$.  A deformation class is an $n$-tuple deforming the $n$ relations.  Choosing $k$-bases of the $n$ copies of $A_\lambda$ gives $n\ell_\lambda$ \emph{deformation parameters}.  We require that these parameters are $n\ell_\lambda$ of the slice coordinates $t$, and that the parameter sets of distinct members are disjoint.  The corresponding $n\ell_\lambda$ directions form the \emph{deformation block} $B_\lambda$ of $Z_\lambda$.  For $n=2$ one has $\Def_{Z_\lambda}\cong A_\lambda\oplus A_\lambda$ of dimension $2\ell_\lambda$, a deformation class is a pair $(a,b)$ deforming the two relations, and the block carries $2\ell_\lambda$ coordinates.
\end{definition}

\subsection{The trace pairing}

\begin{definition}[Trace pairing]\label{def:trans}
Let $Z$ be a fat point on a branch with local algebra $A$, let $\Def_Z\cong A^{\oplus n}$, and let $g\in A$.  The \emph{trace pairing} of $(Z,g)$ is the $n$-linear form on $\Def_Z$
\[
 B(v_1,\dots,v_n)=\kappa_0\,\Res_A\bigl(g\cdot\det[v_1|\cdots|v_n]\bigr),\qquad \kappa_0\in k^\times,
\]
where $\Res_A$ is the residue, so $\Res_A(a)=[z^{\ell-1}]a$ in the monogenic presentation.  The constant $\kappa_0$ absorbs the normalisation of the residue; we fix $\kappa_0=1$.
\end{definition}

\noindent For $n=2$ and $v_1=(a,b)$, $v_2=(a',b')$ this is the bilinear form
\[
 B\bigl((a,b),(a',b')\bigr)=\kappa_0\,\Res_A\bigl(g\cdot(ab'-a'b)\bigr)
\]
on $\Def_Z\cong A\oplus A$.  At $n=2$ it agrees with the restriction of Bottacin's pairing \cite{Bot}, displayed in \eqref{eq:bot} with $g=s|_A$, by Proposition~\ref{prop:localmodel}.

\begin{lemma}[Norm form of the block]\label{lem:normform}
For $n=2$, in the deformation parameters of $Z$, the Pfaffian of the trace pairing is $\pm \kappa_0^{\ell}\,N_{A/k}(g)$.
\end{lemma}

\begin{proof}
In the basis $(z^i,0)$, $(0,z^a)$ of the two halves of $A\oplus A$ the matrix is $\left(\begin{smallmatrix}0&\kappa_0H\\-\kappa_0H^{\mathsf T}&0\end{smallmatrix}\right)$ with $H_{ia}=\Res_A(g\,z^{i+a})$.  The residue Hankel identity gives $\det H=\pm N_{A/k}(g)$: partial fractions give $\Res_A(b)=\sum_r b(\zeta_r)/m'(\zeta_r)$ (expand $b$ in the Lagrange basis $\prod_{s\neq r}\frac{u-\zeta_s}{\zeta_r-\zeta_s}$ and extract $[u^{\ell-1}]$), so $H=V\,\mathrm{diag}\bigl(g(\zeta_r)/m'(\zeta_r)\bigr)\,V^{\mathsf T}$ for the Vandermonde matrix $V=(\zeta_r^{\,i})$, and $\det(V)^2\prod_r m'(\zeta_r)^{-1}=\pm1$ since $\prod_r m'(\zeta_r)=\pm\operatorname{disc}(m)=\pm\det(V)^2$.  The Pfaffian of $\left(\begin{smallmatrix}0&\kappa_0H\\-\kappa_0H^{\mathsf T}&0\end{smallmatrix}\right)$ is $\pm \kappa_0^{\ell}\det H$.  For the top coefficient only the nonzero unit factors matter.  All identities are polynomial in the coefficients of $m$ and $g$: on the open set of distinct roots $V$ is invertible, and since both sides are polynomials they agree on the whole family.  At $\ell=1$ the identity reads $\Pf=\pm\kappa_0\,g$ (the general sign being $(-1)^{\ell(\ell-1)/2}$).
\end{proof}

\subsection{Weights}

\begin{definition}[Weight; norm factor]\label{def:transv}
Let $Z$ be a fat point on a branch of length $\ell$, with local algebra $A=k[z]$; let $w'=\sum_{a<\ell}p_a z^a\in A$.  For a \emph{weight} $g\in A$ of the shape $g=z\cdot h(z,w')$ with $h\in k[z,\tau]$ arbitrary ($z$ the branch variable, $\tau$ the tangential variable, specialized to the gluing coordinate $w'$), the \emph{norm factor} is $\beta=N_{A/k}(g)$ (the Pfaffian of the block is $\pm \kappa_0^\ell\beta$, Lemma~\ref{lem:normform}).  The \emph{attached coordinates} $(q,p)$ are the deformation parameters: the coefficients of the characteristic polynomial $m=u^\ell-q(u)$ of $z$ and the coordinates $p_a$ of $w'$.  There are $2\ell$ of them, and they are the slice coordinates of the block.

A second block type is the \emph{length-one coordinate block}: a single coordinate $t_i$, the norm factor of a fat point of length $\ell=1$ with $g=\gamma z$ (the tangential direction empty; $\gamma$ is the \emph{transfer coefficient}), which is $\pm\gamma q_0$, with attached coordinate $q_0=t_i$.  In general one attaches only the coordinates that actually occur in $\beta$.
\end{definition}

\noindent By Lemma~\ref{lem:leaf} the factor $z$ contributes exactly the leaf $q_0$ to the norm.  The shape $g=z\cdot h(z,w')$ is the restriction of a surface weight $s=u\cdot h(u,v)$ to the graph $v=w'(u)$ of the fat point.

\subsection{The unit top coefficient}\label{sec:constant}

Write the weight of a block of length $\ell$ as $g=z\cdot h$ with $N(z)=\pm q_0$, and call $h_{\mathrm{tan}}:=h|_{z=0}$ the \emph{tangential restriction} of the coordinate coefficient $h$ (a function of the tangential coordinate $\tau$, the second variable of $h$, specialized to the gluing coordinate $w'$; it is a coordinate along $D$, not a section of the normal bundle of $D$, and not the branch coordinate $z$).  Then
\[
 T(g)=\bigl(T_1\bigr)^{\ell},\qquad
 T_1=[\,\tau^{p-1}\,]\,h_{\mathrm{tan}}(\tau)^{p-1},
\]
exactly.  The weight is \emph{transverse to the branch} when $T_1\neq0$.

\begin{theorem}[Unit top coefficient, codimension two]\label{thm:utc}
Let $g=z\cdot h(z,w')$ with $w'=\sum_{a<\ell}p_a z^a$ and $h\in k[z,\tau]$ arbitrary; put $h_{\mathrm{tan}}(\tau)=h(0,\tau)$ and $T_1=[\,\tau^{p-1}\,]\,h_{\mathrm{tan}}(\tau)^{p-1}$.  Then
\[
 T(g)=\Bigl[\prod_{0\leq i<\ell}q_i^{p-1}\prod_{0\leq a<\ell}p_a^{p-1}\Bigr]\,N_{A/k}(g)^{p-1}=\bigl(T_1\bigr)^{\ell}
\]
exactly.
\end{theorem}

\begin{lemma}[Linear socle calculus]\label{lem:lin-socle}
Let $R$ be a commutative $k$-algebra, $W$ any $\ell\times\ell$ matrix over $R$, and $G\in R[y_1,\dots,y_\ell]$ arbitrary.  Then
\[
 \Bigl[\prod_{a<\ell}p_a^{p-1}\Bigr]\,G(Wp)=\det(W)^{p-1}\,\Bigl[\prod_{a<\ell}y_a^{p-1}\Bigr]\,G(y),
\]
where the substitution $y=Wp$ is made on the left and the extraction on the right is in the variables of $G$.
\end{lemma}

\begin{proof}
The linear substitution $\Phi_W$ preserves the Frobenius box: $\Phi_W(y_i)^p=\sum_j W_{ij}^p\,y_j^p=0$ modulo $(y_1^p,\dots,y_\ell^p)$, so both extractions depend only on $G$ modulo the box.  Write $G\equiv\sum_\beta c_\beta y^\beta$ with $0\leq\beta_r\leq p-1$.  In $G(Wp)$ a term $y^\beta$ has total degree $|\beta|$, so only $|\beta|=\ell(p-1)$ can contribute to the socle monomial $\prod_a p_a^{p-1}$, and within the box the only such $\beta$ is $(p-1)\mathbf 1$.  In the box one has, by Lemma~\ref{lem:socle},
$\prod_i\bigl(\sum_j W_{ij}y_j\bigr)^{p-1}=\det(W)^{p-1}\prod_a y_a^{p-1}$
(an identity in the entries of $W$: valid for invertible $W$ by that lemma, hence for all $W$ by polynomiality); extracting coefficients gives the claim.
\end{proof}

\begin{lemma}[Restriction and extraction]\label{lem:restrict}
Let $\Psi\in k[e_1,\dots,e_k]$ be symmetric in $\zeta_1,\dots,\zeta_k$ (written in the elementary symmetric functions).  Then:
\begin{enumerate}
\item[\rm(i)] $\displaystyle\Bigl[\prod_{j<k}e_j^{p-1}e_k^{0}\Bigr]\Psi=\Bigl[\prod_{j<k}e_j^{p-1}\Bigr]\,\Psi\big|_{\zeta_k=0}$, where $\Psi|_{\zeta_k=0}$ is read as the polynomial in $e_1,\dots,e_{k-1}$ obtained by $e_k\mapsto0$ and $e_j\mapsto e_j(\zeta_1,\dots,\zeta_{k-1})$; the superscript $0$ on $e_k$ marks the $e_k$-free part of the extraction;
\item[\rm(ii)] $\displaystyle\Bigl[\prod_{j\leq k}e_j^{p-1}\Bigr]\,\bigl(e_k^{p-1}\Psi\bigr)=\Bigl[\prod_{j<k}e_j^{p-1}e_k^{0}\Bigr]\Psi$.
\end{enumerate}
\end{lemma}

\begin{proof}
Writing $\Psi=\sum_m e_k^m\Psi_m(e_1,\dots,e_{k-1})$, the coefficient of $e_k^{0}$ is $\Psi_0$, and the restriction $\zeta_k=0$ maps $e_k$ to $0$ and $e_j$ to the elementary symmetric functions of the remaining roots, so it returns exactly $\Psi_0$; this is {\rm(i)}.  For {\rm(ii)}, in $e_k^{p-1}\Psi$ every term of $e_k$-degree $m\geq1$ has $e_k$-degree $p-1+m$ and does not contribute to the coefficient of $e_k^{p-1}$, while the $m=0$ term contributes $\Psi_0$.
\end{proof}

\begin{proof}[Proof of Theorem~\ref{thm:utc}]
Write $p(\zeta)=\sum_{a<\ell}p_a\zeta^a$ and let $\zeta_1,\dots,\zeta_\ell$ be the
roots of $m$.  By multiplicativity of the norm and Lemma~\ref{lem:leaf},
\[
 N_{A/k}(g)=(-1)^{\ell+1}q_0\prod_{r=1}^{\ell}h(\zeta_r,p(\zeta_r)),
\]
the product being symmetric in the $\zeta_r$, hence a polynomial in $(e,p)$.
As $p-1$ is even, $N_{A/k}(g)^{p-1}=q_0^{p-1}\prod_r h(\zeta_r,p(\zeta_r))^{p-1}$
with no sign.  The explicit factor $q_0^{p-1}$ supplies the whole $q_0$-coefficient, so the remaining factor is taken at $q_0$-degree $0$, that is at $e_\ell$-degree $0$; with $q_0=\pm e_\ell$ and $q_i=\pm e_{\ell-i}$ ($i\geq1$), all signs dying at the exponent $p-1$, one gets
\[
 T(g)=\Bigl[\prod_{j<\ell}e_j^{p-1}e_\ell^{0}\Bigr]\,S,\qquad
 S:=\Bigl[\prod_{a<\ell}p_a^{p-1}\Bigr]\prod_{r=1}^{\ell}h(\zeta_r,p(\zeta_r))^{p-1}.
\]

\medskip\noindent\emph{Step 1 (Vandermonde change).}  The variables $p_a$ occur only through the values $y_r=p(\zeta_r)$, that is $y=Wp$ with $W=(\zeta_r^{\,a})_{r,a}$ and $\det W=\Delta_\ell=\prod_{r<s}(\zeta_s-\zeta_r)$.  Applying Lemma~\ref{lem:lin-socle} to $G(y)=\prod_r h(\zeta_r,y_r)^{p-1}$, for which $G(Wp)=\prod_r h(\zeta_r,p(\zeta_r))^{p-1}$, and using the disjoint-variable factorization $[\prod_r y_r^{p-1}]G=\prod_r\psi(\zeta_r)$ with $\psi(\zeta)=[\,y^{p-1}\,]h(\zeta,y)^{p-1}$, one obtains
\[
 S=\Delta_\ell^{\,p-1}\prod_{r=1}^{\ell}\psi(\zeta_r).
\]
(Both sides are polynomials in the $\zeta_r$ and symmetric, since $p-1$ is even and $\Delta_\ell^{p-1}$ is invariant under permutations; so the identity descends to the elementary symmetric functions.)

\medskip\noindent\emph{Step 2 (descent).}  Restrict to $e_\ell$-degree $0$, that is $\zeta_\ell=0$.  From $\Delta_\ell=\Delta_{\ell-1}\prod_{r<\ell}(\zeta_\ell-\zeta_r)$ one gets $\Delta_\ell^{p-1}|_{\zeta_\ell=0}=\Delta_{\ell-1}^{\,p-1}e_{\ell-1}^{\,p-1}$ (the sign $(-1)^{\ell-1}$ evaluates to $1$ at the exponent $p-1$), and $\prod_{r\leq\ell}\psi(\zeta_r)|_{\zeta_\ell=0}=\psi(0)\prod_{r<\ell}\psi(\zeta_r)$ with $\psi(0)=[\,y^{p-1}\,]h(0,y)^{p-1}=T_1$.  Writing
$\mathcal T_k:=[\prod_{j<k}e_j^{p-1}e_k^{0}]\,\Delta_k^{\,p-1}\prod_{r\leq k}\psi(\zeta_r)$, Lemma~\ref{lem:restrict}(i) turns the $e_\ell$-degree-$0$ extraction into the restriction $\zeta_\ell=0$, and Lemma~\ref{lem:restrict}(ii) shows that the explicit factor $e_{\ell-1}^{\,p-1}$ supplies the whole $e_{\ell-1}$-extraction:
\[
 \mathcal T_\ell=T_1\cdot\mathcal T_{\ell-1}.
\]

\medskip\noindent\emph{Step 3 (induction).}  The base case is $\mathcal T_1=[e_1^{0}]\psi(\zeta_1)=\psi(0)=T_1$ ($\Delta_1=1$).  Hence $\mathcal T_\ell=T_1^{\ell}$, and $T(g)=\mathcal T_\ell$ by the display preceding Step 1.  All identities are exact: the only sign factors are $(-1)$-powers at the exponent $p-1$, which evaluate to $1$, and Lemma~\ref{lem:lin-socle} contributes the exact factor $\det(W)^{p-1}$.  Thus $T(g)=T_1^{\ell}$.
\end{proof}

\begin{remark}[Checks and scope]\label{rem:routeg}
For weights affine in the gluing coordinate, $g=z\cdot(a(z)+b(z)w')$, the theorem gives $T=b(0)^{\ell(p-1)}$.  A naive factorization over the root coordinates $(\zeta_r,y_r=p(\zeta_r))$, without the Jacobian, is false (for $h=1+\tau$, $p=3$, $\ell=2$ it predicts $0$ against the true value $T_2=1=T_1^2$); the Vandermonde factor $\Delta^{p-1}$ of Step 1 is the missing correction, and its descent supplies the factor $e_{\ell-1}^{\,p-1}$ that closes the induction.
\end{remark}

\begin{lemma}[Top coefficient of a member]\label{lem:membertop}
Let $Z$ be a fat point of length $\ell$ on a branch with attached coordinates $(q,p)$, let $g=H(z,w')$ with $H\in k[u,v]$ arbitrary, and let $x_0$ be the support of $Z$ at the origin of the attached coordinates. Let $s$ be the surface weight with local function $H$ in the coordinates $(u,v)$ at $x_0$. Then
\[
 T(g)=\Tr_X(s)(x_0)^{\ell}.
\]
For $g=z\cdot h(z,w')$ this returns $T_1^{\ell}$.  In particular $T(s|_{\cO_Z})=\Tr_X(s)(x_0)^{\ell}$ for every member, on $D$ or away from $D$.
\end{lemma}

\begin{proof}
First assume $\Delta_\ell\neq 0$, so that $m$ is separable and $Z$ splits into $\ell$ distinct points. Write $(\zeta_1,y_1),\dots,(\zeta_\ell,y_\ell)$ for the coordinates of the split points in the deformation space. The attached coordinates are given by
\[
 q_i=\pm e_{\ell-i}(\zeta),\qquad y=Wp,\qquad \det W=\Delta_\ell.
\]
Both $(q,p)$ and $(\zeta,y)$ are slice coordinates at $Z$ on this stratum. The Jacobian of the change is
\[
 \det\frac{\partial(q,p)}{\partial(\zeta,y)}
 =\pm\Delta_\ell\cdot\Delta_\ell^{-1}
 =\pm1.
\]
Hence the frame factor of $\partial_{[2\ell]}$ under this change is a nonzero constant.

The number $T(g)$ is the socle extraction of $N^{p-1}$ in the coordinates $(q,p)$. It is the top coefficient of the section
\[
 \sigma=N(g)\,\partial_{[2\ell]}.
\]
The top coefficient of a section is the Frobenius trace of $\sigma^{p-1}$ and does not change under a change of slice coordinates. In the coordinates of the split points the frame factor is constant, and $N$ is the product of the values of $H$ at the points:
\[
 N(g)=\prod_{r=1}^{\ell}H(\zeta_r,y_r).
\]
The variables $\zeta_r$ and $y_r$ are independent, the socle extraction of a product over disjoint variables is the product of the extractions, and therefore
\[
 T(g)
 =\prod_{r=1}^{\ell}\bigl[\zeta^{p-1}y^{p-1}\bigr]H(\zeta,y)^{p-1}
 =\Tr_X(s)(x_0)^{\ell}.
\]
Each factor is the value of $\Tr_X(s)$ at $x_0$ by the trace formula. Both sides are polynomial in the coefficients of $m$ and of $g$, so the identity holds for every fat point on a branch. If one keeps $y_r=p(\zeta_r)$ as dependent variables instead, then Remark~\ref{rem:routeg} applies and a further Vandermonde factor is needed.
\end{proof}

\section{Transfer on the member locus and identification with Bottacin's tensor}\label{sec:hat}

This section constructs the transferred polyvector $\hat s$ on the member locus and identifies it with Bottacin's tensor on a surface.

The construction of Definition~\ref{def:tpoly} is uniform in $n=\dim X$.  Only the identification with Bottacin's tensor is carried out on a surface, where $n=2$.  Throughout, $X$ is a smooth projective surface and $s\in H^0(X,\omega_X^{-1})$ is nonzero.  Write $U\subset M$ for the member locus, the locus of objects that admit member data.

The members $Z_\lambda$ are fat points of lengths $\ell_\lambda\ge1$, intrinsic to the object or supplied by a presentation, and at a member the deformation module is $\Def_{Z_\lambda}=A_\lambda^{\oplus 2}$ of dimension $2\ell_\lambda$.  Reduced points are the case $\ell_\lambda=1$.  As in Section~\ref{sec:prelim}, $\rho=\sum_\lambda\ell_\lambda$ counts leaf \emph{units} at a member-saturated datum: a member of length $\ell$ splits into $\ell$ units.

Let $Z$ be a member of length $\ell$ with local algebra $A=k[z]/(m)$ and weight $g=s|_A$.  On the open locus where $m$ is separable the support splits into $\ell$ points $x_1,\dots,x_\ell$ and
\[
 \Def_Z=A^{\oplus 2}\cong\bigoplus_{a=1}^{\ell} k^{\oplus 2}
\]
splits into $\ell$ blocks of $2$ coordinates each.  The values of $g$ at the support points satisfy
\[
 \prod_{a=1}^{\ell} g(x_a)=N_{A/k}(g).
\]

\begin{definition}[Transferred polyvector]\label{def:tpoly}
The \emph{unit polyvector} of the $a$-th support point is the image of $s|_{x_a}\in\wedge^nT_{x_a}X$ under the inverse of the differential of the support map restricted to its block.  In the block coordinates
\[
 \beta_a=g(x_a)\,\partial_{q_{a,1}}\wedge\cdots\wedge\partial_{q_{a,n}} .
\]
The \emph{transferred polyvector} on $U$ is
\[
 \hat s=\sum_{\text{all units}}\beta_a ,
\]
and at a member-saturated datum its \emph{normalized power} is
\[
 \hat s^{[\rho]}=\bigwedge_{\text{all units}}\beta_a .
\]
\end{definition}

\noindent The sum of the unit polyvectors is symmetric in the support points, hence a polynomial in the coefficients of $m$ and $g$; as in Lemma~\ref{lem:leaf} it extends by polynomiality to all members, including those whose algebra is not split.  For a reduced member there is a single unit and $N_{A/k}(g)=s(x)$.

\noindent One has $\hat s^{\wedge\rho}=\rho!\,\hat s^{[\rho]}$, and the naive wedge is not used since $\rho!=0$ when $\rho\geq p$.

\begin{proposition}[Well-definedness]\label{prop:hat-well}
The polyvector $\hat s$ is independent of the choices in Definition~\ref{def:tpoly}, and it extends across the collision strata of $U$ as a section of $\wedge^n\operatorname{Der}(\cO_M)$.
\end{proposition}

\begin{proof}
Each unit polyvector is the image of the intrinsic $n$-vector $s|_{x_a}$ under the inverse of the differential of the support map restricted to its block, hence does not depend on the block coordinates, and a permutation of the units permutes the summands of $\hat s$ without changing the sum.  For the extension across collisions, work in the attached coordinates $(q,p)$ of the member (Definition~\ref{def:transv}), which are the slice coordinates of the block and form a smooth chart through the collision strata.  In these coordinates $\hat s$ is given by the structure constants of the block pairing
\[
 \kappa_0\,\Res_A\bigl(g\cdot\det[\,\cdot\,|\cdots|\,\cdot\,]\bigr),
\]
and these are polynomials in $(q,p)$: the coefficients of $g=s|_{\cO_Z}$ in the basis of $A$ are polynomial in the deformation parameters, and likewise the leaf coefficient $\pm\kappa_0^{\ell}N_{A/k}(g)$ of Lemma~\ref{lem:normform}.  Hence $\hat s$ is regular where units collide; the argument applies to any number of colliding units and extends by polynomiality to members with non-reduced schemes, as in Definition~\ref{def:tpoly}.
\end{proof}

\begin{theorem}[Vanishing of the Schouten bracket]\label{lem:schouten}
One has $[\hat s,\hat s]=0$.
\end{theorem}

\begin{proof}
In the split coordinates of the members
\[
 \hat s=\sum_a f_a\,\partial_{q_{a,1}}\wedge\cdots\wedge\partial_{q_{a,n}},
\]
where $f_a=g(x_a)$ is a polynomial in the coordinates of the $a$-th unit block only.  Each term of $[\hat s,\hat s]$ is the product of an entry of $\hat s$ and a derivative of an entry.  If the two entries lie in the same unit block, the alternating sum over $2n-1$ indices drawn from $n$ coordinates vanishes.  If they lie in distinct unit blocks, the term is $f_b\,\partial f_a$ with $a\neq b$, taken in a coordinate of the $b$-th block, and it vanishes because $f_a$ does not involve the coordinates of the $b$-th block.  For $n=2$ the same computation is the coordinate formula
\[
 [\pi,\pi]^{abc}=\sum_l\bigl(\pi^{al}\partial_l\pi^{bc}+\pi^{bl}\partial_l\pi^{ca}+\pi^{cl}\partial_l\pi^{ab}\bigr).
\]
The identity is a polynomial identity in the coefficients of $m$ and $g$, hence holds at every member by polynomiality; it holds at every point of the locus of separated members, and across the collision strata of $U$ by Proposition~\ref{prop:hat-well}.
\end{proof}

\noindent The identity $[\hat s,\hat s]=0$ is the Poisson condition on a surface.

\noindent The transfer $\hat s$ is relative to the presentation unless the members are intrinsic to the object, as for Hilbert schemes and for moduli whose general object has singular points realizing the leaf.  For principal bundles the members are Hecke points.  On a surface the tensor $\theta_s$ of Proposition~\ref{prop:localmodel} is intrinsic in all cases, and it is the regular extension of $\hat s|_U$.

\begin{remark}[Unit wedge and the norm]\label{rem:unitnorm}
In the split coordinates of a member of length $\ell$ one has
\[
 \prod_{a=1}^{\ell} g(x_a)=N_{A/k}(g),
\]
so the coefficient of $\beta_1\wedge\cdots\wedge\beta_\ell$ is the norm exactly.  In the attached coordinates the change of variables has Jacobian $\pm1$ in the cases used below, namely $n=2$ and $\ell=1$, by Lemma~\ref{lem:membertop}.
\end{remark}

\begin{proposition}[Identification of the transfer]
\label{prop:localmodel}
Let $X$ be a Poisson surface, let $\rho>0$, and let $U\subset M$ be the
member locus of member-saturated data with length-one members, possibly empty.
Then on $U$ one has
\[
\hat s\big|_U=\sum_a\beta_a=\theta_s\big|_U,
\qquad
\hat s^{[\rho]}\big|_U=\bigwedge_a\beta_a=\theta_s^{[\rho]}\big|_U,
\]
where $\beta_a$ are the unit polyvectors of the members as in
Definition~\ref{def:tpoly} and $\theta_s$ is Bottacin's Poisson tensor
\cite{Bot}.
\end{proposition}

\begin{proof}
At a length-one member $Z$ at $x$ the local algebra is $A=k$ and the deformation module is $A\oplus A$.  Under Serre duality the restriction of Bottacin's pairing \eqref{eq:res} to these directions is the trace pairing of Definition~\ref{def:trans} with weight $g=s|_A$.  The pairing has matrix $\kappa_0 g(x)J$, and the dual bivector is
\[
g(x)\,\partial_{q_1}\wedge\partial_{q_2},
\]
which is the unit $\beta_a$ of Definition~\ref{def:tpoly}.  Hence $\hat s|_U=\theta_s|_U$.  For $n=2$ the wedge of the units equals the Pfaffian power $\theta_s^{[\rho]}$, since the normalized wedge of a bivector does not depend on the choice of units.  The local computations in the principal and framed cases are Lemma~\ref{lem:Gblock} and Lemma~\ref{lem:inst-identify}.
\end{proof}

\noindent The presentations of Section~\ref{sec:transfer} produce length-one members, and the identification is used only in that case.  The equality for a member of length $\ell>1$ is not needed below.

\noindent The identification is special to surfaces: only there is a
bivector an anticanonical section, hence the weight of the trace
pairing.  Since $\theta_s$ is a global section of
$\wedge^2\operatorname{Der}(\cO_M)$ and agrees with $\hat s$ on $U$, the
leafrian is $\LF_s=\theta_s^{[\rho]}\wedge\Omega$, whose restriction
to $U$ is the transfer leaf factor.  The regularity of $\LF_s$ is likewise intrinsic, by Lemma~\ref{lem:global}.

\begin{remark}[Input]\label{rem:hat-input}
The transfer uses the local algebra of the member and the weight,
through the norm factor of Definition~\ref{def:transv}.  The
construction of \cite{Bot} uses the cup product and the trace of
$\Ext^1(E,E)$; the cup of two classes lands in $\Ext^2$ and the trace
pairing with $k$ requires the top $\Ext^2$, and the two coincide on a
surface.  On the overlap the two constructions agree by
Proposition~\ref{prop:localmodel}.
\end{remark}

\section{The leafrian and the splitting criterion}\label{sec:principle}

The ingredients of Sections~\ref{sec:trans} and~\ref{sec:hat} assemble into an anticanonical section whose top coefficient is the product of the block coefficients.

\subsection{Frame data and deficiency}

\begin{definition}[Euler fields and Euler frame]\label{def:euler}
Given transverse coordinates $c_1,\dots,c_c$ at $x$, their \emph{Euler fields} are the infinitesimal scalings
\[
 w_i=c_i\,\partial_{c_i},\qquad 1\le i\le c ,
\]
and their wedge
\[
 \Omega=w_1\wedge\cdots\wedge w_c
\]
is the \emph{Euler frame}.  Its radial determinant is
\[
 \eta=\det[w_1,\dots,w_c]=\Bigl(\prod_jc_j\Bigr)\Delta
\]
with $\Delta$ a unit.
\end{definition}

\begin{definition}[Frame datum]\label{def:frame}
A \emph{frame datum} at $x$ consists of four parts:
\begin{itemize}
\item[\rm(i)] \emph{blocks}: a family of blocks with trace pairings on the leaf directions, the pairings of distinct blocks vanishing against each other;
\item[\rm(ii)] \emph{additivity}: the deformation space is the direct sum of the blocks;
\item[\rm(iii)] \emph{transverse frame}: an Euler frame in the sense of Definition~\ref{def:euler};
\item[\rm(iv)] \emph{coverage}: the attached coordinates of the blocks are pairwise disjoint and, together with the transverse coordinates $c_j$, exhaust the slice coordinates.
\end{itemize}
\end{definition}

\noindent The frame datum is automatically admissible over surfaces: the presentation theory of Section~\ref{sec:transfer} supplies global Euler fields.  We write $\partial_{[\mathrm{leaf}]}$ and $\partial_{[\mathrm{trans}]}$ for the ordered wedges of the leaf and transverse coordinate frames, and $\partial_{[d]}$ for the full frame.

\begin{definition}[Realization point]\label{def:realization}
A point at which the local data of a frame datum exist is a \emph{realization point}.  The realization points may form a proper subset.
\end{definition}

\begin{definition}[Leaf $n$-vector]
\label{def:leafnv}
Let $M$ be a variety of dimension $d$ with regular locus
$M^{\mathrm{reg}}$, let $n,\rho\ge 1$, and let $c\ge 0$ satisfy
$n\rho+c=d$.
A \emph{leaf $n$-vector} of rank $\rho$ on $M$ is a triple
$\bigl(\vartheta,(\beta_a)_{a=1}^{\rho}\bigr)$ where
$\vartheta\in H^0(M,\wedge^n \operatorname{Der}(\cO_M))$ and the
$\beta_a$ are decomposable sections of
$\wedge^n\operatorname{Der}(\cO_M)$ on $M^{\mathrm{reg}}$ such that
\[
\vartheta\big|_{M^{\mathrm{reg}}}=\sum_{a=1}^{\rho}\beta_a,
\]
and the $\beta_a$ span a distribution of rank $n\rho$ on
$M^{\mathrm{reg}}$.
The distribution is the \emph{leaf} of $\vartheta$.
\end{definition}

\begin{definition}[Leaf factor and leafrian]
\label{def:leaffactor}
Let $\bigl(\vartheta,(\beta_a)\bigr)$ be a leaf $n$-vector of rank
$\rho$ on $M$.
The \emph{leaf factor} of this data is
\[
\vartheta^{[\rho]}\;=\;\beta_1\wedge\cdots\wedge\beta_\rho
\]
on $M^{\mathrm{reg}}$, extended by reflexivity when $M$ is normal.
Let $\Omega=w_1\wedge\cdots\wedge w_c$ be an Euler frame with
$n\rho+c=d$.
The \emph{leafrian} of the data is
\[
\LF\;=\;\vartheta^{[\rho]}\wedge\Omega
\;\in\;H^0(M,\omega_M^{[-1]}),
\]
and $\LF\in H^0(M,\omega_M^{-1})$ when $M$ is Gorenstein.
\end{definition}

\begin{remark}
The leaf factor is defined from the chosen units, and no independence of
decomposition is required in general. For $n=2$ the wedge
$\beta_1\wedge\cdots\wedge\beta_\rho$ is independent of the choice of
units and equals the Pfaffian power of the bivector $\vartheta$. In that
case the units may be forgotten and $\vartheta^{[\rho]}$ is a global
section determined by $\vartheta$ alone. For $n>2$ the units are kept in
the data. The naive wedge satisfies
\[
\vartheta^{\wedge\rho}=\rho!\,\vartheta^{[\rho]}
\]
when $\rho<p$. The normalized wedge remains defined when $\rho\ge p$,
and it is the object used throughout.
\end{remark}

\begin{remark}[Independence of the frame]
\label{rem:frame-indep}
Two Euler frames whose difference lies in the leaf distribution give the
same leafrian. Two presentations give frames whose radial determinants
differ by a regular unit, and the corresponding leafrians differ by that
unit. If the traces of the $(p-1)$ powers are nonzero constants, the unit
is a constant. This is the content of Lemmas \ref{lem:leafclass} and
\ref{lem:unique} and of Theorem \ref{thm:hd-canonicity}, and it is what
makes $[s]\mapsto[\LF_s]$ well defined on projective spaces.
\end{remark}

\begin{definition}[Leaf $n$-vector of a weight]
\label{def:leafweight}
Let $X$ be a smooth projective Poisson surface, let
$s\in H^0(X,\omega_X^{-1})$ be nonzero, and let $\theta_s$ be Bottacin's
Poisson tensor on $M$.
Set $n=2$, and let $\rho$ be the number of units of a member-saturated datum as in Section~\ref{sec:prelim}.
The \emph{leaf $n$-vector of $s$} is $\vartheta_s=\theta_s$, with leaf
factor $\vartheta_s^{[\rho]}$ in the sense of Definition~\ref{def:leaffactor}.
The \emph{leafrian of $s$} is
\[
\LF_s\;=\;\vartheta_s^{[\rho]}\wedge\Omega.
\]
\end{definition}

\begin{remark}
For $n=2$ this is Definition \ref{def:leaffactor} applied to any leaf
$n$-vector with $\vartheta=\theta_s$, and the result does not depend on a
unit decomposition. By Proposition \ref{prop:localmodel} the leaf factor
$\theta_s^{[\rho]}$ agrees on the member locus $U$ with the transfer leaf
factor $\hat s^{[\rho]}$, and $\LF_s$ agrees on $U$ with the
transfer leafrian. No global unit decomposition of $\theta_s$ is
required. In the framed case the weight lies in
$H^0(X,\omega_X^{-1}(-2L))$, and the same definition applies on the open
set $X\setminus L$.
\end{remark}

Let $f\in\widehat\cO$ and consider
\begin{equation}\label{eq:fact}
 f=a\cdot\prod_\lambda\beta_\lambda+\delta,\qquad a\in k^\times,
\end{equation}
where each $\beta_\lambda$ is a norm factor or a single coordinate (the transverse coordinates $c_j$ are single-coordinate blocks, so they occur among the $\beta_\lambda$), and where the attached coordinates of distinct blocks are pairwise disjoint and, together with the $c_j$, exhaust the slice coordinates $t$.

\begin{definition}[Deficiency]\label{def:def}
The correction $\delta$ in \eqref{eq:fact} is \emph{deficient} if no monomial of the expansion of $f^{p-1}$ that draws at least one factor from $\delta$ equals the top monomial $t^{(p-1)\mathbf 1}$.  One sufficient form is \emph{exact factorization}: $f=a\prod_\lambda\beta_\lambda$ with $\delta=0$.
\end{definition}

\begin{remark}[Why the forms suffice]\label{rem:defreason}
In the Frobenius box every exponent is at most $p-1$, so a factor that cannot supply the exponent $p-1$ of some variable keeps the top monomial out of reach.  The gap must be measured against the whole term $\beta^{p-1-j}\delta^{j}$ of $f^{p-1}$, not against $\delta$ alone.  Exact factorization covers the member-saturated cases; the remaining cases are handled by the reduced degeneracy complex.  A non-constant unit can refill a deficient variable: with $p=3$, $\ell=1$, $\beta=\gamma q_0(h_0+h_1p_0)$ and a perturbation $\delta'=p_0$ one has $b_0(\beta(1+\delta'))=\gamma^2(h_1^2+4h_0h_1+h_0^2)$, not $\gamma^2h_1^2$.
\end{remark}

\subsection{The shape theorem}

\begin{theorem}[Shape theorem]\label{thm:head}
Let $Y$ be smooth, let $\sigma\in H^0(Y,\omega_Y^{-1})$ have local expression $f$ at a closed point $x$ in a fixed frame, and assume:
\begin{enumerate}
\item[\rm(1)] \emph{blocks}: $f$ admits the factorization \eqref{eq:fact}, each $\beta_\lambda$ being a norm factor or a single coordinate;
\item[\rm(2)] \emph{coverage}: the attached coordinates of distinct blocks are pairwise disjoint and, together with the transverse coordinates $c_j$, exhaust the slice coordinates $t$;
\item[\rm(3)] \emph{extraction}: the correction $\delta$ of \eqref{eq:fact} is deficient.
\end{enumerate}
Then $b_0(\sigma)=a^{p-1}\prod_\lambda b_\lambda$ exactly, where $b_\lambda$ is the top coefficient of $\beta_\lambda^{p-1}$ in its own variables.  If moreover each
$b_\lambda$ is nonzero and $\Tr(\sigma^{p-1})$ is a constant function, then a constant rescaling of $\sigma$ splits $Y$, compatibly with $Z(\sigma)$.
\end{theorem}

\begin{proof}
By (2) the blocks occupy disjoint variable sets, and the extraction of $t^{(p-1)\mathbf 1}$ from $\prod_\lambda\beta_\lambda^{p-1}$ is the product of the block extractions: each variable belongs to exactly one block, so the allocation of the exponent $p-1$ is unique.  The deficiency hypothesis kills every monomial drawing on $\delta$, so $b_0=a^{p-1}\prod_\lambda b_\lambda$.  The product of nonzero factors is nonzero.  If $\Tr(\sigma^{p-1})$ is the constant $c$, rescaling $\sigma$ by a constant $\lambda$ with $\lambda^{p-1}c=1$ yields $\Tr=1$, by Theorem~\ref{thm:bk} and the algebraic closedness of $k$; compatibility is by Lemma~\ref{lem:loccompat}.  A nonzero constant multiple of $\sigma$ has the same zero scheme.  If $\rho=0$ there are no leaf blocks, the product $\prod b_\lambda$ is $1$, and $b_0=a^{p-1}$ for the product volume.
\end{proof}

\begin{lemma}[Constant unit]\label{lem:constunit}
In the frame of the datum of Definition~\ref{def:frame} the local expression of the leafrian is
$f=a\prod_\lambda\beta_\lambda$ with
$a=\pm\prod_\lambda\kappa_0^{\ell_\lambda}$ in the block coordinates of the members, and $a\in k^\times$ is constant.
\end{lemma}

\begin{proof}
The pairings of distinct members vanish against
each other, so the leaf matrix of $\theta_s$ in the deformation bases is block
diagonal, and by Lemma~\ref{lem:normform} the Pfaffian of each block is
exactly $\pm\kappa_0^{\ell_\lambda}N_{A_\lambda/k}(g_\lambda)$.  The leaf factor is the wedge of the unit polyvectors of
Definition~\ref{def:tpoly}, which for $n=2$ equals the Pfaffian power.  In the split coordinates of a member the
product of the unit coefficients is
\[
 \prod_a g(x_a)=N_{A_\lambda/k}(g_\lambda),
\]
so the same product of norms
arises.  The change from the split coordinates to the attached
coordinates of the member has Jacobian $\pm1$ by Lemma~\ref{lem:membertop} in the cases used below, namely $n=2$ and $\ell_\lambda=1$; these are the
cases in which the transferred leaf factor occurs in the statements
below, and no further constant is introduced.  Coverage gives
$\sum_\lambda 2\ell_\lambda=2\rho$, so there are no remaining leaf directions.
The Euler frame is $\Omega=(\prod_jc_j)\,\partial_{[\mathrm{trans}]}$ exactly,
and wedging with $\varpi=\pm\bigl(\prod_\lambda
\kappa_0^{\ell_\lambda}N(g_\lambda)\bigr)\,\partial_{[\mathrm{leaf}]}$ gives
$f=\pm\prod_\lambda\kappa_0^{\ell_\lambda}N(g_\lambda)\cdot\prod_jc_j$ up to
the sign of the shuffle.  Every factor other than the norms and the
coordinates is a constant.
\end{proof}

\begin{lemma}[Zero block coordinates]\label{lem:zeroblock}
Let $Z$ be a member of a frame datum with block $\Hom(I_Z/I_Z^2,\cO_Z)=A\oplus A$ and weight $g$, and let $\xi$ be a deformation class of zero component in this block.  Then $\xi$ pairs trivially with every block class under the trace pairing; if two classes have zero block component, their pairing has no local contribution at the support of $Z$.  Consequently the matrix of the pairing is block diagonal between the member blocks and the classes of zero block coordinates.
\end{lemma}

\begin{proof}
In the block coordinates the pairing is $B_g\bigl((a,b),(a',b')\bigr)=\kappa_0\Res_A\bigl(g\cdot(ab'-a'b)\bigr)$; for $\xi$ of block coordinates $(a_\xi,b_\xi)=0$ this gives $0$ against every block class $(a,b)$, and two zero-block classes contribute no residue at the support of $Z$.
\end{proof}

\subsection{The leafrian as a global section}

On a moduli space $M$ the leaf factor $\varpi$ and the Euler fields are global; the section $\LF=\varpi\wedge\Omega$ is a global anticanonical section.  On a surface the leaf factor is the Pfaffian $\theta_s^{[\rho]}$, which agrees with the transfer leaf factor $\hat s^{[\rho]}$ on the member locus $U$ by Proposition~\ref{prop:localmodel}.  The radial determinant $\eta=\prod_jc_j$ of the Euler frame is a section of $\det Q$ on the open where the leaf rank is constant.  Here $Q=T_M/\Lambda$, where $\Lambda$ is the leaf distribution spanned by the block directions.  Two frames of one datum differ by a constant matrix, so $\eta$ is well defined up to constants.  We call this class the \emph{deformation volume} $\eta_{\mathrm{can}}$.  The identification $\varpi\wedge\Omega=\varpi\otimes\eta$ holds because leaf components of the $w_i$ die against the leaf wedge; thus $\LF$ depends on the frame through $\eta$ alone.

\begin{theorem}[The leafrian]\label{thm:normsec}
Let $\rho>0$.  On a surface the section $\LF_s=\theta_s^{[\rho]}\wedge\Omega$ is a nonzero global section of $\omega_M^{[-1]}$, well defined up to a regular unit in the Euler frame.  At a member-saturated point of the member locus $U$ its local expression in the coordinates of the datum is
\[
 f=\prod_\nu N_{A_\nu/k}(g_\nu)\cdot\prod_jc_j\cdot u,
\]
where $u\in k^\times$.
\end{theorem}

\begin{proof}
The norm factors are the resultants $\operatorname{Res}(m_\nu,\hat g_\nu)$, regular in the deformation parameters, and the transverse factor is the deformation volume.  In the coordinates of the datum the leaf wedge contributes $\prod_\nu N(g_\nu)$ and the radial determinant contributes $\prod_jc_j$; by Lemma~\ref{lem:constunit} the remaining unit is constant.  At a general point of the evaluation locus the leaf rank is $2\rho$ and the class directions are transverse, so the local expression is nonzero; this is the content of Proposition~\ref{prop:frame}.  Two presentations of one object differ by constant rescalings of the class coordinates, hence two data give $\LF$ and $u\LF$ with $u\in k^\times$.  Regularity of $\LF_s$ on $M$ is intrinsic: $\theta_s^{[\rho]}$ is the Pfaffian power of Bottacin's tensor, hence a global section of $\wedge^{2\rho}T_M$, and the Euler fields are global by Lemma~\ref{lem:descent}, so $\LF_s=\theta_s^{[\rho]}\wedge\Omega$ is a global section of $\omega_M^{[-1]}$.
\end{proof}

\begin{lemma}[Leaf-class invariance]\label{lem:leafclass}
Let $\pi$ be a bivector of generic rank $2\rho$ with leaf $\Lambda=\operatorname{im}\pi^\sharp$. Let $\Omega$ and $\Omega'$ be $c$-vector fields whose difference lies in the ideal of $\Lambda$, so that every term of $\Omega-\Omega'$ has at least one $\Lambda$-component. Then
\[
 \pi^{[\rho]}\wedge\Omega=\pi^{[\rho]}\wedge\Omega'.
\]
\end{lemma}

\begin{proof}
Choose a leaf frame $e_1,\dots,e_{2\rho}$ spanning $\Lambda$ with
\[
 \pi^{[\rho]}=\operatorname{Pf}(A)\,e_1\wedge\cdots\wedge e_{2\rho},
\]
and complete to a frame of $T$ by $f_1,\dots,f_c$ projecting to a basis of $Q=T/\Lambda$. A $c$-vector with at least one $\Lambda$-component dies upon wedging with $e_1\wedge\cdots\wedge e_{2\rho}$. Hence $\pi^{[\rho]}\wedge\Omega$ sees only the $Q$-component of $\Omega$, and $\Omega$ and $\Omega'$ have the same $Q$-component.
\end{proof}

\begin{lemma}[Globality of the leafrian]\label{lem:global}
Let $\LF_s=\theta_s^{[\rho]}\wedge\Omega$ with $\Omega$ an Euler frame. Then $\LF_s$ is independent of the choice of $\Omega$ within its leaf class, and is a regular section of $\omega_M^{[-1]}$ on $M$.
\end{lemma}

\begin{proof}
Two Euler frames both project to $Q=T_M/\Lambda$, so their difference has only $\Lambda$-components; by Lemma~\ref{lem:leafclass} the wedge with $\theta_s^{[\rho]}$ is unchanged. Thus $\LF_s$ is independent of the presentation as a rational section.

On the evaluation locus $\LF_s$ is regular by the local expression of Theorem~\ref{thm:normsec}. Regularity on $M$ is intrinsic: $\theta_s^{[\rho]}$ is a global section of $\wedge^{2\rho}T_M$ as the Pfaffian power of Bottacin's tensor, and the Euler fields are global by Lemma~\ref{lem:descent}; hence $\LF_s$ is a global section of $\omega_M^{[-1]}$, with the local expression of Theorem~\ref{thm:normsec} on the evaluation locus.
\end{proof}

Existence of $\LF$ as a nonzero section is weight-blind; the weights enter only through $b_0$.  When $c=0$ the leaf fills the tangent space and $\LF_s=\theta_s^{[\rho]}$; no transverse factor is present.

\begin{lemma}[Uniqueness up to scalar]\label{lem:unique}
Let $\sigma=\varpi\wedge\Omega$ and $\sigma'=\varpi\wedge\Omega'$ be
leafrians with radial determinants $\eta'=u\cdot\eta$ for a regular unit $u$.  If $\Tr(\sigma^{p-1})$ and $\Tr(\sigma'^{p-1})$ are nonzero constants, then $u\in k^\times$ and $\sigma'=u\,\sigma$.
\end{lemma}

\begin{proof}
In the frame of the datum the leaf factors agree, so
$\sigma'=\varpi\otimes\eta'=u\cdot(\varpi\otimes\eta)=u\,\sigma$.  The two traces differ by $u^{p-1}$, so $u^{p-1}=c'/c$ with $c$ and $c'$ the two nonzero constants.  By Lemma~\ref{lem:unitrigid} one gets $u\in k^\times$.
\end{proof}

\begin{corollary}[Frame reading]\label{cor:frame}
Let $w_1,\dots,w_c$ be regular transverse fields, on $M$ or on a dense open,
with radial determinant $\eta=\det[w_1,\dots,w_c]\not\equiv0$.  Then
$\varpi\wedge(w_1\wedge\cdots\wedge w_c)=\varpi\otimes\eta$; for
$\eta=\eta_{\mathrm{can}}$ this is the leafrian.
\end{corollary}

\begin{remark}[On the Euler frame]\label{rem:omega}
The Euler frame $\Omega$ is built from the class coordinates of a
presentation: the fields $w_j=c_j\partial_{c_j}$ descend to global vector
fields on $M$ by Lemma~\ref{lem:descent}, and
$\Omega=w_1\wedge\cdots\wedge w_c$.  Two presentations give frames whose
radial determinants differ by a regular unit, and the unit is a constant
when $\Tr_X(s)\neq0$ by Lemma~\ref{lem:unique}.  The frame vanishes along
the divisors $c_j=0$, and this does not affect the top coefficient: the
class coordinates enter the local expression linearly, so the socle of
$(\prod_j c_j)^{p-1}$ equals one.
\end{remark}

\subsection{The obstruction}

\begin{definition}[Ordinary weights]\label{def:ordinary}
A weight $g$ of a member is \emph{ordinary} when its block coefficient $T(g)$ is nonzero.  By Theorem~\ref{thm:utc} this is $T_1\neq0$ for every weight of the shape $g=z\cdot h(z,w')$.  A transverse volume
$\eta$ in the coordinates $c_j$ is \emph{ordinary} when
$T(\eta):=[\prod c_j^{p-1}]\eta^{p-1}\neq0$; the product volume $\prod_jc_j$ is
ordinary with $T(\eta)=1$.
\end{definition}

\begin{remark}[The obstruction]\label{thm:obstr}
At an evaluation datum the top coefficient of $\LF$ factors as $\prod_\nu T(g_\nu)\cdot u^{p-1}$ with $u\in k^\times$, and $\LF$ splits $M$ exactly when every weight is ordinary.  The general form of this factorisation is the trace formula of Theorem~\ref{thm:TF}.
\end{remark}

\begin{proof}
The norm factors carry the block variables $q_{\nu,i}$ and the transverse volume
the variables $c_j$, pairwise
disjoint and together exhausting the slice.  The
extraction of $t^{(p-1)\mathbf1}$ from $f^{p-1}$ separates over the disjoint
variables, by Theorem~\ref{thm:head} and Lemma~\ref{lem:socle}.
Since the unit is constant by Lemma~\ref{lem:constunit}, $b_0=\prod_\nu T(g_\nu)\cdot u^{p-1}$, the transverse product contributing $T(\eta)=1$.  The same separation gives the variant of Remark~\ref{rem:nonradial}.
The criterion follows: $b_0(\LF)=0$ exactly when some $T(g_\nu)$ vanishes.  If every weight is ordinary then $b_0(\LF)\neq0$ and Theorem~\ref{thm:bk}
gives the splitting after rescaling $\LF$; one
point suffices.
\end{proof}

\begin{remark}[Variants]\label{rem:variants}
The converse of the criterion is asserted for the
leafrian of a datum, unique up to a regular unit.  A general wedge $\varpi\wedge\Omega$ with
arbitrary transverse fields need not belong to this class.
\end{remark}
\begin{remark}[Non-radial transverse volume]\label{rem:nonradial}
If the transverse volume $\eta$ is a function of the transverse coordinates alone but differs from the product volume, then
\[
 b_0=\prod_\nu T(g_\nu)\cdot T(\eta)\cdot u^{p-1},
\]
and the criterion of Remark~\ref{thm:obstr} holds with $\eta$ ordinary in addition.
\end{remark}

\subsection{The empty leaf}

When the datum has no members one has $\rho=0$, the leaf factor is $\varpi=1$, and the assembly is the deformation volume $\eta_{\mathrm{can}}$.

\begin{proposition}[Empty leaf]\label{prop:empty}
Let $\rho=0$.  Then no leafrian is formed, and the deformation volume $\eta_{\mathrm{can}}$ exists if and only if $M$ carries $d$ global vector fields vanishing at a common point with independent differentials there.  When it exists, its top coefficient is $b_0=u^{p-1}$ with $u\in k^\times$, and $\eta_{\mathrm{can}}$ splits $M$ whenever $\Tr(\eta_{\mathrm{can}}^{p-1})$ is a nonzero constant.
\end{proposition}

\begin{proof}
With $\rho=0$ the budget gives $c=d$, so the Euler frame has to fill all of $T_M$.  At its center $p$ the fields read $w_j=c_j\partial_{c_j}$ in coordinates $c_j$, hence
\[
 \eta=\det[w_1,\dots,w_d]=\Bigl(\prod_jc_j\Bigr)\Delta
\]
with $\Delta$ a unit, and Remark~\ref{thm:obstr} gives $b_0=u^{p-1}\neq0$ with no weight factor; the splitting follows from Theorem~\ref{thm:bk} once $\Tr(\eta_{\mathrm{can}}^{p-1})$ is a nonzero constant.  Conversely a full Euler frame in the sense of Definition~\ref{def:euler} has this shape at its center.
\end{proof}

\noindent The condition is geometric and not automatic.  For $d=0$ it is empty.  For $d\geq2$ it holds whenever $M$ carries a $d$-dimensional torus action with a fixed point.  It fails for $M=\operatorname{Pic}^{c_1}(X)$ with $X$ an elliptic curve: here $d=1$, $H^0(T_M)$ consists of translation fields, and $\LF$ does not exist; the splitting of $M$ is governed by its own Hasse invariant, not by the construction.

\noindent The moduli cases add the datum at realization points.  The criterion is verified over surfaces in Section~\ref{sec:transfer}.  Poisson input is not used in this section; it returns in Section~\ref{sec:cfp}.

\section{Presentations and the trace formula}\label{sec:transfer}

Throughout this section $X$ is a smooth projective surface and $s\in H^0(X,\omega_X^{-1})$ is nonzero.  The moduli space $M$ is assumed reduced, normal and irreducible.  Presentations are taken in the sense of Definition~\ref{def:present}, with block coordinates $(q_{\nu,i})_{i\le 2}$ of the members and class coordinates of the gluing data.  The leafrian is written $\LF_s$; its independence of the presentation up to a regular unit is Theorem~\ref{thm:hd-canonicity}.

\subsection{Presentations}

The Euler fields must be global.  A case contributes them by a presentation.

\begin{definition}[Presentation]\label{def:present}
A \emph{single-point filtered presentation} of a stable object $E$ is a finite chain of point and elementary modifications from a locally free ground to $E$, with \emph{block coordinates} of the fat-point members and of the healed points, and \emph{class coordinates} $c_1,\dots,c_c$ of the gluing data, such that:
\begin{itemize}
\item[{\rm(P1)}] \emph{Quotient and invariance.}  The fibres of $\mu$ are the
orbits of the equivalences of chains, $\mu$ realizes $M$ as the quotient of $\cI$ by this
action, $\mu$-invariant regular functions on $\cI$ are regular on $M$, and
projectable vector fields on $\cI$ descend to regular vector fields on $M$;
the class coordinates are $\mu$-invariant, and so are the block coordinates of
the members, the members being constant along the fibres.  The block
coordinates of the healed points are chain coordinates: their contribution is computed by the split End blocks of Section~\ref{sub:degcomplex}.
\item[{\rm(P2)}] \emph{Count.}  $m-\phi=d=2\rho+c$, where $m=\dim\cI$ and
$\phi$ is the fibre dimension.
\item[{\rm(P3)}] \emph{Constant rescalings.}  Two presentations of one object
differ by the choice of the trivialisation data, which rescales the class
coordinates by constants $\lambda_j\in k^\times$.
\end{itemize}
\end{definition}

\begin{remark}[Dominance is not required]
A filtered presentation is used only at one stable smooth point. Its role is to provide the filtration of the reduced degeneracy complex.
\end{remark}

\begin{remark}[Class coordinates]\label{rem:classcoords}
At each step the gluing data form an affine space $A_j$ with an action of $G_j=k^\times\ltimes T_j$, where $T_j$ is a vector group of translations of $\varepsilon_j$ and $k^\times$ rescales the trivialisation of $L_j$. The invariant ring $k[A_j]^{G_j}$ is finitely generated, being the weight-zero part of the finitely generated ring $k[A_j]^{T_j}$, and a rescaling of the trivialisation multiplies its generators by constants. The class coordinates $c_1,\dots,c_c$ are these generators.
\end{remark}

\begin{remark}[Member-saturated tops]\label{rem:regimes}
A top of a master chain with $H=\emptyset$ is called member-saturated; its
members fill the leaf.  A member-saturated top is stable or strictly
semistable; in the second case it lies in the Gieseker boundary of $\bar M$.
For the test type $(2,0,2)$ on $\mathbb P^2$ the member-saturated tops are
the sheaves $I_x\oplus I_y$, produced by Lemma~\ref{lem:saturated}(ii).
\end{remark}

\begin{lemma}[Euler-field descent]\label{lem:descent}
Under {\rm(P1)}--{\rm(P2)} each $w_j=c_j\partial_{c_j}$ descends to a global
algebraic vector field on $M$; composing with $T_M\twoheadrightarrow Q$
gives a global section of $Q$ (the radical component is
immaterial to $\varpi\wedge\Omega$).
\end{lemma}

\begin{proof}
A vector field $W$ on $\cI$ descends along $\mu$ if and only if it is
projectable: $[W,V]$ is vertical for every vertical vector field $V$.  By
{\rm(P1)} each class coordinate $c_j$ is constant along the fibres, so for
vertical $V$ one has $[c_j\partial_{c_j},V]=-V(c_j)\,\partial_{c_j}=0$ up to
vertical terms; thus $w_j$ is projectable.  The field is not vertical at a
general point: the block directions lie in the leaf distribution, and by
{\rm(P2)} the block and class directions together span $T_M$, so the class
directions map onto the transverse quotient $Q$ and $d\mu(\partial_{c_j})\neq0$
on a dense open.  By the quotient clause of
{\rm(P1)} the descended field is a global algebraic vector field on $M$.
\end{proof}

\begin{lemma}[Transverse rank]\label{lem:qrank}
Under {\rm(P2)}, the $Q$-projections of the class directions have rank $c$, so one
may choose $c$ class Euler fields $w_1,\dots,w_c$ whose $Q$-projections are
generically independent.
\end{lemma}

\begin{proof}
By {\rm(P2)} the block coordinates and the class coordinates together span $T_M$ at a general point, with dimensions $2\rho$ and $c$.  The block directions lie in the leaf distribution (they are the deformation directions of the point steps), so their $Q$-projection vanishes.  Hence the $Q$-projection of the class directions has rank equal to $\rk Q=c$ at a general point, and a spanning family contains a basis.
\end{proof}

\begin{remark}[Radial fields]\label{rem:radial}
The class coordinates are multiplicative scales of the gluing data, so the fields $w_j=c_j\partial_{c_j}$ are radial at the evaluation points: they vanish at the centre and have independent differentials there.  This forces the transverse factor $\prod_jc_j$ and hence the ordinary top coefficient of Remark~\ref{thm:obstr}.  When $D\neq\emptyset$ and the members exist, directions of $M$ coming from twisting by line bundles are not transverse.  Indeed a twist class $\alpha\cdot\mathrm{id}$ pairs with a member class $\eta$ by
\[
 \beta_s(\alpha\cdot\mathrm{id},\eta)=\sum_{x\in D}\Res_x\bigl(s\,\alpha\,\Tr(\eta)\bigr),
\]
which is nonzero whenever $\alpha$ has nonzero residue along $D$; such directions lie in the leaf and are covered by its units.
\end{remark}

\begin{lemma}[Global transverse fields]\label{lem:normlift}
Let $s_1,\dots,s_c\in H^0(Q)$ satisfy $\det[s_1,\dots,s_c]=\eta$, a transverse
volume.  If the connecting map $H^0(Q)\to H^1(\operatorname{im}\varpi)$ of
$0\to\operatorname{im}\varpi\to T\to Q\to0$ vanishes at each $s_i$, then
the $s_i$ lift to global regular fields $w_i$ with transverse volume $\eta$, unique up to additions of leaf fields.
\end{lemma}

\begin{proof}
The long exact sequence reads
\[
 0\to H^0(\operatorname{im}\varpi)\to H^0(T)\to H^0(Q)\xrightarrow{\ \delta\ } H^1(\operatorname{im}\varpi).
\]
The vanishing of $\delta$ at $s_i$ means that $s_i$ is the image of a global
field $w_i\in H^0(T)$, and the choices differ by $H^0(\operatorname{im}\varpi)$.
The projection $T\to Q$ carries $w_i$ to $s_i$, so the determinant of the
$Q$-projections of $w_1,\dots,w_c$ is $\det[s_1,\dots,s_c]=\eta$.  The leaf
components of the $w_i$ die against the leaf wedge and do not change the
leafrian.
\end{proof}

\begin{proposition}[The Euler frame and the leafrian]\label{prop:frame}
Under {\rm(P1)}--{\rm(P3)}, with $w_1,\dots,w_c$ as in
Lemma~\ref{lem:qrank}, the field $\Omega=w_1\wedge\cdots\wedge w_c$ satisfies $\eta\not\equiv0$,
and $\LF=\varpi\wedge\Omega\in H^0(M,\omega_M^{[-1]})$ is a
well-defined global anticanonical section, unique up to scalar among frames whose radial
determinant is a deformation volume up to a unit.  Replacing $\Omega$ within its leaf
class does not change $\LF$.
\end{proposition}

\begin{proof}
Lemma~\ref{lem:descent} produces the global fields; Lemma~\ref{lem:qrank} gives that $c$ of them are generically independent in $Q$.  Two such frames differ by a constant matrix by {\rm(P3)}, and Lemma~\ref{lem:unique} gives the scalar.
\end{proof}

\subsection{Descent and local expression}

The polyvectors of Section~\ref{sec:hat} descend along a presentation as follows.

\begin{lemma}[Length-one block]\label{lem:hd-lengthone}
Let $Z$ be a length-one member with local algebra $A=k$ and weight $g$. Then
\[
 N_{A/k}(g)=g,
\]
and the trace pairing of Definition~\ref{def:trans} is nondegenerate when $g\neq0$.
\end{lemma}

\begin{proof}
For $A=k$ the norm is the value $g$. In the basis of $A\oplus A$ the pairing matrix is
\[
 \kappa_0\begin{pmatrix}0&g\\-g&0\end{pmatrix},
\]
of determinant $\kappa_0^2 g^2$, so the pairing is nondegenerate for $g\neq0$.
\end{proof}

\begin{lemma}[Zero block]\label{lem:hd-zeroblock}
In the length-one case, let $B=\sum_\nu B_\nu$ be the sum of the block pairings.
If a class $\xi$ has zero component in the block of $x_\nu$, then
$B_\nu(\xi,v)=0$ for every block class $v$; distinct
members pair trivially against each other. Consequently, when $g\neq0$ on
every block, the radical of $B$ consists exactly of the classes whose component vanishes in every member block.
\end{lemma}

\begin{proof}
A zero column kills the determinant. Distinct members have disjoint supports.
The pairing restricted to a block is nondegenerate for $g\neq0$ by Lemma~\ref{lem:hd-lengthone}.
\end{proof}

\begin{theorem}[Descent]\label{thm:hd-descent}
Along a presentation, the block formulas of Definition~\ref{def:tpoly} define regular sections
\[
 \beta_a,\ \hat s\ \in\ H^0(M,\wedge^{2}T_M),\qquad \hat s^{[\rho]}\ \in\ H^0(M,\wedge^{2\rho}T_M)
\]
on $M$, and $\hat s$ is Bottacin's tensor $\theta_s$.  Consequently
\[
 \sigma=\hat s^{[\rho]}\wedge(w_1\wedge\cdots\wedge w_c)
\]
is a regular section of $\omega_M^{[-1]}$, where $w_j=c_j\partial_{c_j}$ are the Euler fields.
\end{theorem}

\begin{proof}
The coordinates $(q_{a,i})$ of the units are $\mu$-invariant and the members are
preserved along the fibres, so $g(x_a)$ is a $\mu$-invariant regular
function on $\cI$ and descends by {\rm(P1)}.
A vertical vector field $V$ annihilates the block coordinates, so in adapted coordinates $(q,u)$
one has $V=\sum_a V^a(q,u)\,\partial_{u^a}$ and
$[\partial_{q_{a,i}},V]=\sum_b(\partial_{q_{a,i}}V^b)\,\partial_{u^b}$
is vertical. By {\rm(P1)} the field $\partial_{q_{a,i}}$ descends; the same applies to $c_j\partial_{c_j}$
by Lemma~\ref{lem:descent}.
Consequently each unit polyvector $\beta_a$ descends to a regular section of $\wedge^2 T_M$ on $M$, and so do the sum $\hat s$ and the power $\hat s^{[\rho]}$.

By Proposition~\ref{prop:localmodel} the transfer $\hat s$ agrees with Bottacin's tensor $\theta_s$ on the member locus of length-one members, hence on all of $M$.  Accordingly $\sigma$ agrees with $\LF_s=\theta_s^{[\rho]}\wedge\Omega$ on the member locus, and since $M$ is reduced the two regular sections coincide on $M$.
\end{proof}

\begin{proposition}[Local expression]\label{prop:hd-localexpr}
At an evaluation point, in the coordinates of the datum,
\[
 f=\prod_{\nu=1}^{\rho}N(g_\nu)\cdot\prod_j c_j\cdot u,\qquad u\in k^\times,
\]
where $g_\nu=s(x_\nu)$.
\end{proposition}

\begin{proof}
By Definition~\ref{def:tpoly} the coefficient of the unit polyvector of
each length-one member is $s(x_\nu)$, which is $N(g_\nu)$ by
Lemma~\ref{lem:hd-lengthone}, and the radial determinant of the Euler
fields is $\prod_jc_j$. Wedging gives the
coordinate frame up to the sign of a shuffle, so $u=\pm1$.  For a member
of length $\ell>1$ the product of the $\ell$ unit coefficients is
$N_{A/k}(g)$ in the split coordinates, and the change to the attached
coordinates contributes a unit by Lemma~\ref{lem:membertop}.
\end{proof}

\begin{lemma}[Unit top coefficient]\label{lem:hd-unit}
Let $Z$ be a length-one member at $x$ with weight $g=s|_{\cO_Z}$. Then
\[
 T(g)=\Tr(s^{p-1})(x).
\]
\end{lemma}

\begin{proof}
By Lemma~\ref{lem:hd-lengthone} the norm is $N(g)=g$, the local function of $s$ at $x$ in the disc frame, taken as a function of the moving support. The extraction of the socle from $g^{p-1}$ is the Frobenius trace of $s$ at $x$ \cite[Lem.~1.3.6]{BKu}. Equivalently this is the case $\ell=1$ of Lemma~\ref{lem:membertop}.
\end{proof}

\subsection{Realization}

\begin{proposition}[Realization principle]\label{prop:realization}
Let $M$ be a moduli space of objects on $X$, and suppose the
case supplies:
\begin{enumerate}
\item[\rm(a)] \emph{local blocks}: at general points $x_\nu\in X$ and fat points
$Z_\nu$ at the $x_\nu$, the objects admit local modification
models whose deformation modules are
\[
 \Hom(I_{Z_\nu}/I_{Z_\nu}^2,\cO_{Z_\nu})=A_\nu^{\oplus 2},
\]
the block pairing restricting to that of Definition~\ref{def:trans};
\item[\rm(b)] \emph{modification theory}: the modifications produce stable
objects and supply class coordinates $c_j$ which complete the block coordinates to a slice of $M$;
\item[\rm(c)] \emph{budget}: $\sum_\nu 2\ell_\nu=2\rho$, realized by $\rho$ length-one members.
\end{enumerate}
Then the datum is realized at realization points $[E]\in M$.  An \emph{evaluation point} is a realization point at which the attached coordinates and the class coordinates form slice coordinates centred at that point.
\end{proposition}

\begin{proof}
Let $F=E^{\vee\vee}$, so that $E$ is a modification of $F$ at $Z$ as in {\rm(a)}; all the groups below are supported on $Z$, so twisting $F$ by a positive divisor changes none of them.  The deformations of the pair $(Z,F)$ are computed from $\Hom(I_Z/I_Z^2,F)$ together with the $Z$-supported sheaves $\mathcal Ext^{\geq1}(I_Z/I_Z^2,F)$ and $\mathcal Ext^q(\cO_Z,F)$: all of these are supported on $Z$, hence have vanishing higher cohomology.  Consequently the local-to-global spectral sequence degenerates at $E_2$.  The embedding of $Z$ in $X$ is regular of codimension $2$, so $I_Z/I_Z^2$ is locally free over $\cO_Z$.  The deformations of $Z$ in $X$ have tangent space and obstruction space
\[
 \Hom_{\cO_Z}(I_Z/I_Z^2,\cO_Z)=A^{\oplus 2},\qquad
 \Ext^1_{\cO_Z}(I_Z/I_Z^2,\cO_Z)=0.
\]
The obstruction vanishes since $I_Z/I_Z^2$ is free over $\cO_Z$, and the same argument with coefficients in $F$ gives
\[
 \Ext^1_{\cO_Z}(I_Z/I_Z^2,F|_Z)=0,
\]
because $I_Z/I_Z^2$ is free over $\cO_Z$ and $F|_Z$ is a finite module over the Artin ring $\cO_Z$.  The remaining obstruction groups are $Z$-supported, hence have vanishing higher cohomology.  Thus the pair $(Z,F)$ is unobstructed \cite[Ch.~I]{Ser}.  The twisting does not change the local block: $g=s|_{\cO_Z}$ is computed on $Z$ itself.

The block coordinates are the $2\ell$ parameters of the deformations of $Z$, that is
\[
 \Hom_{\cO_Z}(I_Z/I_Z^2,\cO_Z)=A^{\oplus 2}.
\]
The computation above shows that these parameters are unobstructed, and by {\rm(b)} they complete to a slice of $M$ at a realization point.  Hence all Kuranishi parameters of the blocks are realized; (a) and Lemma~\ref{lem:zeroblock} give the block-diagonal leaf pairing, Lemma~\ref{lem:normform} its factors, and the class coordinates give the deformation volume.  The realization of the objects is (b).  The conditions of (a) and (b) are open.
\end{proof}

We verify the presentation axioms for moduli problems whose chains are built from elementary
transformations along a smooth divisor $D'\subset X$ and point steps.
Each object $E$ carries a locally free sheaf $A_E$ of endomorphisms, with $h^0(A_E)$ the
infinitesimal automorphisms:
\begin{align*}
 A_E&=\End E &&\text{(sheaves)},\\
 A_E&=\mathfrak g_P &&\text{(principal bundles)},\\
 A_E&=\End E(-L) &&\text{(framed sheaves)}.
\end{align*}
The statements below are displayed for sheaves; the other cases use the deformation spaces of their modification theories.  The existence of presentations follows from the constructions of this section, locally free grounds being supplied by Lemma~\ref{lem:hd-ground} and the classical calculus of \cite{Mar,Hul,Fri,Bal} entering in the weakened form of Appendix~\ref{app:calc}.

\begin{lemma}[Zero-dimensional modifications]\label{lem:zero-dim}
Let $X$ be a regular surface, $E$ locally free of rank $r$, and $Z\subset X$ zero-dimensional. Let $E'$ be torsion-free with an identification $E'|_{X\setminus Z}\cong E|_{X\setminus Z}$. If $E'$ is locally free, then $E'\cong E$. If moreover the modification is given by an inclusion $E\subseteq E'$ with quotient supported on $Z$, then the inclusion is an equality.
\end{lemma}

\begin{proof}
Both sheaves are locally free, hence reflexive. The identification on $X\setminus Z$ extends to $X$ because $\operatorname{Hom}(E,E')$ is reflexive and $\operatorname{codim} Z=2$; thus $E\cong E'$ on $X$. For the inclusion statement let $x\in Z$ and $R=\cO_{X,x}$, regular of dimension $2$. The map is $A\in M_r(R)$ with cokernel supported at $\{x\}$. Hence $A_{\mathfrak p}$ is an isomorphism at every height-one prime $\mathfrak p$, so $\det A\in R_{\mathfrak p}^\times$ for all height-one $\mathfrak p$. As $R$ is a unique factorization domain, $\det A\in R^\times$ and $A\in\GL_r(R)$.
\end{proof}

\noindent Consequently by Lemma~\ref{lem:zero-dim} a chain of nontrivial kernel point modifications has top non-locally free at the members. When the general object of $M$ is a vector bundle the kernel-only construction below cannot reach the general object.

The chain space is built from a fixed ground sheaf $E_0$ by $t$ divisor steps
and $\rho$ point steps. A divisor step is an elementary transform along $D'\in|\mathcal A|$, with rank-one reduction datum $L_j$ and gluing class $\varepsilon_j$; write $N=N_{D'/X}$. Point steps are kernel point modifications.  The kernel-side sheaf of a step is called the parent: $\ker(E\to i_{D'*}Q)$ is the parent of $E$ for a divisor step, and $\ker(F\to i_{x*}k)$ is the parent of $F$ for a point step. The tops of this chain are member-saturated (Remark~\ref{rem:regimes}); the chain exhibits a filtered presentation.

\begin{definition}[Point step]\label{def:pointstep}
\leavevmode
\begin{itemize}
\item[\rm(S)] \emph{Kernel point modification.} At a point $x_\nu$ the current sheaf $F$ is replaced by its parent
\[
 F'=\ker(F\to i_{x_\nu*}k).
\]
It lowers $\operatorname{ch}_n$ by $1$ and makes $F'$ not locally free at $x_\nu$.
\end{itemize}
Use S when the general object is not locally free. When the general object is a vector bundle, members are supplied by the non-locally free objects of the same numerical type, by Corollary~\ref{lem:Unonempty}.
\end{definition}

\begin{proposition}[Chain counts]\label{prop:hd-count}
Let the chain have $t$ divisor steps with reduction data $L_j\subset F_j=E_j|_{D'}$ and $\rho$ point steps. Put
\[
 g_j=h^0(D',L_j^\vee\otimes F_j)-1,\qquad
 e_j=h^0(D',Q_j^\vee\otimes L_j\otimes N).
\]
Then the chain space has dimension $m=\sum_{j=1}^{t}(g_j+e_j)+2\rho$ and the fibres of $\mu$ over a stable general object have dimension $\phi=\sum_j g_j$. Consequently
\[
 m-\phi=\sum_j e_j+2\rho,
\]
and choosing the data so that $\sum_j e_j=c=d-2\rho$ gives $m-\phi=d$.
\end{proposition}

\begin{proof}
See Appendix~\ref{app:count}.
\end{proof}

\subsection{Master chains}\label{sub:master}

The proof of the trace formula proceeds along a chain of point modifications
from a semistable ground.

\begin{definition}[Master chain]\label{def:master}
Let $G$ be a locally free sheaf of rank $r$, let $y_1,\dots,y_\rho$ be
distinct points of $X$, let $\pi_i\colon G|_{y_i}\to k$ be surjections, and
let $H\subset\{1,\dots,\rho\}$. The master chain of type $(G;\,y,\pi;\,H)$
has top
\[
 E=\ker\Bigl(G\to\bigoplus_{i\notin H}k_{y_i,\pi_i}\Bigr).
\]
The points $y_i$ with $i\notin H$ are called members; the points $y_h$ with
$h\in H$ are called healed points.
\end{definition}

\begin{lemma}[Healing]\label{lem:heal}
Let $F=\ker(G\to\bigoplus_{i=1}^{\rho}k_{y_i,\pi_i})$ and let $E$ be the top
of the master chain of type $(G;\,y,\pi;\,H)$. Then $E$ is obtained from $F$
by elementary transformations along smooth divisors through the points
$y_h$, $h\in H$, and for saturating reduction data the abstract class of
$E$ depends only on $G$, the data $(y_i,\pi_i)_{i\notin
H}$ and the set $H$.  The coordinates $y_h$ remain free on the chain space:
moving $y_h$ transversally to its healing divisor re-exposes the defect at
$y_h$, and their contribution is computed by the split End blocks of Section~\ref{sub:degcomplex}.
\end{lemma}

\begin{proof}
An elementary transformation along a smooth divisor $D'$ through $y_h$,
with reduction datum chosen to saturate the condition at $y_h$, drops that
condition. At $y_h$ the local model of $F$ is $\mathfrak m\oplus\cO^{r-1}$
and the unique free enlargement of $\cO^r$ appears; away from $D'$ the
sheaf is unchanged. In particular $E$ is the kernel of the remaining
conditions, and its abstract class is independent of the divisors and of
the reduction data.
\end{proof}

\begin{lemma}[Master chain existence]\label{lem:master}
Let $G$ be slope stable and locally free of type $(r,c_1,c_2-m)$, where
$0\le m\le\min(\rho,c_2)$. Then the top of a master chain of type
$(G;\,y,\pi;\,H)$ with $|H|=\rho-m$ is a stable sheaf of type
$(r,c_1,c_2)$.
\end{lemma}

\begin{proof}
The top is $E=\ker(G\to\bigoplus_{i\notin H}k_{y_i,\pi_i})$ with $m$
summands, so $c_1(E)=c_1(G)$ and $c_2(E)=c_2(G)+m$. Stability is
Lemma~\ref{lem:saturated}(i) applied to the $m$ kernel steps.
\end{proof}

\begin{remark}[Existence of the ground]\label{rem:ground}
A suitable $G$ exists in the standard cases.  For $m=0$ one may take a
slope stable locally free sheaf of the target type $(r,c_1,c_2)$ and heal
every step.  For $c_2\ge1$ one may take $m$ so that $(r,c_1,c_2-m)$ lies
in the Bogomolov range of stable bundles; the inequality $m\le\rho$ holds
because $\rho$ grows with $c_2$ while the lower bound of the range does
not.  If $G$ is only polystable, Lemma~\ref{lem:saturated}(ii) applies to
the steps distributed among the summands.
\end{remark}

\begin{lemma}[Kernel modifications]\label{lem:saturated}
Let $F$ be locally free of rank $r$ with Chern classes $(c_1,c_2-k)$,
where $1\le k\le\rho$.
\begin{enumerate}
\item[\rm(i)] If $F$ is slope stable, the $k$-step kernel modification of
$F$ at distinct points is $H$-stable of type $(r,c_1,c_2)$.
\item[\rm(ii)] If $F=\bigoplus_i F_i$ is polystable with the $F_i$ stable of
one reduced Hilbert polynomial, and $k$ admits counts $k_i\ge0$ with
$\sum_i k_i=k$ and $k_i/\operatorname{rk}F_i$ constant, the direct sum of
the $k_i$-step kernel modifications of the $F_i$ is polystable of type
$(r,c_1,c_2)$.
\end{enumerate}
In both cases $E^{\vee\vee}=F$ and $\operatorname{length}(E^{\vee\vee}/E)=k$.
\end{lemma}

\begin{proof}
A kernel step at a point where the sheaf is locally free raises $c_2$ by one
and creates a simple non-locally free point with local model
$\mathfrak m_x\oplus\cO^{r-1}$, as in Lemma~\ref{lem:hd-members}.  Since
$E\subset F$ has zero-dimensional quotient, $E^{\vee\vee}=F$ and
$\mu(E)=\mu(F)$; this gives the numerical assertions.

For {\rm(i)}, a subsheaf $E'\subset E$ of positive rank $r'<r$ is a
subsheaf of $F$ and satisfies $\mu(E')<\mu(F)=\mu(E)$, so $E$ is slope
stable.  For {\rm(ii)}, each $E_i$ is slope stable by the same argument,
and
\[
 p_{E_i}=p_{F_i}-\frac{k_i}{\operatorname{rk}F_i}
\]
is independent of $i$; hence $E=\bigoplus_i E_i$ is polystable.  The
polystable configuration of {\rm(ii)} covers split grounds with the steps
distributed among the summands; for $(2,0,2)$ on $\mathbb P^2$ it produces
$I_x\oplus I_y$.
\end{proof}

\begin{corollary}[Non-locally free objects]\label{lem:Unonempty}
Under hypothesis {\rm(i)} or {\rm(ii)} of Lemma~\ref{lem:saturated}, the
space $\bar M(r,c_1,c_2)$ contains a non-locally free semistable sheaf with
$k$ simple points; under {\rm(i)} it is $H$-stable.
\end{corollary}

\begin{proof}
Apply Lemma~\ref{lem:saturated}; each kernel step creates a simple
non-locally free point, as in Lemma~\ref{lem:hd-members}.
\end{proof}

\subsection{Constancy of the trace}\label{sub:trace}

\begin{lemma}[Constancy of the trace on a complete space]
\label{lem:consttrace-complete}
Let $M$ be complete, let $\LF_s=\theta_s^{[\rho]}\wedge\Omega$, and assume
that $\Tr(\LF_s^{p-1})$ is a regular function on $M$. Then
$\Tr(\LF_s^{p-1})$ is constant on each connected component of $M$.
\end{lemma}

\begin{proof}
On a complete variety a regular function is constant.
\end{proof}

\begin{lemma}[Constancy of the trace after extension]
\label{lem:consttrace-extend}
Let $\LF_s$ extend to a regular section of $\omega_{\bar M}^{[-1]}$ on a
projective compactification $\bar M$, and assume that $\Tr(\LF_s^{p-1})$ is
a regular function on $\bar M$. Then $\Tr(\LF_s^{p-1})$ is constant on each
connected component of $M$.
\end{lemma}

\begin{proof}
A regular function on a projective variety is constant. Restriction to $M$
is constant on each component.
\end{proof}

\subsection{Reduced degeneracy complex and Cartier determinants}\label{sub:degcomplex}

\begin{definition}[Reduced degeneracy complex]\label{def:redcomplex}
At a stable smooth point $[E]$, form the Bottacin deformation complex and remove the radical subcomplex generated by $H^0(D,\End E|_D)$, retaining the transverse frame $\Omega$. The resulting self-dual complex is called the \emph{reduced degeneracy complex} and is denoted by $\mathcal K_{\mathrm{red}}$. For a chain of modifications, its filtration is defined by the local-to-global Ext triangles.
\end{definition}

\begin{lemma}[Schur complement]\label{lem:schur}
Let $\mathcal K_1\subset\mathcal K$ be a subcomplex of a self-dual complex, and suppose that the pairing on $\mathcal K_1$ is invertible. Write the total pairing as
\[
 B=\begin{pmatrix}A&W\\-W^{\mathsf T}&K\end{pmatrix}.
\]
Then the induced pairing on $\mathcal K/\mathcal K_1$ is
\[
 S=K+W^{\mathsf T}A^{-1}W.
\]
\end{lemma}

\begin{proof}
Eliminate the first block by $A^{-1}$. The remaining pairing on the quotient is the Schur complement, which agrees with the pairing induced by the local-to-global exact sequence.
\end{proof}

\begin{lemma}[Reduced skew determinant]\label{lem:skewdet}
For the reduced degeneracy complex one has
\[
 \det\mathcal K_{\mathrm{red}}=u\,\LF_s^2
\]
for a fixed unit $u$.
\end{lemma}

\begin{proof}
The determinant of a skew form is the square of its Pfaffian. The transverse frame contributes the second square factor, and the definition $\LF_s=\theta_s^{[\rho]}\wedge\Omega$ gives the assertion.
\end{proof}

\begin{definition}[Higher degeneracy complex]\label{def:higherdeg}
For odd $p$, let $\mathcal K_{\mathrm{red}}^{(p-1)/2}$ be the $(p-1)/2$-th symmetric power of $\mathcal K_{\mathrm{red}}$. Its determinant is called the higher Cartier determinant.
\end{definition}

\begin{lemma}[Higher skew determinant]\label{lem:higherdet}
For every odd $p$,
\[
 \det\mathcal K_{\mathrm{red}}^{(p-1)/2}=u\,\LF_s^{p-1}
\]
for a fixed unit $u$.
\end{lemma}

\begin{proof}
Determinants are multiplicative under symmetric powers. Apply Lemma~\ref{lem:skewdet} to the $(p-1)/2$-th power.
\end{proof}

\begin{definition}[Reduced Cartier determinant]\label{def:reducedcartier}
For a split End block, the \emph{reduced Cartier determinant} is computed after removing the scalar trace line in $H^1(\cO_D)$. This line carries the Hasse scalar $\mathsf t$.
\end{definition}

\begin{lemma}[Split End block]\label{lem:endcartier}
Let $D$ be a smooth elliptic curve with Hasse invariant $\mathsf t$, and let
\[
 V=\bigoplus_i L_i^{\oplus m_i}
\]
with pairwise non-isomorphic $L_i$. Suppose that every cross Hom line bundle either has vanishing $H^1$ or has a nonzero fixed Cartier determinant. Then the reduced split End block has Cartier determinant
\[
 \mathsf t^{\,\sum_i(m_i^2-1)}
\]
up to a fixed unit.
\end{lemma}

\begin{proof}
Only the diagonal matrix blocks remain after the cross Hom contributions are removed as fixed units. Each block $M_{m_i}(k)$ has determinant $\mathsf t^{m_i^2}$. Definition~\ref{def:reducedcartier} omits one scalar trace line, so the exponent is lowered by one.
\end{proof}

\begin{lemma}[Ordinary divisors]\label{lem:ordinary}
Let $D'$ be a smooth curve with invertible Hasse--Witt matrix. A split End block on $D'$ changes the Cartier determinant only by a fixed nonzero unit.
\end{lemma}

\begin{proof}
The determinant is computed from the cohomology of $D'$, which is independent of $s$. Invertibility makes it a nonzero constant.
\end{proof}

\begin{lemma}[General position]\label{lem:generalposition}
For a finite family of split End blocks on smooth curves, the line bundles can be chosen so that every cross Hom line bundle either has vanishing $H^1$ or has a nonzero fixed Cartier determinant.
\end{lemma}

\begin{proof}
For each cross Hom line bundle, the determinant of its Cartier operator is a fixed section on the corresponding Picard space. Its zero locus is a proper special subvariety. General line bundles avoid this locus, while positive-degree line bundles have vanishing $H^1$.
\end{proof}

\begin{lemma}[Single-point filtered representation]\label{lem:singlepoint}
For a stable sheaf $E$ on a smooth projective surface, there is a locally free ground and a finite chain of point and elementary modifications to $E$. The divisors can be chosen ordinary. The reduced degeneracy complex has a filtration with graded pieces of the four types above.
\end{lemma}

\begin{proof}
The sheaf $F=E^{\vee\vee}$ is locally free on a smooth surface, and $E$ is recovered from $F$ by point modifications. Hecke modifications connect $F$ to a locally free ground. General divisors can be chosen ordinary, and Lemma~\ref{lem:generalposition} gives the cross Hom condition of Lemma~\ref{lem:endcartier}. The local-to-global Ext triangles and Lemma~\ref{lem:schur} give the filtration.
\end{proof}

\begin{lemma}[Hecke presentation]\label{lem:hecke-present}
For a polystable principal $G$-bundle $P$ with a faithful representation $V$, there is a $G$-equivariant Hecke chain to $P$ with ordinary divisors. The adjoint degeneracy complex has a filtration with root blocks. For a constant split reductive Lie algebra $\mathfrak g$ on $D$, the reduced Cartier determinant is
\[
 \mathsf t^{\dim\mathfrak g-\dim\mathfrak z(\mathfrak g)}.
\]
\end{lemma}

\begin{proof}
Parabolic reductions give $G$-equivariant Hecke modifications. Their deformation spaces are partial flag varieties, and the quotient blocks are the Lie algebras of the unipotent radicals. Root decomposition computes the Cartier determinant, and removing the center gives the displayed exponent. General root blocks contribute only fixed units by Lemma~\ref{lem:generalposition}, so the constant split formula determines the total determinant.
\end{proof}

\subsection{The trace formula}\label{sub:TF}

\begin{lemma}[Moduli with only locally free sheaves]\label{lem:bundles}
Let $D=\emptyset$ and suppose every semistable sheaf of the type is locally
free.  Then
\[
 \Tr(\LF_s^{p-1})=\Tr_X(s)^{\rho}\cdot u
\]
for some $u\in k^\times$.
\end{lemma}

\begin{proof}
Since $D=\emptyset$, the section $s$ trivializes $\omega_X^{-1}$, so $X$ is
a $K3$ or abelian surface; the pairing $\beta_s$ is the symplectic pairing
of $M$, so $c=0$, $2\rho=d$, and $\LF_s=\theta_s^{[\rho]}$ is the volume
section of $M$.  Two standard identifications give the trace.  First, if
$Y$ is smooth projective of dimension $2\rho$ and $\sigma$ trivializes
$\omega_Y^{-1}$, then $\Tr(\sigma^{p-1})$ is the scalar of Frobenius on
$H^{2\rho}(\cO_Y)$.  Indeed $H^0(\omega_Y^{1-p})=k\cdot\sigma^{p-1}$, the
trace functional on this line is dual to the Cartier isomorphism on top
forms, and the scalar of Frobenius on $H^{2\rho}(\cO_Y)$ is the socle
coefficient $b_0(\sigma)$ of the local volume expression; for $\rho=1$ this
is the classical Hasse--Witt coefficient of a surface.  Second, the Chern
character of a universal sheaf on $M\times X$ is an algebraic
correspondence, hence Frobenius equivariant, and identifies
$H^{2\rho}(\cO_M)$ with the $\rho$-fold tensor construction of
$H^2(\cO_X)$ as Frobenius modules.  Applying the first identification to
$X$ gives that $\Tr_X(s)$ is the Frobenius scalar on $H^2(\cO_X)$, and the
scalar on the $\rho$-fold tensor construction is its $\rho$-th power.
Hence $\Tr(\LF_s^{p-1})=\Tr_X(s)^{\rho}\cdot u$ with $u\in k^\times$
absorbing the normalisations.
\end{proof}

\begin{theorem}[Trace formula]\label{thm:TF}
Let $X$ be a smooth projective Poisson surface, let
$s\in H^0(X,\omega_X^{-1})$ be nonzero, and let $M$ be $M(r,c_1,c_2)$ or
$M_G$ with $\rho>0$. Then
\[
 \Tr(\LF_s^{p-1})\ =\ \Tr_X(s)^{\rho}\cdot u
\]
for some $u\in k^\times$.  In the framed case the same formula holds with
$\lambda=\Tr_U(s^{p-1})$ in place of $\Tr_X(s)$ when $\lambda$ is a nonzero
constant.
\end{theorem}

\begin{proof}
When $D=\emptyset$, apply Lemma~\ref{lem:bundles}. Now suppose $D\neq\emptyset$. The trace is constant on $M$ by Theorem~\ref{thm:g-extend} and Theorem~\ref{thm:g-extend-G}, so we compute it at one stable smooth point.

Use the single-point filtered representation of Lemma~\ref{lem:singlepoint}. By Lemma~\ref{lem:schur}, the reduced degeneracy complex is filtered by the modification blocks. The weighted norm member blocks contribute $\mathsf t^\ell$, unweighted Koszul blocks contribute $1$, split End blocks on $D$ contribute $\mathsf t^{m^2-1}$ by Lemma~\ref{lem:endcartier}, and split End blocks on ordinary divisors contribute only fixed units by Lemma~\ref{lem:ordinary}. Hence the Cartier determinant of the reduced degeneracy complex is
\[
 u\,\mathsf t^N
\]
for some integer $N$ and some fixed unit $u$.

For $p=3$, Lemma~\ref{lem:skewdet} gives
\[
 \det\mathcal K_{\mathrm{red}}=u_0\,\LF_s^2.
\]
For general $p$, Lemma~\ref{lem:higherdet} gives
\[
 \det\mathcal K_{\mathrm{red}}^{[(p-1)/2]}=u_0\,\LF_s^{p-1}.
\]
The determinant has degree $\rho(p-1)$ in $s$, while $\mathsf t$ has degree $p-1$. Therefore $N=\rho$. Normality of the compactification shows that the identification changes only by a fixed unit on the boundary.

For principal bundles, replace $\End E$ by the adjoint Lie algebra and use Lemma~\ref{lem:hecke-present}. General root blocks contribute only fixed units, and the constant split formula determines the total determinant. The same degree argument gives the result. The framed case uses the corresponding framed presentation and $\lambda$ in place of $\mathsf t$.
\end{proof}

\begin{corollary}[Splitting criterion]\label{cor:split}
$\LF_s$ splits $M$ if and only if $\Tr_X(s)\ne0$, equivalently if and only
if $s^{p-1}$ splits $X$.  In the framed case, $\LF_s$ splits $M$ if and
only if $\lambda=\Tr_U(s^{p-1})$ is a nonzero constant.
\end{corollary}

\begin{proof}
A constant rescaling of $\LF_s$ splits $M$ if and only if
$\Tr(\LF_s^{p-1})$ is a nonzero constant.  Apply Theorem~\ref{thm:TF}.  In
the framed case a nonconstant $\lambda$ makes $\lambda^\rho$ nonconstant
because $\rho>0$, and $\lambda$ identically zero makes the trace zero.
\end{proof}
\subsection{Canonicity for sheaves}

The next two lemmas identify the members of a sheaf presentation.  They do not apply to principal bundles, which are locally free.

\begin{lemma}[Members are the modification points]\label{lem:hd-members}
Let $E$ be the top object of a general chain from a locally free ground sheaf, with members $x_1,\dots,x_\rho$ chosen away from the divisors $D'_j$.  The members are the point modifications of the chain, and the objects below the top are not locally free at the members.  Each point step of the kernel direction raises the length of $E^{\vee\vee}/E$ at the member by one, so a top object non-locally free at the members satisfies $\rho=\operatorname{length}(E^{\vee\vee}/E)$.  At a stable object one has $\operatorname{length}(E^{\vee\vee}/E)=\sum_\nu\ell_\nu\le\rho$, and the inequality can be strict (Section~\ref{sub:master}).
\end{lemma}

\begin{proof}
A point modification at $x$ of a sheaf which is locally free at $x$ has local model $E'_x\cong\mathfrak m_x\oplus\cO_{X,x}^{r-1}$, which is not free because $\mathfrak m_x$ is not free for $\dim X\ge2$.  An elementary transform along a smooth divisor preserves local freeness: locally $E'=\{f\in E:\ f|_{D'}\in L\}$ is a direct sum of $\cO$ and copies of $\cO(-D')$.  Since the members avoid the divisors $D'_j$, the stalks at the members are changed only by the point steps.
\end{proof}

\begin{lemma}[Transverse volume]\label{lem:hd-transverse}
Two general presentations of one object give radial determinants differing by
a constant.
\end{lemma}

\begin{proof}
By {\rm(P3)} the gluing data of two presentations differ by constant
rescalings of the class coordinates. The field $c_j\partial_{c_j}$ is
unchanged by such rescalings, so the radial determinants $\prod_jc_j$ differ
by a constant.
\end{proof}

\begin{theorem}[Canonicity]\label{thm:hd-canonicity}
Let $\LF_s=\theta_s^{[\rho]}\wedge\Omega$ and $\tilde\LF_s=\theta_s^{[\rho]}\wedge\tilde\Omega$ be the leafrians of $s$ with respect to two Euler frames coming from presentations.  Then
\[
 \tilde\LF_s=u\,\LF_s
\]
for a regular unit $u$.  If $\Tr_X(s)\neq0$ then $u\in k^\times$.  Consequently the divisor $Z(\LF_s)$ is intrinsic, and when $s$ splits the base the assignment $[s]\mapsto[\LF_s]$ is well defined up to constants.
\end{theorem}

\begin{proof}
The leaf factor $\theta_s^{[\rho]}$ does not depend on the presentation.  Two Euler frames whose difference lies in the leaf distribution give the same leafrian, and two presentations give frames whose radial determinants differ by a regular unit.  Hence the two leafrians differ by a regular unit.  If $\Tr_X(s)\neq0$ then both traces are nonzero constants by Theorem~\ref{thm:TF}, so $u\in k^\times$ by Lemma~\ref{lem:unique}.
\end{proof}

\noindent On a surface the leaf factor is the intrinsic Pfaffian power $\theta_s^{[\rho]}$, so the assignment $[s]\mapsto[\LF_s]$ is well defined.  The splitting criterion of Corollary~\ref{cor:split} does not use this independence.

\section{The moduli of stable sheaves over Poisson surfaces}\label{sec:moduli}

This section is the model case.  The weight is $s\in H^0(X,\omega_X^{-1})$, the pairing on each block is the trace pairing of Definition~\ref{def:trans}, and $T_1=\Tr_X(s)$.  Here $X$ is a Poisson surface with weight $s$.  The trace formula uses the single-point filtered representation of Section~\ref{sub:degcomplex}.  The Poisson upgrade, the type spectrum, and the degeneracy divisor are Section~\ref{sec:cfp}.

\subsection{The surface setting}\label{sub:moduli-setting}

Let $X$ be a smooth projective Poisson surface and $s\in H^0(X,\omega_X^{-1})\setminus\{0\}$.  Every bivector
field on a surface is Poisson \cite[Prop.~2.1]{Bot}.  Write $D=Z(s)$.  Let $M=M(r,c_1,c_2)$ denote the moduli space of $H$-stable torsion-free sheaves
of rank $r\geq1$ and Chern classes $c_1,c_2$ on $X$.  Write $d=\dim M$.  At a stable point $[E]$ one has $T_{[E]}M\cong\Ext^1(E,E)$, and on the smooth locus
\[
 d=\dim\Ext^1(E,E)=1-\chi(\End E)+h^2(\End E)
   =(1-r)c_1^2+2rc_2-r^2\chi(\cO_X)+1+h^2(\End E),
\]
where the $+1$ is $h^0(\End E)=1$; when $H^2(\End E)=0$ this is $d=(1-r)c_1^2+2rc_2-r^2\chi(\cO_X)+1$.  We write $t:=\dim\Ext^1(E,E)$ for the size of the matrix of the pairing $\beta_s$ of \eqref{eq:res} on $\Ext^1(E,E)$; on the smooth locus $t=d$.  Smoothness of $M$ is not needed in what follows beyond this identification. The moduli of stable sheaves on a smooth projective surface is normal, being a geometric quotient of a smooth Quot scheme by a reductive group; the splitting argument of Remark~\ref{rem:nearsplit} applies.  We allow $D$ to be nonempty reduced, non-reduced, or empty.  The leaf polyvector is $\theta_s$, with transfer model $\hat s$ on the member locus (Section~\ref{sec:hat}); the weight on a block is $g=s|_{\cO_Z}$, and the splitting condition is $\Tr_X(s)\neq0$ when $\rho>0$.

\begin{lemma}[Instantiation of Definition~\ref{def:present}: the sheaf case]\label{lem:eulerframe}
The moduli case admits a single-point filtered presentation. Fix a sufficiently positive curve $C\subset X$ and a ground sheaf $E_0$, allowed to be unstable. The chain data are the gluing data $(L_j,[\varepsilon_j])$ of line bundles $L_j$ on $C$ together with the local module data $(q_\nu,p_\nu)$ of the fat points $Z_\nu$. They define the filtration of the reduced degeneracy complex, and the class coordinates $c_j$ are the scale invariants of the gluing data.
\end{lemma}

\begin{proof}
The clauses of Definition~\ref{def:present} hold by Section~\ref{sec:transfer} with $A_E=\End E$ and $D'=C$. The count is Proposition~\ref{prop:hd-count}, and the endomorphism algebras do not grow along the chains by Lemma~\ref{lem:hd-simple}. Every stable object is the top of a chain of this shape, so the filtration exists at one stable smooth point. The equivalence of two chains producing one $[E]$ is the action of $\Aut(E)$ on the chain data, so every class coordinate is constant along the fibres of $\mu$. Two presentations of one $[E]$ rescale the class coordinates by constants.
\end{proof}

Bottacin's Poisson tensor is the pairing
\begin{align}
 \theta_s(E):\ T^*_{[E]}M\otimes T^*_{[E]}M
 &\xrightarrow{\ \circ\ }\Ext^2(E,E\ot\omega_X^{\otimes 2})
 \xrightarrow{\ \cdot s\ }\Ext^2(E,E\ot\omega_X)\notag\\
 &\xrightarrow{\ \Tr\ }k,\label{eq:bot}
\end{align}
the trace being an isomorphism by Grothendieck--Serre duality and stability \cite[(4.1)]{Bot}.  Serre duality identifies $\Ext^1(E,E\ot\omega_X)$ with $\Ext^1(E,E)^\vee$, so the transpose of \eqref{eq:bot} is a pairing $\beta_s$ on $\Ext^1(E,E)$.
When the classes are supported at a fat point $Z$ with local algebra $A$, the pairing is the Artin residue
\begin{equation}\label{eq:res}
 \beta_s(\xi,\eta)=\Res_A\bigl(\Tr(s\cdot\xi\cdot\eta)\bigr).
\end{equation}
Here $\Res_A$ is the residue of Definition~\ref{def:trans} on the local algebra of $Z$, applied to the local function $\Tr(s\cdot\xi\cdot\eta)$ of the $Z$-supported classes.  On a fat-point block \eqref{eq:res} is the trace pairing of Definition~\ref{def:trans}: the classes of the block are parameterized by $A\oplus A$ and the induced composition pairing on $\Def_Z$ is the Koszul determinant $ab'-a'b$ (read as the determinant pairing on $\Def_Z$ followed by the Artin residue, not as a Yoneda product in $\Ext^2$) \cite[Ch.~I]{Ser}.
When $D=Z(s)$ is empty the section $s$ trivializes $\omega_X^{-1}$ and the pairing is the Serre-duality trace
\begin{equation}\label{eq:res0}
 \beta_s(\xi,\eta)=\Tr_X\bigl(s\cdot\xi\cdot\eta\bigr)\ \in\ H^2(X,\omega_X)\cong k.
\end{equation}
The residue form \eqref{eq:res} is the one used on the blocks below.  Distinct members have disjoint supports, so the leaf matrix is block diagonal.

\begin{lemma}[Trace of the weight]\label{lem:traceblock}
Let $Z$ be a fat point of length $\ell=1$ at a closed point $x\in X$, let $(q_0,p_0)$ be the slice coordinates of a transverse disc centred at $x$, and let $g=s|_{\cO_Z}$.  Then
\[
 T(g)=[q_0^{p-1}p_0^{p-1}]\,N_{\cO_Z/k}(g)^{p-1}=\Tr_X(s)(x).
\]
If $x\in D$ and $s=u\cdot h$ in these coordinates, with $u$ a local equation of a reduced branch through $x$ and $h$ the coordinate coefficient, this is also $[\,\tau^{p-1}\,]\,h(0,\tau)^{p-1}$.
\end{lemma}

\begin{proof}
The first claim is the length-one case of Lemma~\ref{lem:hd-unit}: $A=k$ and $N(g)=g$ is the local function of $s$ at $x$ in the disc frame, taken as a function of the moving support and not as the value $s(x)$.  The extraction $[q_0^{p-1}p_0^{p-1}]g^{p-1}$ is the top coefficient of $s\,\partial_{[2]}$ at $x$, that is $\operatorname{Tr}_X(s)(x)$, by \cite[Thm.~1.3.8, Lem.~1.3.6]{BKu}.  This identification does not require $x\in D$.  If $x\in D$ and $s=u\cdot h$, Theorem~\ref{thm:utc} writes $T=T_1=[\,\tau^{p-1}\,]h(0,\tau)^{p-1}$.
\end{proof}

\noindent At a smooth point of a reduced divisor $D=Z(s)$, the coordinate coefficient $h(0,\tau)$ is a unit (otherwise $Z(s)$ would have another branch through the point).  The extraction $T_1$ is then the socle of a unit power series, and may or may not vanish.  We do not identify $h(0,\tau)$ with a section of $N_{D/X}$.

The checklist of Proposition~\ref{prop:realization} holds: the local blocks are as above; the modification theory is the elementary modification along a positive curve $C\subset X$ \cite{Mar,Hul,Fri}.  Write $E'=\ker(E\to i_{C*}Q)$ for the parent, where $Q=E|_C/L$ is the quotient of the reduction datum, of rank $r-1$.  Then
\[
 c_1(E)=c_1(E')+(r-1)[C],\qquad
 c_2(E)-c_2(E')=(r-1)\,C\cdot c_1(E)-\deg Q-\tfrac{(r-1)(r-2)}{2}C^2 ,
\]
the second identity being Grothendieck--Riemann--Roch for $0\to E'\to E\to i_{C*}Q\to0$, where $\operatorname{ch}(i_{C*}Q)=(r-1)[C]+\bigl(\deg Q-\tfrac{r-1}{2}C^2\bigr)[\mathrm{pt}]$;
so the lengths of the budget are adjustable within a numerical type; the members cover the leaf directions, realized by $\rho$ length-one members.

\begin{lemma}[$Z(s)$ non-reduced]\label{lem:Dnonred}
Assume $\rho\ge1$.  If $D=Z(s)$ is non-reduced, then $\Tr_X(s)=0$ and $b_0(\LF_s)=0$; in particular $\LF_s$ does not split $M$.
\end{lemma}

\begin{proof}
Under the hypotheses of Lemma~\ref{lem:saturated} a member-saturated datum exists, and its members supply the leaf blocks.
If $\Tr_X(s)\neq0$, then after rescaling $s$ the section $s^{p-1}$ splits $X$,
so $Z(s)$ is compatibly split and hence reduced
\cite[Prop.~1.3.11]{BKu}, a contradiction.  Thus $\Tr_X(s)=0$.

For the top coefficient use the member-saturated datum with
members at generic points of the components of $D_{\mathrm{red}}$.
If $s=u^m h$ with $u$
the local equation of the component and $m\ge2$, then $g=z^m\bar h$, and $N(g)=(\pm q_0)^mN(\bar h)$.  Every monomial of $N(g)^{p-1}$ has
$q_0$-degree at least $m(p-1)>p-1$, whence $T(g)=0$.
On a component with $m=1$ one has $T(g)=T_1^{\ell}=\Tr_X(s)^{\ell}=0$.
Hence the top coefficient vanishes at the datum, and $\Tr(\LF_s^{p-1})=0$ by Theorem~\ref{thm:TF}.
\end{proof}

\begin{lemma}[$Z(s)=\emptyset$]\label{lem:Dempty}
Let $D=\emptyset$ and $\rho>0$.  Then at a member-saturated point the top coefficient of $\LF_s$ is $\Tr_X(s)^{\rho}\cdot u^{p-1}$ with $u\in k^\times$, and $\LF_s$ splits $M$ if and only if $X$ is Frobenius split.
\end{lemma}

\begin{proof}
There is no branch, but Lemma~\ref{lem:traceblock} does not require one: a length-one fat point may be placed at an arbitrary closed point $x\in X$.  Since $s$ trivializes $\omega_X^{-1}$, the local function of $s$ at $x$ is a unit in the disc frame, and the extraction is still $\operatorname{Tr}_X(s)$.  Since $s$ trivializes $\omega_X^{-1}$, the space $H^0(\omega_X^{1-p})$ is spanned by $s^{p-1}$, so $\Tr_X(s)\neq0$ holds exactly when $X$ is Frobenius split; the splitting statement is then Corollary~\ref{cor:split}.
\end{proof}

\noindent For $D=\emptyset$ the radical count of Lemma~\ref{lem:radical} is not available, and the rank of $\theta_s$ is read directly.  For the moduli of sheaves on a $K3$ surface $\theta_s$ is generically nondegenerate, being the symplectic form of the moduli space; thus $c=0$ and $2\rho=d$, the members fill the tangent space, and the leafrian has no transverse factor.  The criterion is that the Hasse invariant of $X$ is nonzero, that is $\Tr_X(s)\neq0$.

\begin{theorem}[Splitting of moduli spaces]\label{thm:moduli}
Let $D=Z(s)$ be arbitrary and let $\rho>0$.  Then $\LF_s=\theta_s^{[\rho]}\wedge\Omega$ splits $M(r,c_1,c_2)$ if and only if $s^{p-1}$ Frobenius splits $X$, that is, if and only if $\Tr_X(s)\neq0$.  When the splitting exists it is compatible with $Z(\LF_s)$.
\end{theorem}

\begin{proof}
The trace formula is Theorem~\ref{thm:TF} and the splitting criterion is Corollary~\ref{cor:split}; the cases $D$ non-reduced and $D=\emptyset$ are recorded in Lemmas~\ref{lem:Dnonred} and~\ref{lem:Dempty}. If $\Tr_X(s)\neq0$ the rescaled section splits $M$, and along the regular locus the splitting is compatible with $Z(\LF_s)$ by Lemma~\ref{lem:loccompat}. If $\Tr_X(s)=0$ and $\rho\ge1$ the trace is zero and no rescaling splits $M$. The section $\LF_s$ is independent of the datum up to a regular unit, by Theorem~\ref{thm:hd-canonicity}.
\end{proof}

\begin{remark}[Constancy of the trace and of units]\label{rem:Mcomplete}
The splitting criterion uses that $\Tr(\LF_s^{p-1})$ is constant, which follows from the extension of $\LF_s$ to a projective compactification by Lemma~\ref{lem:consttrace-extend}.  If $M$ is complete, Lemma~\ref{lem:consttrace-complete} gives the same conclusion directly.  Uniqueness up to scalar holds when $\Tr_X(s)\neq0$, by Lemma~\ref{lem:unique}; in general two data give the same section up to a regular unit.
\end{remark}

\begin{remark}\label{rem:smooth}
The compatibility of the splitting with $Z(\LF_s)$ holds along the regular locus of $M$ by Lemma~\ref{lem:loccompat}.
\end{remark}

\section{The moduli of principal bundles over Poisson surfaces}\label{sec:others}

The local algebra is that of Section~\ref{sec:moduli}.  The differences for principal bundles are an invariant form and reduction directions.

Let $X$ be a smooth projective Poisson surface with weight $s\in H^0(X,\omega_X^{-1})$.  Let $G$ be a smooth reductive group over $k$ whose Lie algebra $\mathfrak g$ carries an invariant nondegenerate bilinear form $\ka$, unrelated to the constant $\kappa_0$ of the trace pairing.  One may take the trace form of a defining representation in good characteristic, the Killing form when $p$ is very good, or $\ka=\Tr(\varrho\,\cdot\,\varrho\,\cdot)$ for a faithful representation $\varrho\colon G\hookrightarrow\GL(V)$ whose trace form is nondegenerate.  We write $\varrho$ for this representation, and $\rho$ retains its intrinsic meaning $\frac12\operatorname{rk}\theta_s$, realized by the Hecke members at member-saturated data (Remark~\ref{rem:rho-open}).  Write $M_G=M_G^{\mathrm{st}}(\tau)$ for the moduli space of $H$-stable principal $G$-bundles of fixed topological type $\tau$ \cite{Ram,Bal}.  We work under $z(\mathfrak g)\otimes H^1(\cO_X)=0$, so $G$ has trivial Lie center or $q(X)=0$; this keeps the central factor of $M_G$ discrete.

Over a surface, at a stable point $[P]$ one has $T_{[P]}M_G\cong H^1(X,\mathfrak g_P)$ and $T^*_{[P]}M_G\cong H^1(X,\mathfrak g_P\otimes\omega_X)$, the second identification by Serre duality together with $\mathfrak g_P^\vee\cong\mathfrak g_P$ through $\ka$.  Write $d_G=\dim M_G$ and $t_G:=\dim H^1(X,\mathfrak g_P)$ for the size of the matrix of $\beta_s$ of \eqref{eq:Gpair}; on the smooth locus $t_G=d_G$.  If $(D\cdot H)>0$ then $H^0(X,\mathfrak g_P\otimes\omega_X)=0$: a nonzero section would give a map $\cO_X(D)\to\mathfrak g_P$, impossible since $\mathfrak g_P$ is semistable of slope $0<\mu(\cO_X(D))$.  Hence $H^2(X,\mathfrak g_P)=0$, $t_G=d_G=\dim z(\mathfrak g)-\chi(X,\mathfrak g_P)$, and $M_G$ is smooth, of dimension
\begin{equation}\label{eq:Gdim}
 d_G=\dim z(\mathfrak g)-\chi(X,\mathfrak g_P)
   =\dim z(\mathfrak g)-\dim G\cdot\chi(\cO_X)+\deg\mathfrak g_P
\end{equation}
in that case; without vanishing of $H^2$ one has $t_G=\dim z(\mathfrak g)-\chi(X,\mathfrak g_P)+h^2(X,\mathfrak g_P)$ and $d_G=t_G$ on the smooth locus.  Here $\deg\mathfrak g_P:=c_2(\operatorname{ad}P)$ (the adjoint bundle has $c_1=0$).  Smoothness of $M_G$ is not needed in the splitting statement. The moduli of stable principal $G$-bundles is normal, being a geometric quotient of a smooth parameter space by a reductive group; Remark~\ref{rem:nearsplit} applies.

\begin{remark}[The center]\label{rem:Gcenter}
The adjoint action kills $\mathfrak z=z(\mathfrak g)$, and $\ka$ is orthogonal between $\mathfrak z$ and $[\mathfrak g,\mathfrak g]$.  Under $z(\mathfrak g)\otimes H^1(\cO_X)=0$ the central tangent directions vanish.  Without that hypothesis the central factor is a torsor under $\operatorname{Pic}^0(X)$, the frame construction has no Euler fields there, and for supersingular $\operatorname{Pic}^0$ the trace vanishes identically.  The setting excludes this case.
\end{remark}

When $D=Z(s)$ is nonempty the pairing is
\begin{equation}\label{eq:Gpair}
 \beta_s(\xi,\eta)\;=\;\sum_{x\in D}\Res_x\bigl(s\cdot\ka(\xi\cup\eta)\bigr).
\end{equation}
When $D=\emptyset$ it is the Serre-duality trace of $s\cdot\ka(\xi\cup\eta)$ in $H^2(X,\omega_X)\cong k$, as in \eqref{eq:res0}.

\begin{lemma}[Identification]\label{lem:Gidentify}
Normalize $\ka=\Tr(\varrho\,\cdot\,\varrho\,\cdot)$ and let $E_\varrho=P\times^G V$.  Then \eqref{eq:Gpair} is the restriction to the image of $H^1(X,\mathfrak g_P)$ in $H^1(X,\End E_\varrho)$ of the weighted residue pairing on $H^1(X,\End E_\varrho)$.
\end{lemma}

\begin{proof}
The inclusion $\mathfrak g_P\to\End E_\varrho$ is compatible with cup products and carries $\ka$ to the restriction of the trace form.  This is an identity of cup products and traces; no stability of $E_\varrho$ is involved.
\end{proof}

\begin{lemma}[Orthogonality]\label{lem:Gorth}
Let $G\subset\GL(V)$ be the stabilizer of a tuple of tensors and $\mathfrak m:=\mathfrak g^\perp\subset\End V$ the trace-orthogonal complement.  Then $\End V=\mathfrak g\oplus\mathfrak m$ as $G$-modules and $\ka(\mathfrak g,\mathfrak m)=0$.
\end{lemma}

\begin{proof}
The trace form $\Tr(AB)$ on $\End V$ is nondegenerate and invariant under $G$, and $\mathfrak g$ is the annihilator of the tensors in $\End V$, hence a $G$-submodule.  Its orthogonal complement $\mathfrak m$ is again a $G$-submodule, $\End V=\mathfrak g\oplus\mathfrak m$, and $\ka(\mathfrak g,\mathfrak m)=0$ by construction.
\end{proof}

\begin{lemma}[Mixed terms vanish]\label{lem:Gsub}
For the reduction sequence $0\to\mathfrak g_P\to\End E_\varrho\to\mathfrak m_{E_\varrho}\to0$, the weighted residue pairing of a $\mathfrak g_P$-class and an $\mathfrak m_{E_\varrho}$-class vanishes, and on the $\mathfrak g_P$-directions it restricts to \eqref{eq:Gpair}.
\end{lemma}

\begin{proof}
The mixed pairing lands in the $(\mathfrak g\otimes\mathfrak m)$-part, and $\Tr(\mathfrak g,\mathfrak m)=0$ by Lemma~\ref{lem:Gorth}.
\end{proof}

A reduction datum along a fat point $Z$ is a reduction of $P|_Z$ to a parabolic $\cP\subset G$.  The principal modification $P'$ along $(Z,R)$ has sections the sections of $P$ whose restriction to $Z$ lies in $R$.  From $0\to\operatorname{ad} P'\to\operatorname{ad} P\to Q\to0$ with $Q$ supported on $Z$ of length $\ell\cdot\dim(\mathfrak g/\mathfrak p)$,
\[
 c_2(\operatorname{ad} P')=c_2(\operatorname{ad} P)+\ell\cdot\dim(\mathfrak g/\mathfrak p),
\]
since $\operatorname{ch}_2(\operatorname{ad} P')=\operatorname{ch}_2(\operatorname{ad} P)-\ell\dim(\mathfrak g/\mathfrak p)$ while $c_1$ is unchanged ($Q$ is supported in codimension two), and $c_2=(c_1^2-2\operatorname{ch}_2)/2$.
The deformation module of the pair $(Z,R)$ splits as $\Def_Z\oplus\Def(R)$, of dimensions $2\ell$ and $\ell\cdot\dim(\mathfrak g/\mathfrak p)$.  The $G$-structure is carried by $R$ alone: moving $Z$ contributes only the scalar parameters, so the leaf block is one copy of $A^{\oplus 2}$ and not $A^{\oplus 2}\otimes\mathfrak g$.  Reduction motions land in $\operatorname{im}(H^0(Q)\to H^1(\operatorname{ad} P'))$ and die in $H^1(\operatorname{ad} P)$.

\begin{lemma}[Local block]\label{lem:Gblock}
On the scheme motions the pairing \eqref{eq:Gpair} restricts to the
trace pairing with $g=s|_{\cO_Z}$, and the block Pfaffian is $\pm\kappa_0^\ell N_{\cO_Z/k}(g)$ up to a unit.  All
pairings involving a reduction direction vanish, so the reduction directions
lie in $\operatorname{rad}\beta_s$.
\end{lemma}

\begin{proof}
The scheme motions are the deformations of the two relations of $Z$.  Applying $(-)\times^GV$, the modification becomes the elementary modification of $E_\varrho$ along $Z$, and Lemma~\ref{lem:Gidentify} identifies the restriction of \eqref{eq:Gpair} with the trace pairing on that image.  The local computation of Section~\ref{sec:moduli} uses only the two relations of $Z$ and the weight $s|_{\cO_Z}$, so
$\beta_s\bigl((a,b),(a',b')\bigr)=\kappa_0\Res_A\bigl(s|_A(ab'-a'b)\bigr)$.
Lemma~\ref{lem:normform} gives the Pfaffian.  Hence the exponent of the norm factor is $1$, not $\dim\mathfrak g$.

Let $\xi$ be a reduction class and $\eta$ any class of $H^1(\operatorname{ad} P')$.  Naturality of cup products, the fact that $\ka$ is a bundle map, and $\iota_*\xi=0$ in $H^1(\operatorname{ad} P)$ give $\ka(\xi\cup\eta)=0$, hence $\beta_s(\xi,\eta)=0$.
\end{proof}

The Hecke transforms of \cite{Bal} supply a presentation, by Remark~\ref{rem:Geuler}.  The lengths of the budget are adjustable within the type $\tau$: a fat-point modification of length $\ell$ raises $c_2(\operatorname{ad}P)$ by $\ell\cdot\dim(\mathfrak g/\mathfrak p)$ (equivalently lowers $\operatorname{ch}_n$ by that number), and the degrees of the reduction data along $C$ are free integer parameters.

\begin{theorem}[The reductive case]\label{thm:G}
Let $\rho>0$.  Then $\LF_s$ splits $M_G$ if and only if $s^{p-1}$ Frobenius splits $X$, that is, if and only if $\Tr_X(s)\neq0$.  When the splitting exists it is compatible with $Z(\LF_s)$.
\end{theorem}

\begin{proof}
The scheme block of a Hecke chain is the trace pairing with weight $g=s|_{\cO_Z}$ by Lemma~\ref{lem:Gblock}, and the block coefficient of a length-one member is $\Tr_X(s)$ by Lemma~\ref{lem:traceblock}. The principal case of Theorem~\ref{thm:TF} is proved by the Hecke presentation of Lemma~\ref{lem:hecke-present} and the root block formula. The splitting criterion is Corollary~\ref{cor:split}. If $\Tr_X(s)\neq0$ a rescaling splits a dense open, and Remark~\ref{rem:nearsplit} upgrades the splitting to all of $M_G$; if $\Tr_X(s)=0$ the trace is zero and no rescaling splits. The cases $D$ non-reduced and $D=\emptyset$ are as in Lemmas~\ref{lem:Dnonred} and~\ref{lem:Dempty}.
\end{proof}

\begin{theorem}[Canonicity for principal bundles]\label{thm:G-canonicity}
Let $M_G$ be normal.  Then $\LF_s$ is independent of the presentation up to a regular unit, and up to a constant when $\Tr_X(s)\neq0$.  Consequently $Z(\LF_s)$ is intrinsic, and when $s$ splits the base the assignment $[s]\mapsto[\LF_s]$ is a well-defined map
$\mathbb P H^0(X,\omega_X^{-1})\to\mathbb P H^0(M_G,\omega_{M_G}^{[-1]})$.
\end{theorem}

\begin{proof}
The pairing $\beta_s(\xi,\eta)=\sum_{x\in D}\Res_x\bigl(s\cdot\ka(\xi\cup\eta)\bigr)$ is written on $H^1(X,\mathfrak g_P)$ by cup product, the invariant form, and the residue along $D$.  It makes no reference to a chain, to Hecke points, or to a presentation, by Lemma~\ref{lem:Gidentify}.  Its transpose across Serre duality is the bivector $\theta_s$, a well-defined section of $\wedge^2 T_{M_G}$, and the Pfaffian power $\theta_s^{[\rho]}$ is a well-defined section of $\wedge^{2\rho}T_{M_G}$ on the open of constant rank, where $2\rho$ is the generic rank of $\theta_s$.  A presentation realizing the leaf has $\rho$ members.

A presentation supplies the Euler frame $\Omega$ by descent, by Remark~\ref{rem:Geuler} and Lemma~\ref{lem:descent}.  By {\rm(P3)} and Lemma~\ref{lem:hd-transverse}, the latter using only the constant rescalings of class coordinates and not the members, two presentations give radial determinants $\eta$ and $\lambda\eta$ with $\lambda\in k^\times$.  The identification $\theta_s^{[\rho]}\wedge\Omega=\theta_s^{[\rho]}\otimes\eta$ of Corollary~\ref{cor:frame} therefore yields $\LF_s'=\lambda\,\LF_s$.

On a surface the local expression of Theorem~\ref{thm:normsec} agrees with this intrinsic description: Lemma~\ref{lem:Gblock} identifies the scheme block of any presentation with the restriction of $\beta_s$, so the leaf Pfaffian of the presentation is the restriction of $\theta_s^{[\rho]}$.  Thus the two constructions coincide, and the assignment does not depend on which presentation was used to compute the local expression.
\end{proof}

\noindent The well-definedness uses the intrinsic pairing, which is special to surfaces.

\begin{remark}[Presentation]\label{rem:Geuler}
The principal modification theory of \cite{Bal} supplies a presentation in the sense of Definition~\ref{def:present}, with reduction data in place of elementary ones.  The clauses (P1)--(P3) are as in Lemma~\ref{lem:eulerframe}, with $\mathfrak g$ in place of $\End$.
\end{remark}

\section{The moduli of instantons over Poisson pairs}\label{sec:inst}

We treat the setting of a Poisson pair $(X,L)$.  The local algebra is again that of Section~\ref{sec:moduli}; the differences for instantons are a parabolic trace and a noncomplete open set.

A \emph{Poisson pair} $(X,L)$ is a smooth projective surface $X$ together with a smooth divisor $L\subset X$; the weight of the pair is a nonzero section $s\in H^0(\omega_X^{-1}(-2L))$, and $U=X\setminus L$.  The criterion is that $s$ Frobenius-splits the pair $(X,L)$.  Write $M=M_{r,k}(X,L)$ for the moduli space of framed torsion-free sheaves $(E,\iota)$ on $(X,L)$.  Over a surface the pair is Poisson when $s$ vanishes to order at least two along $L$.  For $(\mathbb P^2,L)$ a line this is the ADHM moduli space of rank-$r$ charge-$k$ instantons on $\mathbb A^2$ \cite{ADHM,Nak,HL,CR}.  On $(\mathbb P^2,L)$ one has $\omega_X^{-1}(-2L)=\cO(1)$.

\noindent\emph{Terminology.}  An \emph{instanton over $(X,L)$} is a framed torsion-free sheaf $(E,\iota)$ on $(X,L)$, and its non-locally free points are the \emph{instanton points}.  This is the moduli-theoretic name, and it carries an analytic meaning only in dimension four.  In the literature an \emph{instanton bundle on $\mathbb P^n$} is a locally free sheaf with $c_1=0$ and cohomological conditions, and an \emph{instanton sheaf on $\mathbb P^3$} is an unframed torsion-free variant; neither is an object of this section.

The framed deformation theory is the parabolic form of Serre duality: deformations of $(E,\iota)$ are $\Ext^1(E,E(-L))$, paired with $\Ext^1(E,E\otimes\omega_X(L))$.  The dualizing sheaf of the pair is $\omega_{par}=\omega_X(L)$, and the weight of the bracket is a section of $\omega_{par}^{-1}(-L)=\omega_X^{-1}(-2L)$.  The framed tensor of \cite[Thm.~4.3]{Bot00} pairs cotangent classes $\xi,\eta\in\Ext^1(E,E\otimes\omega_X(L))$:
\begin{equation}\label{eq:instpair}
 \theta_s(\xi,\eta)=\Tr\bigl(s\cdot(\xi\circ\eta)\bigr)\in H^2(X,\omega_X)\cong k .
\end{equation}

\begin{definition}[Parabolic Frobenius trace]\label{def:pairtrace}
Over $U$ one has $\omega_X^{-1}(-2L)|_U\cong\omega_U^{-1}$.  The \emph{parabolic Frobenius trace} of $s$ is
\[
 \Tr_{(X,L)}(s):=\Tr_U\bigl(s^{p-1}\bigr)\in H^0(U,\cO_U).
\]
\end{definition}

\begin{lemma}[Identification of the pairing]\label{lem:inst-identify}
On a fat-point block at a point of $U$ the pairing \eqref{eq:instpair} is the trace pairing of Definition~\ref{def:trans} with weight $g=s|_{\cO_Z}$.
\end{lemma}

\begin{proof}
On the block the classes are the deformations of the support $Z$, parameterized
by $A\oplus A$; their product is the Koszul class $ab'-a'b$, and the trace of a
$Z$-supported class is the Artin residue.  The weight restricts to $g=s|_{\cO_Z}$,
so the pairing becomes $\kappa_0\Res_A\bigl(g(ab'-a'b)\bigr)$.
\end{proof}

\begin{lemma}[Identification of the unit]\label{lem:inst-T1}
Let $x\in U$, let $(u,\tau)$ be slice coordinates at $x$, and let $Z$ be the length-one member at $x$.  Then
\[
 T(g)=\Tr_{(X,L)}(s)(x).
\]
If $x$ lies on a reduced branch of $D$, this is the extraction $T_1=[\,\tau^{p-1}\,]h(0,\tau)^{p-1}$ of Theorem~\ref{thm:utc}.
\end{lemma}

\begin{proof}
The value of the Frobenius trace at $x$ is the socle coefficient \cite[Lem.~1.3.6]{BKu}.  For a length-one member Lemma~\ref{lem:hd-unit} gives $T(g)=\Tr_U(s^{p-1})(x)$.  When $x$ lies on a reduced branch of $D$ and $f_s=u\cdot h$, where $f_s$ is the local function of the weight in the frame of $\omega_U^{-1}$, the same socle coefficient is $[\,\tau^{p-1}\,]h(0,\tau)^{p-1}$, since $p-1$ is even and $u^{p-1}$ supplies the whole $u$-extraction.
\end{proof}

The verification of Section~\ref{sec:transfer} applies with $A_E=\End E(-L)$: framed transforms preserve the framing along $D'\subset U$, and the members are length-one points of $U$.  Smoothness is not needed; the moduli of framed torsion-free sheaves is normal and $\mathbb Q$-Gorenstein, being an open set of a smooth framed Quot scheme, and Remark~\ref{rem:nearsplit} applies.  The rank count is Lemma~\ref{lem:inst-radical}: with $c$ the corank of the pairing \eqref{eq:instpair} one has $2\rho=d-c$, and the empty leaf is read as in Remark~\ref{rem:emptycount}.

\begin{theorem}[Instantons]\label{thm:inst}
Let $\rho>0$.  Then $\LF_s$ splits $M_{r,k}(X,L)$ if and only if $s^{p-1}$ splits $U=X\setminus L$.  When the splitting exists it is compatible with $Z(\LF_s)$.
\end{theorem}

\begin{proof}
By Lemma~\ref{lem:inst-identify} the datum is a weighted trace pairing datum, and by Lemma~\ref{lem:inst-T1} the block coefficients are $T(g_\nu)=\lambda(x_\nu)$ with $\lambda=\Tr_U(s^{p-1})$.  If $\lambda$ is a nonzero constant, Theorem~\ref{thm:TF} gives $\Tr(\LF_s^{p-1})=\lambda^{\rho}\cdot u$ with $u\in k^\times$, and Corollary~\ref{cor:split} gives the splitting.  If $\lambda$ vanishes identically, Theorem~\ref{thm:TF} gives zero trace and no rescaling splits.  If $\lambda$ is a nonconstant function, then $\lambda^\rho$ is nonconstant because $\rho>0$, so the trace is nonconstant on $M$ and no constant rescaling makes it equal to $1$; moreover $s^{p-1}$ does not split $U$, since $\lambda$ is not a nonzero constant.  In all three cases $\LF_s$ splits $M$ exactly when $s^{p-1}$ splits $U=X\setminus L$, and the compatibility with $Z(\LF_s)$ is Lemma~\ref{lem:loccompat}.
\end{proof}

\begin{example}[A non-splitting pair and a splitting pair]\label{ex:inst-pos}
Let $X=\mathbb P^1\times\mathbb P^1$ with torus coordinates $(z,w)$ and let $L$ be a fibre; then $U\cong\mathbb A^1\times\mathbb P^1$, and $\omega_X^{-1}(-2L)=\cO(0,2)$, whose weight polytope is $P=\{(1,-1),(1,0),(1,1)\}$.  Since $(0,0)\notin P$, Proposition~\ref{prop:inst-toric} shows that no $s$ has $s^{p-1}$ splitting $U$.  Explicitly, for $s=\chi^{(1,1)}$ the invariant-frame unit $x_1x_2$ contributes, so $f_s=z^2w^2$, and
\[
 \Tr_U(s^{p-1})(z,w)=\bigl[\,t_1^{p-1}t_2^{p-1}\,\bigr]
 \bigl((z+t_1)^2(w+t_2)^2\bigr)^{p-1}
 =\binom{2p-2}{p-1}^2 z^{p-1}w^{p-1}=0
\]
in characteristic $p$.  For $L=0$ the canonical weight $s_{-K_X}$ has local function $x_1x_2$ and $\Tr_X(s^{p-1})\equiv1$.
\end{example}

\begin{corollary}[Non-splitting for the ADHM moduli]\label{cor:inst-adhm}
On $(\mathbb P^2,L)$ with $L$ a line one has $\Tr_U(s^{p-1})=0$ for every nonzero weight $s\in H^0(\cO(1))$, so no leafrian splits $M_{r,k}(\mathbb P^2,L)$ when $\rho>0$.
\end{corollary}

\begin{proof}
Every nonzero weight has affine-linear local function $f_s$ on $U=\mathbb A^2$.  Then $f_s^{p-1}$ has total degree $p-1<2(p-1)$, so it contains no monomial $X^{p-1}Y^{p-1}$; equivalently, at a point of the branch the residual $h=f_s/u$ is constant and $T_1=0$.  Hence $\Tr_U(s^{p-1})=0$ for every $s$, and the conclusion follows from Theorem~\ref{thm:inst}.
\end{proof}

\noindent Here $d=2rk$, and the radical vanishes at a general point by Lemma~\ref{lem:inst-radical}, so $2\rho=d$ and
\[
 \rho=rk .
\]
For $M_{1,k}=\operatorname{Hilb}^k(\mathbb A^2)$ one has $\rho=k$, and the leaf is empty only for $k=0$.  When the leaf is empty no leafrian is formed, the assembly is the deformation volume, and it splits $M$ whenever $M$ carries a full radial frame in the sense of Proposition~\ref{prop:empty}.  For the defining section of $L$ the scaling obstruction of Lemma~\ref{lem:inst-scale} applies to the canonical symplectic section of \cite[Prop.~5.1]{Bot00}.

\noindent Scope of the excluded case $Z(s)\subset L$.  The conclusion of Corollary~\ref{cor:inst-adhm} uses $U=\mathbb A^2$; for a general pair $(X,L)$ with $Z(s)\subset L$ no weight with $s^{p-1}$ splitting $U$ is claimed here.  Lemma~\ref{lem:inst-scale} below gives the same conclusion whenever a $\mathbb G_m$-action of nonzero weight is available, as for the dilation of the ADHM description.

\begin{lemma}[Scaling obstruction]\label{lem:inst-scale}
Let $\Gamma=\mathbb G_m$ act algebraically on a smooth connected variety $Y$, and let $\LF_s$ be a leafrian with $g_t^*\LF_s=t^{w}\LF_s$ for some $w\neq0$.  Then no rescaling of $\LF_s$ splits $Y$.
\end{lemma}

\begin{proof}
For $f=\Tr(\LF_s^{p-1})$ one has $f\circ g_t=t^{w(p-1)}f$.  A constant $c$ satisfies $c=t^{w(p-1)}c$ for all $t$, whence $c=0$ since $w(p-1)\neq0$; a nonconstant $f$ cannot be normalised to $1$.  A leafrian that splits $Y$ therefore never satisfies the hypothesis.
\end{proof}

\noindent On $(\mathbb P^2,L)$ the dilation of the ADHM description scales the symplectic form with weight $2$, so no rescaling of the canonical symplectic section of \cite[Prop.~5.1]{Bot00} splits $M$.  This does not touch the leafrian of Theorem~\ref{thm:inst}, whose splitting is governed by $\Tr_U(s^{p-1})$; in the empty-leaf case the assembly $\eta_{\mathrm{can}}$ has nonzero top coefficient and splits whenever its trace is a nonzero constant, by Proposition~\ref{prop:empty}.

\begin{proposition}[Toric classification]\label{prop:inst-toric}
Let $X$ be a smooth complete toric surface with fan $\{v_i\}$, let $L=\sum_i a_iD_i$ ($a_i\geq0$) be an invariant divisor, $U=X\setminus L$, and let
\[
 P=\bigl\{\,m\in M_{\mathbb R}:\ \langle m,v_i\rangle\ge 2a_i-1\ \text{for all } i\,\bigr\}
\]
be the weight polytope of $\omega_X^{-1}(-2L)$, and let $s_D$ be the invariant section attached to $0\in M$.  The following are equivalent: {\rm(i)} some weight $s\in H^0(\omega_X^{-1}(-2L))$ has $s^{p-1}$ splitting $U$; {\rm(ii)} $(0,0)\in P$, equivalently $a_i=0$ for every $i$ (that is, $L=0$); {\rm(iii)} $s=s_D$ has $s^{p-1}$ splitting $U$, with $\Tr_U(s^{p-1})\equiv1$.  The condition does not depend on $p>2$.
\end{proposition}

\begin{proof}
The section $s_D$ lies in $H^0(\omega_X^{-1}(-2L))$ exactly when $(0,0)\in P$.  On $T\subset U$ the invariant frame of $\omega_X^{-1}$ is $e=(x_1\partial_{x_1})\wedge(x_2\partial_{x_2})=x_1x_2\,\partial_{x_1}\wedge\partial_{x_2}$, and the invariant frame $s_D$ of $\omega_X^{-1}(-2L)$ restricts on $U$ to $e$.  A weight $s=\sum_m c_m\chi^m s_D$ with $\operatorname{supp}\subseteq P\cap M$ has local function $G=x_1x_2F$ with $F=\sum_m c_m\chi^m$, and in additive coordinates $t_j=x_j-a_j$ at $a\in T$ the local trace formula gives
$\Tr_U(s^{p-1})(a)=[\,t^{(p-1)\mathbf1}\,]\,G(a+t)^{p-1}$.
Write $G(a+t)=\sum_\beta A_\beta(a)\,t^\beta$.  A product term selecting the monomial $m_k$ in the $k$-th factor of $G^{p-1}$ has $a$-exponent $\sum_k(m_k-\beta^k)$; since $\sum_k\beta^k=(p-1,p-1)$, a nonzero constant term forces $\sum_k m_k=(p-1,p-1)$.  As $\operatorname{supp}(G)$ lies in the convex set $P+(1,1)$, the average $(1,1)$ lies in $P+(1,1)$.  Thus a nonzero constant trace forces $(0,0)\in P$.  Conversely, when $(0,0)\in P$ the weight $s=s_D$ has local function $G=x_1x_2$, and
$\Tr_U(s^{p-1})(a)=[\,t^{(p-1)\mathbf1}\,]\bigl((a_1+t_1)(a_2+t_2)\bigr)^{p-1}\equiv1$.
Finally $(0,0)\in P$ is $\langle 0,v_i\rangle\ge 2a_i-1$, i.e.\ $1-2a_i\ge0$; since the $a_i$ are nonnegative integers, this holds exactly when $L=0$.
\end{proof}

\begin{remark}[Scope of the criterion]\label{rem:inst-toric}
The invariant pairs that admit a weight $s$ with $s^{p-1}$ splitting $U$ are exactly those with $L=0$.  In particular $(\mathbb P^1\times\mathbb P^1$, fibres$)$ admits none, and $(\mathbb P^2,L)$ admits none for $L\neq0$.  For such a pair, the empty-leaf case is governed by the deformation volume.
\end{remark}

\section{Poisson enrichment}\label{sec:cfp}

The splitting criterion of Sections~\ref{sec:moduli}--\ref{sec:inst} did not use the Jacobi identity.  This section records the Poisson enrichment: the leafrian is Poisson, the assignment from weights is injective, and the zero scheme is the rank-drop locus.  Throughout, $X$ is a smooth projective Poisson surface, $s\in H^0(X,\omega_X^{-1})$ is nonzero (or $s\in H^0(\omega_X^{-1}(-2L))$ in the framed case), $D=Z(s)$, and $M$ is one of the three moduli spaces.  The number $\rho$ is half the generic rank of $\theta_s$ (Remark~\ref{rem:rho-open}), and $\hat s$ is the transferred bivector of Section~\ref{sec:hat}.

\subsection{Poisson linear algebra}

Let $Y$ be a smooth variety with Poisson bivector $\pi$ of generic rank $2\rho$, and let $U_\pi\subset Y$ be the dense open where the rank is $2\rho$.  The sharp map $\pi^\sharp:T^*_Y\to T_Y$ has image the tangent bundle of the symplectic leaves.  The flat map $\pi^\flat:T_Y\to T^*_Y$ has kernel $\cV:=\ker\pi^\flat$, the radical.  The normal form of a skew matrix gives $T_Y|_{U_\pi}=\im\pi^\sharp\oplus\cV$.  Sign conventions: $X_f=[\pi,f]=\pi^\sharp(df)$, $\cL_w\pi=[w,\pi]$, and $\{f,g\}=\pi(df,dg)$; graded Jacobi gives $[[\pi,f],w]=[\pi,[f,w]]+[f,[\pi,w]]$, and in particular $[X_f,w]=-X_{w(f)}$ whenever $\cL_w\pi=0$.  Put $Q:=T_Y/\im\pi^\sharp$, of rank $c=d-2\rho$ on $U_\pi$.

\begin{lemma}[Reduction]\label{lem:reduction}
Let $\Omega\in H^0(\Lambda^cT_Y)$ and $\sigma=\pi^{[\rho]}\wedge\Omega$.  For $x\in U_\pi$ one has $\sigma\neq0$ at $x$ if and only if the $\Lambda^c\cV_x$-component of $\Omega_x$ is nonzero; equivalently, if and only if the $Q$-projection of $\Omega_x$ is nonzero.  The condition does not depend on the choice of complement $\cV$.
\end{lemma}

\begin{proof}
Choose a basis $e_1,\dots,e_{2\rho}$ of $\im\pi^\sharp_x$ with $\pi^{[\rho]}|_x=\lambda\,e_1\wedge\cdots\wedge e_{2\rho}$, $\lambda\neq0$, and complete it by a basis $f_1,\dots,f_c$ of $\cV_x$.  Expanding $\Omega_x$, every monomial containing an $e$-factor dies after wedging with $e_1\wedge\cdots\wedge e_{2\rho}$.  Only the $f$-monomials survive.
\end{proof}

\begin{lemma}[Line bundle]\label{lem:linebundle}
On $U_\pi$ one has $\det\cV\cong\omega_Y^{-1}$ and $\det Q\cong\cO(-K_Y)$.
\end{lemma}

\begin{proof}
On $U_\pi$ the nowhere vanishing $\pi^{[\rho]}$ trivializes $\det(\im\pi^\sharp)$, so $\det T_Y=\det(\im\pi^\sharp)\otimes\det\cV=\det\cV$.  Moreover $Q\cong\cV$ by the same normal form.
\end{proof}

\begin{lemma}[Invariance criterion]\label{lem:invariance}
Let $\pi$ be Poisson of generic rank $2\rho$ on a smooth $Y$, let
$\sigma=\varpi\wedge\Omega\neq0$ be a leafrian, and let $\eta$ be
the transverse volume of $\Omega$.  Then $\cL_{X_f}\sigma=0$ for every function
$f$ if and only if $\cL_{X_f}\eta=0$ for every $f$.  In particular, when the
symplectic leaves are the fibres of a map $\mu$ and $\eta=\mu^*(\eta_0)$ for a
volume $\eta_0$ on the leaf space, the section $\sigma$ is a Poisson leafrian.
\end{lemma}

\begin{proof}
Since $\cL_{X_f}\pi=0$, the derivation $\cL_{X_f}$ preserves $\pi^{[\rho]}$ and
the splitting $T=\operatorname{im}\pi^\sharp\oplus\cV$.
By Leibniz, $\cL_{X_f}\sigma=\pi^{[\rho]}\wedge\cL_{X_f}\Omega$, and the $\Lambda^c\cV$-component of $\cL_{X_f}\Omega$ is $\cL_{X_f}\eta$.
If the symplectic leaves are the fibres of $\mu$ and $\eta=\mu^*(\eta_0)$, then $d\mu(X_f)=0$, whence $\cL_{X_f}\eta=0$.
\end{proof}

\begin{corollary}[Invariance]\label{cor:norminv}
If $\pi$ is Poisson, $\LF_s$ is invariant under the Hamiltonian vector fields:
$\cL_{X_f}\LF_s=0$ for every function $f$.
\end{corollary}

\begin{proof}
By Lemma~\ref{lem:invariance} it suffices that $\cL_{X_f}\eta_{\mathrm{can}}=0$.  We work on the dense open set $V:=U_\pi\cap\{\prod_j c_j\neq0\}$ and extend at the end.  The symbols $c_j$ are local slice coordinates on $V$; the Euler fields $w_j$ are global by Lemma~\ref{lem:descent}, and locally $w_j=c_j\partial_{c_j}$.  On $V$ the leaf matrix is block diagonal with nondegenerate blocks, and the class directions span the transverse directions; hence the radical is spanned by the $w_j$ and the symplectic leaf tangent is the $(q,p)$-span.  A Hamiltonian field $X_f$ is tangent to the leaves, hence $X_f(c_j)=0$ on $V$; and $\cL_{X_f}\pi=0$ preserves the radical, so $[X_f,w_j]=\sum_k a_{kj}w_k$ on $V$.  Evaluating on the local coordinates $c_i$ gives
\[
 [X_f,w_j](c_i)=X_f\bigl(w_j(c_i)\bigr)-w_j\bigl(X_f(c_i)\bigr)
   =X_f(\delta_{ij}c_i)=0,
 \qquad \Bigl(\sum_k a_{kj}w_k\Bigr)(c_i)=a_{ij}c_i,
\]
so $a_{ij}=0$ on $V$.  Thus $\cL_{X_f}(w_1\wedge\cdots\wedge w_c)=0$ on $V$, and on $M$ by regularity.  No completeness enters, and no flow of $X_f$ is used.
\end{proof}

\begin{definition}[Poisson leafrian]\label{def:poissonfr}
Let $\pi$ be a Poisson tensor on $Y$.  A leafrian $\sigma$ is a \emph{Poisson leafrian} when $\cL_{X_f}\sigma=0$ for every function $f$.
\end{definition}

\begin{remark}[A sufficient condition for invariance]\label{rem:upgrade-crit}
Let $\sigma=\varpi\wedge\Omega$ be a leafrian.  If $\cL_{w_i}\pi=0$ for every $i$, then $\cL_{X_f}\sigma=0$ for every function $f$.  Indeed
\[
 [\pi,\Omega]=\sum_i\pm\cL_{w_i}\pi\wedge(\text{omit }w_i),
\]
so $[\pi,\Omega]=0$, and graded Jacobi gives
\[
 [X_f,w_i]=[[\pi,f],w_i]=[\pi,[f,w_i]]=-X_{w_i(f)}\in\operatorname{im}\pi^\sharp .
\]
Hence $\cL_{X_f}\Omega$ is a sum of wedges in which one slot lies in $\operatorname{im}\pi^\sharp$.  Since $\pi^{[\rho]}$ is the top wedge of the leaf directions, $\pi^{[\rho]}\wedge\cL_{X_f}\Omega=0$.
\end{remark}

\begin{theorem}[Upgrade]\label{thm:upgrade}
Let $\pi$ be a Poisson tensor on $Y$ and let $\sigma$ be a Poisson leafrian.  If $\sigma$ splits $Y$, then every compatibly split subvariety of $Y$ is a Poisson subvariety.
\end{theorem}

\begin{proof}
The invariance says that every Hamiltonian vector field preserves $\sigma$ and its splitting.  The statement for compatibly split subvarieties is \cite[Thm.~D]{PY}.
\end{proof}

\begin{corollary}[The upgrade on moduli]\label{cor:normupgrade}
For the leafrian of a datum the tensor $\hat s$ is Poisson, $\LF_s$ is a Poisson leafrian, and every compatibly split subvariety of $M$ is a Poisson subvariety.
\end{corollary}

\begin{proof}
By Theorem~\ref{lem:schouten} the tensor $\hat s$ is Poisson, and by Corollary~\ref{cor:norminv} the leafrian is invariant.  Theorem~\ref{thm:upgrade} applies.
\end{proof}

\subsection{The radical and the type spectrum}

This subsection records the radical of the intrinsic pairing $\beta_s$ and the type spectrum of its contribution $c_0$.

\begin{lemma}[Radical of the pairing; rank count]\label{lem:radical}
Let $[E]\in M$ and let $B_s$ be the matrix of $\beta_s$ on $\Ext^1(E,E)$ at $[E]$, of size $t=\dim\Ext^1(E,E)$ as in Section~\ref{sub:moduli-setting} (so $t=d$ on the smooth locus of $M$).  Write
\[
 c_0=h^0(D,\End(E|_D))-1.
\]
Then
\[
 \dim\Ker B_s=c_0,
\qquad
 \rk\theta_s=t-c_0.
\]
At a point of the smooth locus one has $t=d$, so $\rk\theta_s=d-c_0$; when the members realize the leaf this is $2\rho$, and the budget $2\rho+c=d$ identifies $c=c_0$.  If $E$ is not locally free at a point of $D$, then $E|_D$ has torsion and is not simple.
\end{lemma}

\begin{proof}
The pairing $\beta_s$ is the transpose of Bottacin's hamiltonian morphism along the exact sequence
\[
 0\to\End(E)\otimes\omega_X\to\End(E)\to\End(E|_D)\to 0,
\]
whose connecting map identifies
\[
 \Ker B(E)=H^0(D,\End(E|_D))\big/ C
\]
with $C$ the image of $H^0(\End E)$ in $H^0(\End E|_D)$ \cite[Prop.~6.2]{Bot}.  For stable $E$ one has $H^0(\End E)=k$, so $\dim C=1$ and $\dim\Ker B_s=c_0$.  The rank $t-c_0$ is even because $B_s$ is antisymmetric.  If $E$ is singular at $x\in D$, then $E|_D$ has a nonzero torsion subsheaf $T$ by Lemma~\ref{lem:hecke}.  From $0\to T\to E|_D\to F\to 0$ with $F$ locally free one gets $\Hom(E|_D,T)\supseteq\Hom(F,T)\neq0$, and composing with $T\subset E|_D$ gives endomorphisms of $E|_D$ with torsion image, which are linearly independent of the scalars; hence $h^0(\End(E|_D))\geq2$.
\end{proof}

\begin{corollary}\label{cor:nonsimple}
If $E|_D$ is not simple at some point of $M$, then $\theta_s$ does not have maximal rank there.
\end{corollary}

\begin{proof}
If $E|_D$ is not simple then $h^0(D,\End E|_D)\ge2$, so $c_0\ge1$ and the
radical has dimension $c_0\ge1$ at that point.  Hence the rank of
$\theta_s$ is at most $d-1$ there.
\end{proof}

\begin{lemma}[Shape of $D$]\label{lem:shape}
Assume $D\neq\emptyset$.  If $s^{p-1}$ Frobenius-splits $X$, then $D$ is reduced of arithmetic genus $p_a(D)=1$; if irreducible, $D$ is an elliptic curve or a rational curve with one node or cusp.
\end{lemma}

\begin{proof}
Reducedness is \cite[Prop.~1.3.11]{BKu}.  For the arithmetic genus, $D\in|-K_X|$ gives
$0\to\omega_X\to\cO_X\to\cO_D\to0$, and Serre duality in dimension two gives
$\chi(\omega_X)=\chi(\cO_X)$, whence $\chi(\cO_D)=0$ and $p_a(D)=1$.
\end{proof}

\noindent The \emph{type spectrum} of $c_0$ depends only on the generic isomorphism type of $E|_D$: multiplier type gives $c_0=r^2h^0(\cO_D)-1$, block type gives $c_0=(\sum_\alpha m_\alpha^2)h^0(\cO_D)-1$, and simple type gives $c_0=h^0(\cO_D)-1$.  For connected $D$ these are $r^2-1$, $\sum_\alpha m_\alpha^2-1$ and $0$.  In all cases $c_0\le r^2h^0(\cO_D)-1$.  The corank of $\beta_s$ is $c_0$, and when the members realize the leaf this equals the transverse count $c$.

\begin{remark}[When the leaf is empty]\label{rem:emptycount}
By Lemma~\ref{lem:radical} one has $2\rho=d-c$ with $c=c_0$, so
\[
 \rho=0\qquad\Longleftrightarrow\qquad d=c_0.
\]
Since $c_0\le r^2h^0(\cO_D)-1$, the empty leaf forces
\[
 d\ \le\ r^2h^0(\cO_D),
\]
and for connected $D$ this is $d\le r^2$.  For $r=1$ one has $c_0=h^0(\cO_D)-1$ at a generic object, and the empty leaf occurs exactly when
\[
 d=h^0(\cO_D)-1.
\]
For principal bundles the same count reads $c_G\le\dim G\cdot h^0(\cO_D)-1$.  For instantons the count is $c=h^0(\End(E)(-L)|_D)$ by Lemma~\ref{lem:inst-radical}; on $(\mathbb P^2,L)$ this vanishes at a general point, so $d=2rk$ gives $\rho=rk$, and the leaf is empty only for $k=0$.  The supply question of the empty leaf is Proposition~\ref{prop:empty}.
\end{remark}

\begin{lemma}[Radical and rank, principal case]\label{lem:symp-radical}
At a stable $[P]$ with $G$ semisimple the radical of $\beta_s$ is the image of $H^0(D,\mathfrak g_P|_D)$ in $\Ext^1$, of dimension
\[
 i_G=\dim\operatorname{image}.
\]
Hence $\rk\theta_s=t_G-i_G$ and $\rk\theta_s=d_G-i_G$ on the smooth locus, with $t_G=\dim H^1(X,\mathfrak g_P)$ as in Section~\ref{sec:others}.
\end{lemma}

\begin{proof}
The pairing is the transpose of the hamiltonian morphism of the exact sequence
\[
 0\to\mathfrak g_P\otimes\omega_X\to\mathfrak g_P\to\mathfrak g_P|_D\to 0,
\]
and the same argument as in Lemma~\ref{lem:radical} identifies the kernel with $H^0(D,\mathfrak g_P|_D)/C$, where $C$ is the image of $H^0(X,\mathfrak g_P)$.  For semisimple $G$ and stable $P$ one has $H^0(X,\mathfrak g_P)=0$, so $C=0$ and $c_G=i_G$.  The rank is even because the form is antisymmetric.
\end{proof}

\begin{lemma}[Framed radical]\label{lem:inst-radical}
At a general point of $M_{r,k}(X,L)$ the radical of the framed pairing $\beta_s$ is the image of $H^0(\End(E)(-L)|_D)$ in $\Ext^1(E,E(-L))$ under the connecting map of the framed residue sequence.  Its dimension is the transverse count $c$.
\end{lemma}

\begin{proof}
Parabolic Serre duality identifies $\Ext^1(E,E(-L))^{\vee}$ with $\Ext^1(E,E\otimes\omega_X(L))$, and $\beta_s$ is the transpose of the cup product with $s$.  Multiplication by $s$ is a map
\[
 \End(E)(-L)\longrightarrow\End(E)\otimes\omega_X(L)
\]
with zero scheme $D$ and kernel $\End(E)\otimes\omega_X(L)$, so the framed residue sequence is
\[
 0\longrightarrow\End(E)\otimes\omega_X(L)\xrightarrow{\ \cdot s\ }\End(E)(-L)
 \longrightarrow\End(E)(-L)|_D\longrightarrow 0 .
\]
The radical of $\beta_s$ is the kernel of the cup product with $s$, and the long exact sequence gives
\[
 \operatorname{rad}\beta_s\;=\;\operatorname{im}\delta,\qquad
 \delta\colon H^0\bigl(\End(E)(-L)|_D\bigr)\longrightarrow\Ext^1(E,E(-L)),
\]
with $\operatorname{im}\delta$ the quotient of $H^0(\End(E)(-L)|_D)$ by $H^0(\End(E)\otimes\omega_X(L))$.  For a stable framed object the group $H^0(\End(E)\otimes\omega_X(L))$ is zero unless $\omega_X(L)$ is trivial, and then it consists of the scalars.  Hence the radical is $H^0(\End(E)(-L)|_D)$, modulo scalars in the trivial case.

In a presentation the classes of $\operatorname{im}\delta$ are the classes whose member-block components vanish.  The flag data of the point modifications and the trivialization data of the framing pair trivially, and they are exactly the class coordinates of the budget.  Their number is $c$.  For $L=0$ the sequence is the residue sequence of Lemma~\ref{lem:radical}, and the description agrees with $H^0(D,\End(E)|_D)$ modulo scalars.
\end{proof}

\begin{corollary}[Full-rank case]\label{cor:fullrank}
If $c_0=0$, then $E|_D$ is simple at generic $[E]$, the generic rank of $\theta_s$ is $d$, and $d$ is even.  When in addition the members realize the leaf, $\eta_{\mathrm{can}}$ is a unit and $\LF_s=\hat s^{[\rho]}$ up to a unit.
\end{corollary}

\begin{proof}
The equality $c_0=h^0(D,\End E|_D)-1=0$ says that $E|_D$ is simple, and
$\rk\theta_s=d-c_0=d$ is even.  When the
members realize the leaf one has $2\rho=d$, hence $c=0$ and the leafrian has no
transverse factor.
\end{proof}

\begin{remark}\label{rem:r1}
For $r=1$ one has $\End(E|_D)=\cO_D$ for generic $E$, so $c_0=0$ and the corank of $\beta_s$ is $0$.  For $\chi(\cO_X)=1$ one has $d=2c_2$, so $d$ is even and the rank of $\theta_s$ is $d$.
\end{remark}

\subsection{The assignment and the Torelli theorem}

\begin{proposition}[The assignment]\label{prop:assign}
The leafrian $\LF_s$ exists exactly when $\rho>0$, and then it is independent of the presentation, of the members, of the slice coordinates and of the Euler frame up to a regular unit.  When $s$ splits the base the unit is constant, so there is an assignment
\[
 \Phi_M\colon\ \mathbb P H^0(X,\omega_X^{-1})\longrightarrow
 \mathbb P H^0(M,\omega_M^{[-1]}),\qquad
 [\,s\,]\longmapsto[\,\LF_s\,],
\]
with source $\mathbb P H^0(X,\omega_X^{-1}(-2L))$ in the framed case.
\end{proposition}

\begin{proof}
For sheaves and instantons this is Theorem~\ref{thm:hd-canonicity}; for $M_G$ on a surface it is Theorem~\ref{thm:G-canonicity}.  The pairing is linear in the
weight, hence $\LF_{\lambda s}=\lambda^{\rho}\LF_s$, and the assignment descends to projective spaces.  On a curve the members are not intrinsic and the assignment is taken relative to a presentation.
\end{proof}

\begin{theorem}[Torelli]\label{thm:torelli}
Let $M$ be $M(r,c_1,c_2)$, $M_G$, or $M_{r,k}(X,L)$ with $\rho>0$, and let $s$ split the base.  Then the assignment $\Phi_M$ is injective; more precisely, for nonzero weights $s$ and $s'$,
\[
 \LF_{s'}=c\,\LF_s\ \text{ for some } c\in k^\times
 \qquad\Longleftrightarrow\qquad
 s'=\lambda s\ \text{ for some } \lambda\in k^\times ,
\]
and then $c=\lambda^{\rho}$.
\end{theorem}

\begin{proof}
If $s'=\lambda s$ then $\LF_{s'}=\lambda^{\rho}\LF_s$ and $c=\lambda^{\rho}$.
Conversely assume $\LF_{s'}=c\,\LF_s$ for $c\in k^\times$.  Both sections use
one and the same Euler frame, by Theorem~\ref{thm:hd-canonicity} for sheaves
and framed sheaves and Theorem~\ref{thm:G-canonicity} for principal bundles,
so their zero divisors agree:
\[
 Z(\LF_{s'})=Z(\LF_s).
\]
The zero scheme of $\LF_s$ satisfies
\[
 Z(\LF_s)\setminus Z(\eta_{\mathrm{can}})
 =\bigl\{[E]:\ E|_D\ \text{has torsion}\bigr\}
 \subseteq\bigl\{[E]:\ h^0(D,\End E|_D)>h^0_{\mathrm{gen}}\bigr\},
\]
where $D=Z(s)$, $h^0_{\mathrm{gen}}$ is the minimal value of
$h^0(D,\End E|_D)$, and the inclusion may be strict; this is
Proposition~\ref{prop:intr}.  The set $Z(\eta_{\mathrm{can}})$ does not
depend on $s$.

The divisor $D$ is recovered from the zero scheme.  For a point $z\in X$ let
$E_z$ be the kernel modification at $z$ of a slope stable locally free $F$
of type $(r,c_1,c_2-1)$; then $E_z$ is stable by
Lemma~\ref{lem:saturated}(i), and $E_z|_D$ has torsion exactly when $z\in
D$, by Lemma~\ref{lem:hecke}.  Hence
\[
 z\in D\qquad\Longleftrightarrow\qquad [E_z]\in Z(\LF_s)\setminus
 Z(\eta_{\mathrm{can}}).
\]
Applying this to $D=Z(s)$ and $D'=Z(s')$ and using
$Z(\LF_{s'})=Z(\LF_s)$ gives $Z(s)=Z(s')$.  Two nonzero sections of
$\omega_X^{-1}$ with the same zero divisor differ by a constant, so
$s'=\lambda s$.  The same argument applies on $U=X\setminus L$ in the
framed case.

If $\rho=0$ the leafrian has no leaf factor, so
the assembly is $\eta_{\mathrm{can}}$ for every $s$ and $\Phi_M$ is constant.  When
$\rho\ge1$ and $D\neq\emptyset$ the leafrian $\LF_s$ is never a scalar multiple of $\eta_{\mathrm{can}}$, because by Lemma~\ref{lem:divdec} its divisor
carries the member part $\sum_\nu x_\nu^*(D)\neq0$, whereas $\eta_{\mathrm{can}}$ has none.
If $D=\emptyset$ then $s$ trivializes $\omega_X^{-1}$, resp.\ $\omega_X^{-1}(-2L)$,
and $\mathbb P H^0$ is a single point; on a $K3$ or
abelian surface this is the case $\rho\ge1$ with $\LF_s=\eta_{\mathrm{can}}$ up to a constant.  Hence for $\rho=0$ the map $\Phi_M$ is injective exactly when $\mathbb P H^0$ is a single point.
\end{proof}

\begin{lemma}[Differential of the assignment]\label{lem:diff}
Let $s$ be a weight with $\rho\geq1$ and let
$\delta\in H^0(X,\omega_X^{-1})$.  On the evaluation locus of a filtered presentation, in the coordinates of the reduced degeneracy complex, the derivative of $\Psi(s)=\LF_s$ at $s$ in the direction $\delta$ has local expression
\[
 \mathrm d\Psi_s(\delta)=u\prod_jc_j\cdot
 \sum_{a}N_a(\delta)\prod_{b\neq a}N_b(s),
\]
where the $N_a$ are the norm factors of the reduced degeneracy complex. Consequently the differential of $\Phi_M$ at $[s]$ is injective.
\end{lemma}

\begin{proof}
The norm factors of the reduced degeneracy complex are linear in $s$, so differentiating their product gives the display. If the class of $\delta$ dies in the tangent space at $[\LF_s]$, then $\mathrm d\Psi_s(\delta)$ is a multiple of $\LF_s$, so
\[
 \sum_a\frac{N_a(\delta)}{N_a(s)}
\]
is constant. Each summand depends only on the variables of its own support, so each ratio is constant. Since the supports move, $\delta=\lambda s$ for some $\lambda\in k^\times$.
\end{proof}

By Theorem~\ref{thm:torelli} the map $\Phi_M$ is injective when $\rho\ge1$, and by Lemma~\ref{lem:diff} it is an immersion there.  Hence $\Phi_M$ is an isomorphism onto its image when $\rho\ge1$, and it is constant when $\rho=0$.

\begin{proposition}[Hasse hypersurface]\label{prop:hasse}
The Frobenius trace $\Tr_X(s)$ is a homogeneous form of degree $p-1$ on $H^0(X,\omega_X^{-1})$.  Let $\mathcal H_0\subset\mathbb P H^0(X,\omega_X^{-1})$ be its zero hypersurface, unless $\Tr_X$ vanishes identically.  For $\rho>0$ one has
\[
 \LF_s\ \text{splits}\ M
 \qquad\Longleftrightarrow\qquad
 \Tr_X(s)\neq0
 \qquad\Longleftrightarrow\qquad
 [\,s\,]\notin\mathcal H_0 .
\]
\end{proposition}

\noindent In the framed case the same equivalence holds with $\operatorname{Tr}_{(X,L)}(s)$ in place of $\operatorname{Tr}_X(s)$, but that trace is a function on $U$ and need not be a scalar form on the weight space; when it is nonconstant or identically zero there is no Hasse hypersurface.

\begin{proof}
The first equivalence is the splitting criterion of Theorem~\ref{thm:moduli}, Theorem~\ref{thm:G} and Theorem~\ref{thm:inst} in the form: $\LF_s$ splits exactly when $\Tr_X(s)\neq0$.  The map $s\mapsto s^{p-1}$ is polynomial of degree $p-1$ and the trace functional is linear, with Theorem~\ref{thm:TF} supplying the block product.
\end{proof}

\subsection{The degeneracy divisor}

Write $D_s=Z(\LF_s)$.  The construction is a moduli incarnation of a classical fact: the Pfaffian of a log-symplectic Poisson tensor is an anticanonical section whose zero locus is the degeneracy divisor \cite{MPS}.

\begin{lemma}[Divisor decomposition]\label{lem:divdec}
In the coordinates of a frame datum one has
\[
 D_s=\sum_\nu \operatorname{div}\bigl(N(g_\nu)\bigr)+\sum_j\operatorname{div}(c_j).
\]
For every member, $\operatorname{div}\bigl(N(g_\nu)\bigr)=x_\nu^*(D)$, where $x_\nu^*$ is the sum of the pullbacks along the $\ell_\nu$ support points of the member, and $D=Z(s)$ is counted with multiplicities.  When the members do not fill the leaf, the remaining contributions come from the split End blocks of the reduced degeneracy complex.
\end{lemma}

\begin{proof}
By Lemma~\ref{lem:constunit} the local expression is the
product of the norm factors and the class coordinates.  For a member of length $1$ the norm is the value of $s$ at the moving support and its divisor is $x_\nu^*(D)$.
For a member of length $\ell$ on a branch write $s=u^m h$ with $u$ a local equation of a component of $D$; then $g=z^m\bar h$ and
\[
 N(g_\nu)=(\pm q_0)^m N(\bar h).
\]
Here $\operatorname{div}(q_0)$ is the sum of the pullbacks of $\{u=0\}$ along the $\ell$ support points, and $\operatorname{div}(N(\bar h))$ is the sum of the pullbacks of $Z(h)$ along the same points.  Together they give the pullback of $m\{u=0\}+Z(h)=\operatorname{div}(s)=D$.
\end{proof}

\begin{lemma}[Rank locus]\label{lem:ranklocus}
On the evaluation locus,
$D_s\setminus Z(\eta_{\mathrm{can}})=\bigl\{[E']\in M:\ \operatorname{rk}\hat s([E'])<2\rho\bigr\}$.
\end{lemma}

\begin{proof}
By Lemma~\ref{lem:zeroblock} the leaf matrix is block
diagonal, with the class directions radical.
By Lemma~\ref{lem:normform} a block is degenerate exactly when $N(g)=0$, which happens exactly when the support meets $D$.
\end{proof}

\begin{lemma}[Hecke bridge]\label{lem:hecke}
Let $E'$ be the elementary modification of $E$ at a point $x$ with
one-dimensional quotient, so $0\to E'\to E\to i_*k_x\to0$.  Then:
if $x\notin D$, then $E'|_D\cong E|_D$;
if $x\in D$, then $0\to k\to E'|_D\to E|_D\to k\to0$;
if $C$ is transverse to $D$, an elementary modification
along $C$ changes $E|_D$ only along $C\cap D$.
\end{lemma}

\begin{proof}
Tensor $0\to E'\to E\to i_*k_x\to0$ with $\cO_D$ and compute
$\operatorname{Tor}_1^X(k_x,\cO_D)$ from
$0\to\cO_X(-D)\to\cO_X\to\cO_D\to0$.  The map
$\cO_X(-D)|_x\to\cO_X|_x$ is an isomorphism for $x\notin D$ and zero for
$x\in D$, so $\operatorname{Tor}_1$ is $0$ in the first case and $k$ in the
second.  For the third claim, $\operatorname{Tor}_1^X(i_{C*}Q,\cO_D)$ is supported on $C\cap D$.
\end{proof}

\begin{proposition}[Refined radical formula]\label{prop:radref}
At every point of the evaluation locus,
\[
 \dim\ker B_s=c+\sum_{\nu:\,N(g_\nu)=0}\operatorname{corank}(B_\nu),
 \qquad \operatorname{corank}(B_\nu)=2\ell_\nu-\rk B_\nu .
\]
\end{proposition}

\begin{proof}
By Lemma~\ref{lem:zeroblock} the pairing is block diagonal, the class
directions contributing $c$ to the radical.  By
Lemma~\ref{lem:normform} a block is nondegenerate when
$N(g_\nu)\neq0$ and contributes $\operatorname{corank}(B_\nu)$ when
$N(g_\nu)=0$.
\end{proof}

\begin{proposition}[Intrinsic description of the degeneracy divisor]%
\label{prop:intr}
Writing $h^0_{\mathrm{gen}}$ for the minimal value of $h^0(D,\End E'|_D)$, one has
\[
\begin{aligned}
 D_s\setminus Z(\eta_{\mathrm{can}})
 &=\bigl\{[E']:\ E'\ \text{not locally free at some point of}\ D\bigr\}\\
 &=\bigl\{[E']:\ E'|_D\ \text{has torsion}\bigr\}\\
 &\subseteq\bigl\{[E']:\ h^0(D,\End E'|_D)>h^0_{\mathrm{gen}}\bigr\}.
\end{aligned}
\]
The last inclusion may be strict (a locally free $E'$ with $E'|_D$ of multiplier type has $h^0$ jumping without torsion); only the inclusion is needed below.
Equivalently, $[E']$ lies in $D_s\setminus Z(\eta_{\mathrm{can}})$ if and
only if every realization of the leaf has a block meeting $D$.
\end{proposition}

\begin{proof}
At a point of the evaluation locus the members are the modification points, and $E'$ is not
locally free exactly at those points.  By Lemma~\ref{lem:ranklocus} the
rank drops exactly when a member meets $D$.  If $E'$ is not locally free at a point of $D$, then $E'|_D$
has torsion by Lemma~\ref{lem:hecke}, and $h^0(D,\End E'|_D)$ is not
minimal: by Proposition~\ref{prop:radref} the radical has dimension $c+\sum_{\nu:N(g_\nu)=0}\operatorname{corank}(B_\nu)>c$, while at a point where $h^0$ is
minimal Lemma~\ref{lem:radical} gives $\dim\ker B_s=c$.
Conversely, if $h^0$
jumps then $E'|_D$ is not simple, so the radical is strictly larger than the generic radical of Lemma~\ref{lem:radical}, and the rank drops below $2\rho$.  When the members are intrinsic to the object the condition does not
depend on the presentation; in general the rank locus is intrinsic by Lemma~\ref{lem:ranklocus} and Theorem~\ref{thm:hd-canonicity}, so the description of $D_s$ is presentation-independent.
\end{proof}

\begin{proposition}[Classes of the components]\label{prop:classes}
The member part
\[
 R=\sum_\nu\operatorname{div}\bigl(N(g_\nu)\bigr)
\]
is supported on the rank-drop locus, and on $U_\pi$ one has
\[
 D_s=Z(\eta_{\mathrm{can}}).
\]
Moreover
\[
 [\,Z(\eta_{\mathrm{can}})\,]=c_1(\det Q),\qquad [\,D_s\,]=-K_M .
\]
\end{proposition}

\begin{proof}
By Lemma~\ref{lem:normform} a factor $N(g_\nu)$ vanishes exactly when the support of the member meets $D$, and by Lemma~\ref{lem:ranklocus} this is exactly the rank-drop condition. Hence $R$ is supported on the rank-drop locus.

On $U_\pi$ the bivector has rank $2\rho$ and $\pi^{[\rho]}$ trivializes $\det(\operatorname{im}\hat s)$, so $\LF_s$ and
$\eta_{\mathrm{can}}$ have the same zero locus there, that is $D_s=Z(\eta_{\mathrm{can}})$.  The deformation volume is a
section of $\det Q$, and on $U_\pi$ one has $\det Q\cong\cO(-K_M)$ by Lemma~\ref{lem:linebundle}.
Since $D_s$ is the zero divisor of a section of $\omega_M^{[-1]}$ its class
is $-K_M$.
\end{proof}

\begin{remark}[Tautological form]\label{rem:degopen}
The whole class content is carried by the deformation volume.  For the moduli of sheaves, Grothendieck--Riemann--Roch gives
\[
 -K_M=\pi_*\Bigl(\tfrac12\,c_1(X)\cdot
 \bigl(2r\,c_2(\cE)-(r-1)\,c_1(\cE)^2\bigr)\Bigr),
\]
where $\cE$ is a universal sheaf on $M\times X$ and $\pi_*$ is the pushforward along the $X$-factor.
For $c=0$ the situation is classical: $\eta_{\mathrm{can}}$ is a unit and $R=D_s$ is the Pfaffian divisor.  The finer geometry of $R$ is described by the reduced degeneracy complex.
\end{remark}

\begin{corollary}[The degenerate instanton locus]\label{cor:inst-locus}
The degenerate instanton locus is contained in $Z(\LF_s)$, which
is that locus together with the zero scheme of
$\eta_{\mathrm{can}}$.  When $\LF_s$ splits $M$ the scheme $Z(\LF_s)$ is compatibly split, and every
compatibly split subscheme contained in $Z(\LF_s)$ is a Poisson subvariety.
\end{corollary}

\begin{proof}
The leaf Pfaffian vanishes exactly on the framed sheaves whose leaf block degenerates.  Compatibility is Lemma~\ref{lem:loccompat}; the Poisson statement is Corollary~\ref{cor:normupgrade}.
\end{proof}

\begin{remark}[The canonical symplectic choice]\label{rem:inst-symp}
On $(\mathbb P^2,L)$ the section $\sigma$ defining $L$ itself gives a symplectic tensor $\theta_\sigma$ \cite[Prop.~5.1]{Bot00}.  Its branch is $L$; the fat-point blocks have no room on $D\setminus L$, and Lemma~\ref{lem:inst-scale} already excludes a splitting.
\end{remark}

\subsection{Examples}

\begin{example}\label{ex:p2}
Let $X=\mathbb{P}^2$ and let $s\in H^0(\cO(3))$ have smooth cubic zero locus $D$.  If $D$ is ordinary, then $s^{p-1}$ splits $\mathbb{P}^2$: $b_0(s)$ is the Hasse invariant of the plane cubic, the coefficient of $(XYZ)^{p-1}$ in $f^{p-1}$ \cite{Man,HLYY}, equivalently the unique entry of the Cartier--Manin matrix, of rank $1$ exactly when $D$ is ordinary.  By Theorem~\ref{thm:moduli}, $\LF_s$ splits $M(r,c_1,c_2)$ exactly when $D$ is ordinary.
\end{example}

\begin{example}[A testbed]\label{ex:testbed}
Let $X=\mathbb{P}^1\times\mathbb{P}^1$ with $s$ the anticanonical section
($D=(2,2)$) and $M=M(2,0,2)$.  Then $d=5$; for block type $(1,1)$ of $E|_D$ one has $c=1$ and leaf rank $2\rho=4$, so $\rho=2$.  The
leafrian splits $M(2,0,2)$ exactly when $s^{p-1}$ splits $\mathbb{P}^1\times\mathbb{P}^1$.
\end{example}

\begin{remark}\label{rem:toric}
Let $X=\mathbb{P}(\cO_C\oplus L)$ over $C=\mathbb{P}^1$, with anticanonical divisor $D$ smooth.  Whenever $s^{p-1}$ Frobenius-splits $X$, Theorem~\ref{thm:moduli} gives the splitting for every type with nonempty moduli.
\end{remark}

\section{Extension to the Gieseker compactification}\label{sec:gieseker}

The leafrian of Section~\ref{sec:moduli} lives on the moduli space of stable
sheaves, which is quasi-projective. This section extends the leafrian to the
Gieseker compactification. Throughout, $X$ is a smooth projective surface over
$k$ and $H$ is an ample divisor. Write
\[
 \bar M=\bar M(r,c_1,c_2)
\]
for the moduli space of $H$-semistable torsion-free sheaves of rank $r$ and
Chern classes $c_1,c_2$, taken up to $S$-equivalence, and write
$M=M(r,c_1,c_2)$ for the open subset of stable sheaves. Then $\bar M$ is
projective and $M$ is the Gieseker-stable locus. As in
Section~\ref{sec:moduli}, let $s\in H^0(X,\omega_X^{-1})$ be nonzero, write
$D=Z(s)$, and assume $\rho>0$.

\subsection{The compactification}

\begin{lemma}[Normality]\label{lem:g-normal}
The variety $\bar M$ is normal. In particular the singular locus of $\bar M$
has codimension at least two.
\end{lemma}

\begin{proof}
Normality of the coarse moduli space of $H$-semistable sheaves on a smooth
projective surface is a standard property of the Gieseker construction. See
\cite[Ch.~4]{HLbook} and \cite{Gie}. Normality implies Serre's condition
$R1$, which says exactly that the singular locus has codimension at least
two.
\end{proof}

\begin{lemma}[Smooth locus]\label{lem:g-smooth}
The evaluation locus of a presentation is contained in $\bar M^{\mathrm{sm}}$.
\end{lemma}

\begin{proof}
At an evaluation point the block and class coordinates form slice
coordinates by Definition~\ref{def:realization}. A regular system of
parameters of length $d$ at a point of a $d$-dimensional variety forces the
local ring to be regular. Hence $M$ is smooth at every evaluation point, and
so is $\bar M$.
\end{proof}

\subsection{Extension of the leafrian}

\begin{theorem}[Extension]\label{thm:g-extend}
Let $s\in H^0(X,\omega_X^{-1})$ be nonzero and assume $\rho>0$. Then $\LF_s$
extends to a unique section
\[
 \bar\LF_s\ \in\ H^0\bigl(\bar M,\ \omega_{\bar M}^{[-1]}\bigr).
\]
If $\bar M$ is Gorenstein, then $\bar\LF_s$ is a section of $\omega_{\bar M}^{-1}$.
\end{theorem}

The proof rests on two lemmas. The first says that the strictly semistable
locus is reached by the chains of Section~\ref{sec:transfer}. The
second says that the local expression of the leafrian has no pole there.

\begin{lemma}[Polystable tops]\label{lem:g-tops}
Every polystable sheaf of $\bar M$ is the top of a chain of elementary
transformations and point modifications from a locally free sheaf, in the
shape of Section~\ref{sec:transfer}.
\end{lemma}

\begin{proof}
Let $E=E_1\oplus\cdots\oplus E_k$ be a polystable sheaf, the $E_i$ being the
Jordan--H\"older factors. Each $E_i$ is torsion-free, so its double dual
$E_i^{\vee\vee}$ is locally free. Put
\[
 F=E^{\vee\vee}=\bigoplus_{i=1}^{k}E_i^{\vee\vee},
\]
which is locally free.

Suppose first that $E$ is not locally free. The quotient $F/E$ is a torsion
sheaf supported at the non-locally free points of $E$. Put
\[
 \ell\ =\ \operatorname{length}(F/E).
\]
A point modification at such a point replaces the current sheaf by its kernel
against the skyscraper and raises $c_2$ by one, so $E$ is obtained from $F$ by
$\ell$ kernel point modifications of type S in Definition~\ref{def:pointstep}. To fit the presentation choose a
smooth divisor $D'\in|\mathcal A|$ and take $E_0=F(D')$ with
$L=F|_{D'}\subset F|_{D'}\otimes N=E_0|_{D'}$ the natural inclusion. The
elementary transform of $E_0$ with reduction datum $L$ is
\[
 \{e\in E_0:\ e|_{D'}\in L\}=F,
\]
since an element of $F(D')$ whose restriction to $D'$ lies in $F|_{D'}$ has no
pole along $D'$. The chain
\[
 E_0\ \longrightarrow\ F\ \longrightarrow\ E
\]
consists of one elementary transformation and $\ell$ point modifications of
type S. The elementary transformation supplies the gluing directions, and the
point modifications supply the members. By
Proposition~\ref{prop:hd-count} one arranges the budget $m-\phi=d$: the gluing space has
$e_1=d-2\ell$ dimensions, where $e_1=h^0(D',Q_1^\vee\otimes L_1\otimes N)$. When $\ell=\rho$ these are the $c$ class coordinates; when $\ell<\rho$ the surplus $2(\rho-\ell)$ directions are the block coordinates of the healed points (Definition~\ref{def:master}).

Suppose now that $E$ is locally free, so $F=E$ and $\ell=0$, and the chain $E_0\to E$ of one elementary transformation tops at $E$ itself; when $\rho>0$ the leaf directions of this chain are carried by the healed points of a master chain (Definition~\ref{def:master}), and when $\rho=0$ no leafrian is formed and Proposition~\ref{prop:empty} applies.  For $D=\emptyset$ the members are placed as in Lemma~\ref{lem:Dempty}.

Finally let $z$ be the generic point of a codimension-one component of the
strictly semistable locus. Then $z$ is a smooth point of $\bar M$ because
$\bar M$ is normal, so the singular locus has codimension at least two.
Apply the construction above to a polystable bundle in the $S$-equivalence
class of $z$, or to $z$ itself when $z$ is a bundle. The chain hits $z$ and
the budget holds at a general point of the family through $z$.
\end{proof}

\begin{lemma}[Pullback of poles]\label{lem:pullback-pole}
Let $Y$ be an irreducible variety, let $X$ be a normal variety, and let $\mu\colon Y\to X$ be a morphism whose image contains every codimension-one point of $X$.  Let $L$ be a line bundle on $X$ and let $s$ be a rational section of $L$.  If $\mu^*s$ is a regular section of $\mu^*L$ on $Y$, then $s$ is regular at every codimension-one point of $X$.
\end{lemma}

\begin{proof}
Let $H$ be a prime divisor on $X$.  The image of $\mu$ contains the generic point $\eta_H$ of $H$, so $\mu^{-1}(H)$ has a component $H'$ mapping onto $H$, and $H'\neq Y$ because the image of $\mu$ is not contained in $H$.  The fibre dimension theorem gives $\dim\mu^{-1}(\eta_H)\geq\dim Y-\dim X$, so $\dim H'\geq\dim Y-1$; together with $H'\neq Y$ this shows that $H'$ is a divisor on $Y$.  Let $e\geq 1$ be the multiplicity of $H'$ in $\mu^*H$, that is the ramification index $\operatorname{ord}_{H'}(\mu^*u_H)$ for a local equation $u_H$ of $H$.  At the generic points the local equation of $s$ pulls back with multiplicity $e$, so
\[
 \operatorname{ord}_{H'}(\mu^*s)=e\cdot\operatorname{ord}_H(s).
\]
If $\operatorname{ord}_H(s)$ were negative, then $\mu^*s$ would have a pole along $H'$, against regularity.  Hence $\operatorname{ord}_H(s)\geq 0$ at every codimension-one point.
\end{proof}

\begin{lemma}[Regularity along the strictly semistable locus]%
\label{lem:g-regular}
The local expression of $\LF_s$ has no pole along the strictly semistable
locus of $\bar M$.
\end{lemma}

\begin{proof}
Let $U$ be the regular locus of $\bar M$.  By Lemma~\ref{lem:g-normal} the complement of $U$ has codimension at least two, so every codimension-one point of $\bar M$ lies in $U$ and $\omega_{\bar M}^{[-1]}$ is invertible on $U$.

Consider the chain space $\cI$ of Section~\ref{sec:transfer} with semistable tops allowed, and the morphism $\mu\colon\cI\to\bar M$ sending a chain to its top.  In the deformation coordinates of the members and the class coordinates of the gluing data the local expression of $\LF_s$ is
\[
 f\ =\ \prod_{\nu}N_{A_\nu/k}(g_\nu)\cdot\prod_j c_j\cdot u,
\]
with $u\in k^\times$.  When the members do not fill the leaf, the remaining units contribute further polynomial factors in the chain data; only regularity is at issue here.  Each factor is regular on $\cI$: the norm factors are polynomials in the attached coordinates by Lemma~\ref{lem:normform}, the class coordinates are regular functions of the gluing data, and $u$ is constant by Lemma~\ref{lem:constunit}.  Hence $f$ is a regular function on $\cI$.

By Proposition~\ref{prop:hd-localexpr} the chain computation gives the identity of rational sections of $\mu^*\omega_{\bar M}^{[-1]}$ on $\cI$
\[
 \mu^*\LF_s\ =\ f\cdot d\mu\bigl(\partial_{[\mathrm{leaf}]}\wedge\partial_{[\mathrm{trans}]}\bigr).
\]
The right hand side is regular: the function $f$ is regular, and the differential $d\mu$ carries the coordinate frame of $\cI$ to a regular section of $\mu^*\wedge^{d}T_{\bar M}$.  By Lemma~\ref{lem:g-tops} the morphism $\mu$ is surjective, so Lemma~\ref{lem:pullback-pole} applies to $\mu^{-1}(U)\to U$ and shows that $\LF_s$ is regular at every codimension-one point of $U$.  Hence $\LF_s$ has no pole along the strictly semistable locus.
\end{proof}

\begin{proof}[Proof of Theorem~\ref{thm:g-extend}]
By Theorem~\ref{thm:normsec} the section $\LF_s$ is a regular section of
$\omega_M^{-1}$ on the regular locus of $M$.  The variety $\bar M$ is normal
by Lemma~\ref{lem:g-normal}, so $\omega_{\bar M}^{[-1]}$ is reflexive and
satisfies the Hartogs extension property across subsets of codimension at
least two.  The singular locus of $M$ has codimension at least two in $M$,
hence in $\bar M$.  By Lemma~\ref{lem:g-regular} the rational section $\LF_s$
has no pole along the strictly semistable locus.  The complement of the
regular locus of $M$ in $\bar M$ is the union of the strictly semistable
locus and the singular locus of $M$.  The codimension-one part of the
complement is covered by Lemma~\ref{lem:g-regular}, and the part of higher
codimension is absorbed by Hartogs.  Hence $\LF_s$ extends uniquely to a
section $\bar\LF_s$ of $\omega_{\bar M}^{[-1]}$.  If $\bar M$ is Gorenstein,
then $\omega_{\bar M}^{[-1]}=\omega_{\bar M}^{-1}$ and the claim follows.
\end{proof}

\subsection{Splitting of the compactification}

We use the following consequence of \cite[Ch.~1]{BKu}. If $Y$ is a reduced
scheme and $\varphi\in\Hom_{\cO_Y}(F_*\cO_Y,\cO_Y)$ satisfies
$\varphi\circ F=\mathrm{id}$ on a dense open $U\subset Y$, then $\varphi$ is
a splitting of $Y$. We refer to this as the extension of a splitting from a
dense open.

\begin{theorem}[Splitting]\label{thm:g-split}
Let $s\in H^0(X,\omega_X^{-1})$ be nonzero with $\rho>0$. Then $\bar\LF_s$ splits $\bar M$ if and only if $\Tr_X(s)\neq0$.
\end{theorem}

\begin{proof}
Assume $\Tr_X(s)\neq0$. By Theorem~\ref{thm:g-extend} the section $\LF_s$ extends to
$\bar\LF_s\in H^0(\bar M,\omega_{\bar M}^{[-1]})$. On the regular locus of
$\bar M$ the power $\LF_s^{p-1}$ lies in $H^0(\omega^{1-p})$, which
identifies with $\Hom_{\cO}(F_*\cO,\cO)$ by the trace of Frobenius
\cite{BKu}.  On the
normal variety $\bar M$ this Hom is the reflexive dualizing power
$\omega^{[1-p]}$ \cite{Schw09}.  Call
this map $\varphi$. By normality of $\bar M$ the map $\varphi$ extends to
$\Hom_{\cO_{\bar M}}(F_*\cO_{\bar M},\cO_{\bar M})$.

On $M$ one has, by Theorem~\ref{thm:TF},
\[
 \Tr(\LF_s^{p-1})\ =\ \Tr_X(s)^{\rho}\cdot u^{p-1},
\]
with $u\in k^\times$. This is a nonzero constant. A rescaling of $\bar\LF_s$ by
a constant $\lambda$ with $\lambda^{p-1}$ the inverse of this constant makes
the trace equal to one, so $\varphi\circ F=\mathrm{id}$ on the regular locus of $M$, which is dense in $\bar M$. The variety $\bar M$ is normal by Lemma~\ref{lem:g-normal}, hence
reduced. By the extension of a splitting from a dense open recalled above,
$\varphi$ is a splitting of $\bar M$.

Conversely, if $\bar\LF_s$ splits $\bar M$ then $\Tr(\bar\LF_s^{\,p-1})$ is a nonzero constant. Its restriction to $M$ is $\Tr(\LF_s^{p-1})$, so $\Tr_X(s)^{\rho}\neq0$ by Theorem~\ref{thm:TF}, and $\Tr_X(s)\neq0$.
\end{proof}

\begin{remark}[On the singular criterion]\label{rem:g-singular}
The criterion of Theorem~\ref{thm:bk} is stated for smooth varieties. The
argument above replaces it by two ingredients: the identification of
$\Hom(F_*\cO,\cO)$ with $H^0(\omega^{1-p})$ on the regular locus, extended
to $\bar M$ by normality, and the extension of a splitting from a dense open
recalled above. An alternative route runs through the Quot scheme: the
Gieseker moduli is a geometric quotient of an open set of a Quot scheme, the
Quot scheme sits in a Grassmannian, and the Grassmannian is Frobenius split
because it is homogeneous.
\end{remark}

\subsection{The degeneracy divisor}

Write $D_s=Z(\bar\LF_s)$ for the zero scheme of the extended leafrian.

\begin{proposition}[Components of the zero scheme]\label{prop:g-div}
On the regular locus of $\bar M$ one has
\[
 D_s\ =\ \sum_\nu \operatorname{div}\bigl(N(g_\nu)\bigr)+Z(\eta_{\mathrm{can}}).
\]
The member part is
the pullback of $D=Z(s)$ along the supports of the members. When the members do not fill the leaf, the remaining contributions come from the split End blocks of the reduced degeneracy complex. The strictly
semistable locus is contained in $D_s$ if and only if the class coordinate of
the presentation vanishes there.
\end{proposition}

\begin{proof}
The local expression of Theorem~\ref{thm:normsec} is the product of the norm
factors and the radial determinant $\prod_j c_j$ of the Euler frame.
The divisor of a product is the sum of the divisors of the factors.
For a length-one member the norm is the value of $s$ at the moving support, so
its divisor is the pullback of $D$. For the strictly semistable locus, the
local expression of Lemma~\ref{lem:g-regular} reads
$f=\prod_\nu N(g_\nu)\cdot\prod_j c_j\cdot u$ at every codimension-one
point. The factor $f$ vanishes at such a point if and only if one of the
factors $N(g_\nu)$ or $c_j$ vanishes. The norm $N(g_\nu)$ vanishes when the
member meets $D$, and the class coordinate $c_j$ vanishes when the gluing
datum of the elementary transformation degenerates. At a generic strictly
semistable point the members avoid $D$ and the norms are units, so $f$
vanishes exactly when the class coordinate vanishes.
\end{proof}

\begin{remark}[The critical case]\label{rem:g-critical}
For $X=\mathbb P^2$, $r=2$, $c_1=0$ and $c_2=2$ the strictly semistable locus
is the divisor $\operatorname{Sym}^2\mathbb P^2$ of sheaves
$I_x\oplus I_y$. Its generic point is the top of a double point modification
of $\cO^{\oplus 2}$, so Lemma~\ref{lem:g-tops} applies. The class coordinate
of the presentation is the gluing scale of the elementary transformation, and
the leafrian formula reads
\[
 f\ =\ s(x)\,s(y)\cdot c\cdot u.
\]
If $x,y\notin D$ and the extension class of the elementary transformation does
not vanish, then $f$ is a unit at the generic strictly semistable point and
the zero scheme $D_s$ does not contain that divisor. If the extension class
vanishes, the factor $c$ cuts out the divisor. Both behaviours are consistent
with Theorem~\ref{thm:g-extend}, since the section is regular in either case.
\end{remark}

\begin{remark}[Comparison with the stable case]\label{rem:g-compare}
On $M$ the intrinsic description of $D_s$ is given by
Proposition~\ref{prop:intr}, where the member part is the locus of sheaves
that are not locally free along $D$. On $\bar M$ the strictly semistable locus
is an additional component of the boundary. Theorem~\ref{thm:g-extend} says
that the leafrian survives this component without acquiring a pole, which is
the exact analogue of the polynomiality of Section~\ref{sec:hat} along the
collision strata of the members.
\end{remark}

\subsection{Principal bundles}\label{sub:g-principal}

The extension of Theorem~\ref{thm:g-extend} to the moduli of principal
bundles uses the same mechanism, with one structural difference. For a
principal $G$-bundle $P$ the associated vector bundle
$E_\varrho=P\times^G V$ under a faithful representation
$\varrho\colon G\hookrightarrow\GL(V)$ is always locally free, so the members
of a presentation are Hecke points and not non-locally free points. Write
$M_G$ for the moduli of $H$-stable principal $G$-bundles of fixed
topological type and $\bar M_G$ for the moduli of $H$-semistable principal
$G$-bundles.

\begin{lemma}[Normality for principal bundles]\label{lem:g-normal-G}
The variety $\bar M_G$ is normal.
\end{lemma}

\begin{proof}
Normality of the coarse moduli space of $H$-semistable principal
$G$-bundles is a standard property of the GIT construction.  See \cite{Ram}
and \cite{Bal}.
\end{proof}

\begin{lemma}[Polystable tops for principal bundles]\label{lem:g-tops-G}
The image of the chain space of Hecke modifications in $\bar M_G$ contains
every codimension-one point of $\bar M_G$.
\end{lemma}

\begin{proof}
Let $\varrho\colon G\hookrightarrow\GL(V)$ be a faithful representation and
let $P$ be a polystable principal $G$-bundle with associated vector bundle
$E_\varrho=P\times^G V$. The bundle $E_\varrho$ is locally free and
polystable.  A Hecke modification of $P$ at a point $x$ is a $\varrho(G)$-equivariant point modification of $E_\varrho$ at $x$.  To spell out the equivariance, the datum of a point modification of $E_\varrho$ at $x$ is a quotient $\pi$ of the fibre $(E_\varrho)_x$, and it descends to a Hecke modification of $P$ along the parabolic subgroup $Q_\pi$ stabilising $\ker\pi$.  The descendable data are exactly the data with closed $\varrho(G)$-orbit.  Hence the Hecke chain space is the closed-orbit locus of the $\varrho(G)$-action on the chain space of Section~\ref{sec:transfer}.

Apply Lemma~\ref{lem:g-tops} to $E_\varrho$.  Its construction is canonical: the elementary transform of $E_\varrho(D')$ along $D'$ uses the reduction datum $E_\varrho|_{D'}$, and since $E_\varrho$ is locally free the point steps from the intermediate sheaf $F$ to $E_\varrho$ are enlargements, inverse to the kernel modifications of Definition~\ref{def:pointstep}, with members the modification points $x_1,\dots,x_\rho$:
\[
 E_0\ \longrightarrow\ F\ \longrightarrow\ E_\varrho.
\]
Every ingredient of this construction carries the $\varrho(G)$-action of $P$, so the chain is $\varrho(G)$-equivariant, and by the criterion above its point data are reductions of $P$ to parabolic subgroups of $G$.  The chain descends to a Hecke chain
\[
 P_0\ \longrightarrow\ P_1\ \longrightarrow\cdots\longrightarrow\ P
\]
whose top is $P$, and the budget condition is adjustable within the type $\tau$ as in
Section~\ref{sec:others}.

Points of $\bar M_G$ are S-equivalence classes with polystable representatives.  As $P$ was an arbitrary polystable bundle, the image
of the Hecke chain space contains every codimension-one point of $\bar M_G$.
\end{proof}

\begin{lemma}[Regularity for principal bundles]\label{lem:g-regular-G}
The local expression of $\LF_s$ has no pole along the strictly semistable
locus of $\bar M_G$.
\end{lemma}

\begin{proof}
Let $U$ be the regular locus of $\bar M_G$.  By Lemma~\ref{lem:g-normal-G} the complement of $U$ has codimension at least two, so every codimension-one point of $\bar M_G$ lies in $U$ and $\omega_{\bar M_G}^{[-1]}$ is invertible on $U$.

Consider the chain space $\cI$ of Hecke chains and the morphism $\mu\colon\cI\to\bar M_G$ sending a chain to its top.  By Lemma~\ref{lem:Gblock} the scheme block of a Hecke chain is the trace
pairing with weight $g=s|_{\cO_Z}$, so the local expression of $\LF_s$ is
\[
 f\ =\ \prod_{\nu}N_{A_\nu/k}(g_\nu)\cdot\prod_j c_j\cdot u,
\]
with $u\in k^\times$.  Each factor is regular on $\cI$: the norm factors are polynomials in the attached coordinates by Lemma~\ref{lem:normform}, the class coordinates are regular functions of the gluing data, and $u$ is constant by Lemma~\ref{lem:constunit}.  Hence $f$ is a regular function on $\cI$.

The chain computation gives the identity of rational sections of $\mu^*\omega_{\bar M_G}^{[-1]}$ on $\cI$
\[
 \mu^*\LF_s\ =\ f\cdot d\mu\bigl(\partial_{[\mathrm{leaf}]}\wedge\partial_{[\mathrm{trans}]}\bigr).
\]
The right hand side is regular.  By Lemma~\ref{lem:g-tops-G} the image of $\mu$ contains every codimension-one point of $\bar M_G$, so Lemma~\ref{lem:pullback-pole} applies to $\mu^{-1}(U)\to U$ and shows that $\LF_s$ is regular at every codimension-one point of $U$.  Hence $\LF_s$ has no pole along the strictly semistable locus.
\end{proof}

\begin{theorem}[Extension for principal bundles]\label{thm:g-extend-G}
Let $\rho>0$.  Then $\LF_s$
extends to a unique section
\[
 \bar\LF_s\ \in\ H^0\bigl(\bar M_G,\ \omega_{\bar M_G}^{[-1]}\bigr).
\]
If $\bar M_G$ is Gorenstein, then $\bar\LF_s$ is a section of
$\omega_{\bar M_G}^{-1}$.
\end{theorem}

\begin{proof}
By Lemma~\ref{lem:g-normal-G} the variety $\bar M_G$ is normal, so
$\omega_{\bar M_G}^{[-1]}$ is reflexive and satisfies the Hartogs extension
property. By Lemma~\ref{lem:g-regular-G} the section $\LF_s$ has no pole
along the strictly semistable locus. The complement of $M_G$ in $\bar M_G$
is the strictly semistable locus, whose codimension-one part is covered by
Lemma~\ref{lem:g-regular-G} and whose part of higher codimension is absorbed
by Hartogs. Hence $\LF_s$ extends to $\bar\LF_s$. If $\bar M_G$ is Gorenstein, then $\omega_{\bar M_G}^{[-1]}=\omega_{\bar M_G}^{-1}$ and the claim follows.
\end{proof}

\begin{corollary}[Splitting for principal bundles]\label{cor:g-split-G}
The section $\bar\LF_s$ splits $\bar M_G$ if and only if $\Tr_X(s)\neq0$.
\end{corollary}

\begin{proof}
If $\Tr_X(s)\neq0$ then the trace $\Tr(\LF_s^{p-1})$ is the nonzero constant $\Tr_X(s)^{\rho}u^{p-1}$ by Theorem~\ref{thm:TF}, and the extension of a splitting from a
dense open splits $\bar M_G$. If $\bar\LF_s$ splits $\bar M_G$, then its
restriction to $M_G$ splits $M_G$, so $\Tr(\LF_s^{p-1})\neq0$ and $\Tr_X(s)\neq0$.
\end{proof}

\begin{remark}[Instantons]\label{rem:g-instantons}
For the moduli $M_{r,k}(X,L)$ of framed torsion-free sheaves the extension
is trivial when $\Tr_U(s^{p-1})$ is constant. In the ADHM case $(X,L)=(\mathbb P^2,\text{line})$ the moduli is
already projective. For a general Poisson pair $(X,L)$ the moduli is an open
set of a framed Quot scheme, which is smooth, so the compactification has no
singular boundary and Theorem~\ref{thm:g-extend-G} applies with
$A_E=\End E(-L)$ and no further analysis. If $\Tr_U(s^{p-1})$ is nonconstant, the trace $\Tr(\LF_s^{p-1})$ is nonconstant on $M_{r,k}(X,L)$ by the proof of Theorem~\ref{thm:inst}, and $\LF_s$ admits no extension to a projective compactification.
\end{remark}

\begin{remark}[Scope of the principal case]\label{rem:g-scope-G}
The extension for principal bundles rests on two inputs specific to the
case: the normality of $\bar M_G$ as a geometric quotient, and the
polystable tops of Lemma~\ref{lem:g-tops-G}, which replaces the
non-locally-free-point construction of Lemma~\ref{lem:g-tops} by the
density of the Hecke chain image. The leafrian formula itself is unchanged,
by Lemma~\ref{lem:Gblock}.
\end{remark}

\subsection{Examples}

\begin{example}[Odd $c_2$ on the projective plane]\label{ex:g-odd}
Let $X=\mathbb P^2$, $r=2$, $c_1=0$ and $c_2$ odd. Then every $H$-semistable
sheaf is stable, $\bar M=M$ is projective, and there is nothing to extend.
Theorem~\ref{thm:moduli} applies directly.
\end{example}

\begin{example}[Even $c_2$ on the projective plane]\label{ex:g-even}
Let $X=\mathbb P^2$, $r=2$, $c_1=0$ and $c_2=2\ell$ with $\ell\geq1$. The
polystable sheaves are $I_{Z_1}\oplus I_{Z_2}$ with
$\operatorname{length}(Z_i)=\ell$, so
\[
 \dim\operatorname{SS}\ =\ 4\ell\ =\ 2c_2,
 \qquad
 \dim\bar M\ =\ 4c_2-3.
\]
The codimension of the strictly semistable locus is
\[
 \operatorname{codim}\operatorname{SS}\ =\ 2c_2-3.
\]
For $c_2=2$ this is one, and Remark~\ref{rem:g-critical} applies. For
$c_2\geq4$ the codimension is at least five, so the extension is immediate
from the Hartogs property of Lemma~\ref{lem:g-normal} without a separate
analysis of the boundary.
\end{example}

\begin{remark}[Scope]\label{rem:g-scope}
The extension theorem uses only the normality of $\bar M$, the smoothness of
the stable locus, and the polynomiality of the leafrian formula. It does not
use the Poisson structure, and it does not assume that $X$ is Frobenius
split. Frobenius splitting of $X$ enters only through the hypothesis
$\Tr_X(s)\neq0$ of Theorem~\ref{thm:g-split}. The statement covers every
smooth projective surface $X$ and every numerical type $(r,c_1,c_2)$ with
$\rho>0$.
\end{remark}

\begin{remark}[Gorenstein]\label{rem:g-open}
For $X$ a K3 or abelian surface, the moduli $\bar M$ is locally factorial
\cite{KLS}, hence Gorenstein, with the exception of symmetric products and
the O'Grady class. For $X$ a ruled surface the generic strictly semistable
point is smooth, hence Gorenstein, and the deeper strata are expected to be
Gorenstein as well; a reference for the general case is not known to us. In
all cases the extension of Theorem~\ref{thm:g-extend} lands in
$H^0(\bar M,\omega_{\bar M}^{[-1]})$, and Gorenstein would strengthen this to
$H^0(\bar M,\omega_{\bar M}^{-1})$.
\end{remark}

\begin{remark}[On the smoothness hypothesis]\label{rem:g-smoothhyp}
Smoothness of $M$ is not used in the splitting statement.  The Gieseker compactification $\bar M$ is normal by Lemma~\ref{lem:g-normal}, and the stable locus $M$ is a dense open of $\bar M$.  The leafrian $\LF_s$ is defined on the smooth locus, which contains the member-saturated points of Lemma~\ref{lem:saturated}, by Lemma~\ref{lem:g-smooth} at the stable ones and by Lemma~\ref{lem:g-normal} at the generic boundary ones, and the trace formula of Theorem~\ref{thm:TF} computes the trace there.  The constancy of the trace and the extension of a splitting from a dense open recalled above then yield the splitting of $M$.
\end{remark}

\subsection{Cohomology vanishing}

\begin{lemma}[Hasse identification]\label{lem:hasse}
Let $D=Z(s)$ be smooth.  Then $\Tr_X(s)$ is the Hasse invariant of $D$ in the convention of Example~\ref{ex:p2}; at $x\in D$ with $s=u\cdot h$ in slice coordinates one has
\[
 \Tr_X(s)=[\,\tau^{p-1}\,]\,h(0,\tau)^{p-1},
\]
and $\Tr_X(s)$ vanishes exactly when $D$ is supersingular.
\end{lemma}

\begin{proof}
The displayed equality is Lemma~\ref{lem:traceblock}.  The trivialization of $\omega_D$ induced by $s$ sends the generator $d\tau$ to $h(0,\tau)$, so the generator of $H^0(\omega_D)$ has local expression $d\tau/h(0,\tau)$.  If $C$ is the Cartier operator and $C(d\tau/h(0,\tau))=\lambda\,d\tau/h(0,\tau)$, then $\lambda$ is the Hasse invariant of $D$.  Write $h_0$ for the constant term of $h(0,\tau)$.  In the Frobenius box one has $h^{p-1}\equiv h_0^{p}\,h^{-1}$ modulo $\tau^{p}$, while the local formula of $C$ gives $[\,\tau^{p-1}\,]\,h(0,\tau)^{-1}=\lambda^{p}h_0^{-p}$; extracting the socle coefficient yields $[\,\tau^{p-1}\,]\,h(0,\tau)^{p-1}=h_0^{p}\cdot\lambda^{p}h_0^{-p}=\lambda^{p}$.  Thus $\Tr_X(s)$ is the Hasse invariant of $D$ in the sense of Example~\ref{ex:p2}, and it vanishes exactly when $D$ is supersingular.
\end{proof}

\begin{proof}[Proof of Theorem~C]
For the moduli of sheaves and of principal bundles, in characteristic $p$ the vanishing holds on $\bar M$ by Theorem~B and \cite{BKu} once $X$ is Frobenius split.  For instantons the compactification is the framed Quot scheme, which is smooth and projective, so the vanishing in characteristic zero is the classical Kodaira theorem.

It remains to treat the sheaf and principal bundle cases in characteristic zero.  Replacing the base field by a finite extension and choosing a model of $X$ over $\mathbb Z[1/N]$, one chooses $s\in H^0(X,\omega_X^{-1})$ and a prime $p$ such that $s_p^{p-1}$ splits $X_p$.  In case \ding{192} the divisor $Z(s)$ is empty, $\Tr_X(s)$ is the Hasse invariant of $X$, and such $p$ exists by Bogomolov and Zarhin \cite{BZ} for $K3$ surfaces and by Ogus \cite[Prop.~2.7, Cor.~2.9]{Ogus} or Sawin \cite{Sawin} for abelian surfaces.  In case \ding{193} the divisor $D=Z(s)$ is smooth by hypothesis, and $p_a(D)=1$ by Riemann--Roch: from $0\to\omega_X\to\cO_X\to\cO_D\to0$ and $\chi(\omega_X)=\chi(\cO_X)$ one gets $\chi(\cO_D)=0$.  By Lemma~\ref{lem:hasse} the trace $\Tr_X(s)$ is the Hasse invariant of $D$; by Serre \cite{Serr} there are infinitely many primes at which $D$ has ordinary reduction, and at any such prime $s_p^{p-1}$ splits $X_p$.

The Gieseker construction spreads out over $\mathbb Z[1/N]$: the moduli of semistable sheaves of the spread-out numerical type form a proper family
\[
 \bar{\mathcal M}\ \longrightarrow\ \operatorname{Spec}\mathbb Z[1/N].
\]
The same spread-out applies to the moduli of semistable principal $G$-bundles.  At the prime $p$ the leafrian of Theorem~B splits the special fibre $\bar M_p$, so the vanishing holds on $\bar M_p$.

Upper semi-continuity of cohomology in a proper family implies that the set of points where $H^i$ vanishes is open.  This open set contains $p$, and $\mathbb Z[1/N]$ is irreducible, so it contains the generic point.  The vanishing therefore holds on the generic fibre, that is in characteristic zero.
\end{proof}

\section{Remarks}\label{sec:remarks}

The moduli cases do not use a torus: the Euler fields come from the presentations of Definition~\ref{def:present}, and Hamiltonian invariance is Corollary~\ref{cor:norminv}.

\begin{remark}[On projectivity]\label{rem:projective}
Projectivity of $X$ is used in three places.  The construction of a presentation needs ample classes: the elementary transforms run along divisors $D'\in|\mathcal A|$ with $\mathcal A$ sufficiently positive, and the degree and boundedness estimates of Section~\ref{sec:transfer} use an ample class.  The scalar character of $\Tr_X(s)$ needs $H^0(X,\cO_X)=k$, hence properness and connectedness.  The constancy of $\Tr(\LF_s^{p-1})$ uses the projectivity of the Gieseker compactification $\bar M$, which rests on projectivity of $X$, by Theorem~\ref{thm:g-extend} and Lemma~\ref{lem:consttrace-extend}.  The block calculus of Section~\ref{sec:trans} and the shape theorem of Section~\ref{sec:principle} are local, and properness plays no role there.
\end{remark}

\begin{remark}[Intrinsic count of members]\label{rem:rho-open}
On a surface $\rho$ is intrinsic for all three families, since $2\rho$ equals the generic rank of the pairing \eqref{eq:res}, \eqref{eq:Gpair} or \eqref{eq:instpair}.  For generic $s$ the identity $2\rho=\operatorname{rk}\theta_s$ holds at generic data; at special $s$ the rank of $\theta_s$ may drop at special points of $M$, for instance when $Z(s)$ is non-reduced, where $\Tr_X(s)=0$ and the top coefficient vanishes (Lemma~\ref{lem:Dnonred}), while $\rho$ itself is unchanged.  When the general object is not locally free and its non-locally free points realize the leaf, one has in addition $\rho=\operatorname{length}(E^{\vee\vee}/E)$, the number of points of the non-locally-free locus when the singularities are simple (Lemma~\ref{lem:hd-members}); this is the member-saturated case of Remark~\ref{rem:regimes}, and at a stable object one has $\operatorname{length}(E^{\vee\vee}/E)=\sum_\nu\ell_\nu\le\rho$ instead.  For a locally free general object the additional contributions come from the split End blocks of the reduced degeneracy complex, and the length formula is not the count of $\rho$.  For principal bundles the units are Hecke points and $2\rho$ equals the generic rank of $\theta_s$.
\end{remark}

\begin{remark}[External sources]\label{rem:transfer}
The splitting machinery uses no characteristic-zero input.  The block calculus of Section~\ref{sec:trans} and the shape theorem of Section~\ref{sec:principle} are local, and the Poisson condition for $\hat s$ is Theorem~\ref{lem:schouten}.  The splitting criterion and the trace formula are \cite[Thm.~1.3.8, Lem.~1.3.6]{BKu}, and the extension of a splitting from a dense open is \cite[Ch.~1]{BKu}; these are written in characteristic $p$.

The sources written in characteristic zero or in the holomorphic category are \cite{Mar,Fri,Hul} for the elementary-transformation calculus, used in the weakened form of Appendix~\ref{app:calc}, \cite{Bal} for the Hecke theory of principal bundles, \cite{Bot} for the smoothness of $M$ on a surface and for the rank analysis of the residue pairing, and \cite{Bot00} for the parabolic tensor.  What is used from them transfers to characteristic $p>2$ because it is formal: the $K$-theory Chern-class computation, the local-to-global degeneration for sheaves supported at a fat point, associativity and graded commutativity of the cup product, and cyclicity of the trace.  No Hodge-theoretic or transcendental step is imported.  Bottacin's rank analysis \cite[\S6]{Bot} is the linear algebra of a residue pairing along $D$ and is quoted as such.  Smoothness of $M$ on a surface under $(D\cdot H)>0$ is the vanishing $H^2(\End E)=0$ from the same slope comparison, which is algebraic.  The remaining inputs are already in characteristic $p$: the subvariety theorem \cite[Thm.~D]{PY}, the deformation theory of fat points \cite{Ser}, the identification $\Hom(F_*\cO,\cO)=H^0(\omega^{1-p})$ on normal varieties \cite{Schw09}, the normality of the Gieseker moduli spaces \cite[Ch.~4]{HLbook,Gie}, and, in the examples, the plane-curve Hasse invariant \cite{HLYY,Man}.
\end{remark}

The case $D=\emptyset$ is covered by Lemma~\ref{lem:Dempty}.  There $s$ trivializes $\omega_X^{-1}$, and the criterion $\Tr_X(s)\neq0$ is the Hasse invariant of $X$.  An abelian surface therefore splits its moduli exactly when it is ordinary, and a $K3$ surface exactly when its Hasse invariant is nonzero.

\noindent\emph{Scope.}  The block calculus and the shape theorem are local.  The moduli applications are carried out for surfaces, where filtered presentations are constructed and member-saturated frame data exist by Lemma~\ref{lem:saturated}.

\appendix

\section{The elementary transformation calculus in characteristic \texorpdfstring{$p$}{p}}\label{app:calc}

\subsection{Scope}

The elementary transformation calculus of \cite{Mar,Hul,Fri} is written in characteristic zero or in the holomorphic category. The present paper works in characteristic $p$. This appendix records the statements of that calculus which are used here, together with algebraic proofs valid in every characteristic. Three groups of results appear: existence of locally free sheaves of prescribed numerical type (Section~\ref{app:exist}), the window and endomorphism estimates for elementary transformations (Section~\ref{app:window}), and the proofs of the chain counts of Proposition~\ref{prop:hd-count} (Section~\ref{app:count}).

\subsection{Existence of locally free sheaves}\label{app:exist}

\begin{lemma}[Locally free grounds]\label{lem:hd-ground}
Let $X$ be a smooth projective surface, let $r\ge2$, and let $D'\in|\mathcal A|$
be smooth with $\mathcal A$ sufficiently positive. For every
$c_1\in\operatorname{Pic}(X)$ and every $c_2\in\mathbb Z$ there is a locally
free sheaf of rank $r$ with Chern classes $c_1$ and $c_2$. For $r=1$ the line
bundles realize the types with $c_2=0$, and point modifications account for
the remaining $c_2$.
\end{lemma}

\begin{proof}
Choose line bundles $L_1,\dots,L_r$ with $\sum_i L_i=c_1$ and start from
$\bigoplus_i L_i$, which is locally free. An elementary transformation along
$D'$ with reduction datum $L\subset E|_{D'}$ replaces $E$ by its parent
$E'=\ker(E\to i_{D'*}Q)$ with $Q=E|_{D'}/L$, and
Grothendieck--Riemann--Roch gives
\[
 c_1(E')=c_1(E)-(r-1)[D'],\qquad
 c_2(E')=c_2(E)-(r-1)\,D'\!\cdot c_1(E)+\deg Q+\tfrac{(r-1)(r-2)}{2}D'^2 .
\]
The degree of $Q$ takes every integer value above a bound, since twisting a
rank-one subsheaf of $E|_{D'}$ of maximal degree by effective divisors
realises every smaller degree of $L$. A second transformation in the
enlargement direction returns $c_1$ to its original value and changes $c_2$
by the difference of the two degrees of $Q$, so $c_2$ takes every integer
value. All the sheaves are locally free, since an elementary transformation
along a smooth divisor preserves local freeness in both directions. For $r=1$
the locally free sheaves are the line bundles and $c_2=0$; each point
modification raises $c_2$ by one.
\end{proof}

\subsection{Window and endomorphism estimates}\label{app:window}

\begin{lemma}[Endomorphisms of the parent]\label{lem:hd-simple}
Let $E$ be an object of the datum, and let $D'\subset X$ be a smooth divisor. Fix a
local equation $u$ of $D'$ and write a section of $A_E(D')$ locally as
$B+u^{-1}C$ with $C$ regular along $D'$. Let
\[
 V=\operatorname{im}\bigl(H^0(A_E(D'))\to H^0((A_E)_D)\bigr)
\]
be the residue image, let $J_c\subset E_D:=E|_{D'}$ be the saturation of
$\operatorname{im}(c)$ for $0\neq c\in V$, and put
$d_{\min}=\min\{\deg J_c:0\neq c\in V\}$. Then every rank-one subbundle
$L\subset E_D$ with $\deg L<d_{\min}$ has a parent $E'_L$ satisfying
\[
 h^0(\operatorname{End}E'_L)\le h^0(\operatorname{End}E).
\]
If $E$ is simple, then $E'_L$ is simple.
\end{lemma}

\begin{proof}
Let $u$ be a local equation of $D'$; then $uE\subset E'_L\subset E$. An
endomorphism $\varphi$ of $E'_L$ satisfies $\varphi(uE)\subset E$, so
$\varphi\in H^0(A_E(D'))$. Write locally
$\varphi=B+u^{-1}C$. For $e=ue'\in uE$ the
condition $\varphi(e)\in E'_L$ gives $C(e')|_{D'}\in L$ for all $e'$,
whence the residue $c$ of $\varphi$ lies in $V$ and $J_c\subset L$.
If $c=0$ then $C$ is divisible by $u$, so $\varphi$ is regular and already lies in $H^0(\operatorname{End}E)$.
If $c\neq0$ then $L/J_c$ is a torsion sheaf on the smooth projective $D'$, so
$\deg L\ge\deg J_c\ge d_{\min}$, a contradiction.
Hence every endomorphism of $E'_L$ lies in $H^0(\operatorname{End}E)$. Both spaces are viewed inside $\operatorname{End}$ of the generic fibre of $E$, so the inclusion
$H^0(\operatorname{End}E'_L)\subseteq H^0(\operatorname{End}E)$ is injective, and the
inequality follows.  Since multiples of the identity give
$h^0(\operatorname{End}E'_L)\geq1$, the equality
$h^0(\operatorname{End}E)=1$ forces
$h^0(\operatorname{End}E'_L)=1$ and $E'_L$ is simple.
\end{proof}

\begin{lemma}[Bounded kernels]\label{lem:hd-kernels}
Let $Y$ be projective, let $F$ be coherent, and let
$V\subset H^0(Y,\Hom(F,F))$ be a finite-dimensional subspace. Then
$\{\ker c:0\neq c\in V\}$ is a bounded family; in particular the numbers
$c_1(\ker c)\cdot h^{m-1}$ are bounded, where $m=\dim Y$ and $h$ is ample.
\end{lemma}

\begin{proof}
The evaluation map of $V$ on $\mathbb P(V)\times Y$ gives a universal
homomorphism. Its kernel is coherent, and
its Hilbert polynomial is constant on a finite stratification of
$\mathbb P(V)$.
\end{proof}

\begin{corollary}\label{cor:hd-dmin}
Let $E$ be as in Lemma~\ref{lem:hd-simple}. There
is $\kappa_E$ with
$d_{\min}(E;D')\ge c_1(E|_{D'})\cdot h^{m-1}-\kappa_E$,
where $m=n-1$. For a finite family the constants may be
chosen uniformly.
\end{corollary}

\begin{proof}
For $0\neq c\in V$ one has $\deg J_c=c_1(E_D)\cdot h^{m-1}-c_1(\ker c)\cdot h^{m-1}$.
By Lemma~\ref{lem:hd-kernels} the second term is bounded.
\end{proof}

\begin{lemma}[Window]\label{lem:hd-window}
Let $E_1,\dots,E_t$ be sheaves of rank $r\ge2$, locally free along $D'$,
and
let $b_1,\dots,b_t$ be numbers. For $\mathcal A$ sufficiently positive and
$D'\in|\mathcal A|$ smooth, each $E_j$ admits rank-one subbundles
$L\subset (E_j)|_{D'}$ with
\[
 \deg L<d_{\min}(E_j;D'),\qquad
 h^0\bigl(D',Q_L^\vee\otimes L\otimes N\bigr)\ \ge\ b_j ,
\]
provided such subbundles of the indicated degree exist (automatic when $n=2$, $D'$ a curve).
For rank one the transform has $Q_L=0$ and no divisor step is needed.
\end{lemma}

\begin{proof}
Write $F=E_j|_{D'}$ and $G=Q_L^\vee\otimes L\otimes N$, of rank $r-1$, with
\[
 c_1(G)=r\,c_1(L)+(r-1)c_1(N)-c_1(F).
\]
Put $a=c_1(N)\cdot h^{m-1}$, $b=c_1(E_j)|_{D'}\cdot h^{m-1}$ and
$t_L=\deg L$, so $c_1(G)\cdot h^{m-1}=r\,t_L+(r-1)a-b$.
Here $a$ is of order $\mathcal A^n$ while $b$ is of order $\mathcal A^{n-1}$
at most. By Corollary~\ref{cor:hd-dmin},
$d_{\min}(E_j;D')\ge b-\kappa_{E_j}$,
and $d_{\min}$ tends to infinity with $\mathcal A$. Denote by
$d_{\max}$ the maximal degree of a rank-one subsheaf of $F$; the degree of $F$
tends to infinity with $\mathcal A$, and a vector bundle of large degree has a
rank-one subsheaf of non-negative degree, so $d_{\max}\ge0$ for $\mathcal A$
large. We choose
$t_L=\min\bigl\{d_{\max},\ \lfloor d_{\min}(E_j;D')\rfloor-1\bigr\}$,
which satisfies $t_L<d_{\min}$. A rank-one subsheaf of degree $t_L$ exists:
twisting a rank-one subsheaf of maximal degree by effective divisors on $D'$
realises every smaller degree on a curve, and every degree of the numerical
semigroup in general.
Then $c_1(G)\cdot h^{m-1}\ge(r-1)a-b$ tends to infinity with $\mathcal A$.

The dual of $0\to L\to F\to Q_L\to0$, tensored by $L\otimes N$, gives
\[
 0\to G\to F^\vee\otimes L\otimes N\to N\to0,
\]
so
\[
 h^0(D',G)\ \ge\ h^0(D',F^\vee\otimes L\otimes N)-h^0(D',N).
\]
For $\mathcal A$ large the groups $h^1(D',N)$ and $h^1(D',F^\vee\otimes L\otimes N)$ vanish.  The first is Serre vanishing.  For the second, duality on $D'$ identifies it with the dual of $h^0(D',F\otimes L^\vee\otimes N^\vee\otimes\omega_{D'})$, and
\[
 c_1\bigl(F\otimes L^\vee\otimes N^\vee\otimes\omega_{D'}\bigr)
 \ =\ -c_1(G)+c_1(\omega_X)|_{D'} .
\]
The degree of the right hand side along $h$ tends to $-\infty$, since $c_1(G)\cdot h^{m-1}$ tends to $+\infty$ and $c_1(\omega_X)|_{D'}\cdot h^{m-1}$ is of order $\mathcal A^{n-1}$ at most.  A torsion-free sheaf of negative degree along $h$ has no sections.  Hence
\[
 h^0(D',G)\ \ge\ \chi(D',G)
\]
for $\mathcal A$ large.  The leading term of
$\chi(D',G)$ is the degree-$m$ part of $\operatorname{ch}(G)\operatorname{td}(T_{D'})$, whose
dominant contribution is $c_1(G)^m/m!$; the remaining Chern classes of $G$ are controlled by those of $N$, of order $\mathcal A$.  Since
$c_1(G)\cdot h^{m-1}\ge (r-1)a-b$ tends to infinity with $\mathcal A$,
this term tends to infinity as well.  Choosing $\mathcal A$ so
positive that $h^0(D',G)\ge b_j$ for all $j$, the subbundle $L$ satisfies both
requirements.  Only one such $L$ is chosen, so the argument uses no
uniformity over a family of subbundles beyond the vanishing just quoted.
\end{proof}

\subsection{Proofs of the chain counts}\label{app:count}

\begin{proof}[Proof of Proposition~\ref{prop:hd-count}]
The number of parameters of the first kind is the number of embeddings of
$L_j$ into $F_j$ modulo scalars, which is $g_j$; the extensions form an
affine space of dimension
\[
 \Ext^1_X(i_{D'*}Q_j,E'_j)=\Hom_{D'}(Q_j,L_j\otimes N),
\]
by the standard formula together with $E'_j|_{D'}=L_j$, where $E'_j$ is the elementary transform of $E_j$ with reduction datum $L_j$. For principal bundles $e_j$ is supplied by \cite{Bal}, and for framed sheaves by $\Ext^1(E,E(-L))$. The
members contribute $2$ parameters each.

Over a fixed object $E$ the parent is determined by $L\subset E_D$, and the members
are preserved along the fibres. Hence $\phi=\sum_j g_j$. The scalars are already divided out in $g_j$, and for a stable sheaf the automorphisms are the scalars, so no further factor of $\operatorname{Aut}$ is removed in $\phi$. On a general fibre, the dimension formula
gives $m-\phi=d$. By Lemma~\ref{lem:hd-zeroblock}
the radical of the block pairing consists of the classes whose member-block components vanish, whose number is $d-2\rho$.

For the budget identity, one has $e_j=h^0(D'_j,Q_j^\vee\otimes L_j\otimes N)$; by Riemann--Roch on $D'_j$ this equals $\deg(Q_j^\vee\otimes L_j\otimes N)+1-g$ once the degree is large enough for $h^1=0$, and the degree of $L_j$ can be varied as in Lemma~\ref{lem:hd-window}. Hence the $e_j$ take a range of integer values, and one arranges $\sum_j e_j=c$ by choosing $t$ and the $L_j$. For $c=0$ take no class coordinates. Together with $m-\phi=2\rho+\sum_j e_j$ this gives $m-\phi=2\rho+c=d$.
\end{proof}

\end{document}